%% file: blockExtractionUnitaryBrownianMotionv2.tex
\pdfoutput=1

\documentclass[11pt,twoside]{amsart}
\usepackage{mathbbol}
\usepackage[usenames,dvipsnames]{xcolor}
\usepackage[colorlinks=true,linkcolor=NavyBlue,urlcolor=RoyalBlue,citecolor=PineGreen,
hypertexnames=false]{hyperref}
\makeatletter
\renewcommand{\tocsection}[3]{%
  \indentlabel{\@ifnotempty{#2}{\bfseries\ignorespaces#1 #2.\,\,}}\bfseries#3}
\renewcommand{\tocsubsection}[3]{%
  \indentlabel{\@ifnotempty{#2}{\ignorespaces#1 #2\quad}}#3}

\newcommand\@dotsep{4.5}
\def\@tocline#1#2#3#4#5#6#7{\relax
  \ifnum #1>\c@tocdepth 
  \else
    \par \addpenalty\@secpenalty\addvspace{#2}%
    \begingroup \hyphenpenalty\@M
    \@ifempty{#4}{%
      \@tempdima\csname r@tocindent\number#1\endcsname\relax
    }{%
      \@tempdima#4\relax
    }%
    \parindent\z@ \leftskip#3\relax \advance\leftskip\@tempdima\relax
    \rightskip\@pnumwidth plus1em \parfillskip-\@pnumwidth
    #5\leavevmode\hskip-\@tempdima{#6}\nobreak
    \leaders\hbox{$\m@th\mkern \@dotsep mu\hbox{.}\mkern \@dotsep mu$}\hfill
    \nobreak
    \hbox to\@pnumwidth{\@tocpagenum{\ifnum#1=1\bfseries\fi#7}}\par
    \nobreak
    \endgroup
  \fi}
\AtBeginDocument{%
\expandafter\renewcommand\csname r@tocindent0\endcsname{0pt}
}
\def\l@subsection{\@tocline{2}{0pt}{2.5pc}{5pc}{}}
\makeatother
\usepackage{tikz}
\usepackage{tikz-cd}
\usepackage{subfigure}
\usepackage{verbatim}
\usepackage[margin=0.7in]{geometry}

\usepackage[T1]{fontenc}
\usepackage{lmodern}

\usepackage{amsthm}
\usepackage{amsmath}
\usepackage{amssymb}
\usepackage{amsfonts}
\usepackage{enumerate}
\usepackage[english]{babel}
\usepackage{mathtools}
\usepackage{stmaryrd}
\DeclareFontFamily{U}{MnSymbolC}{}
\DeclareFontShape{U}{MnSymbolC}{m}{n}{
  <-5.5> MnSymbolC5
  <5.5-6.5> MnSymbolC6
  <6.5-7.5> MnSymbolC7
  <7.5-8.5> MnSymbolC8
  <8.5-9.5> MnSymbolC9
  <9.5-11.5> MnSymbolC10
  <11.5-> MnSymbolCb12
}{}

\newcommand{\lmss}[1]{\textrm{\normalfont{{\fontfamily{lmss}\selectfont #1}}}}
\input{environnement}

\input{notation}
\address{Sorbonne Universit\'e, Sorbonne Paris Cit\'e, CNRS, Laboratoire de Probabilit\'es Statistique et Mod\'elisation, LPSM, F-75005 Paris, France}
\email{nicolas.gilliers@sorbonne-universite.fr}
\title{Matricial approximations of higher dimensional master fields}
\author{Nicolas Gilliers}
\begin{document}
\begin{abstract}
We study matricial approximations of master fields constructed in \cite{nico1}. These approximations (in non-commutative distribution) are obtained by extracting blocks of a Brownian unitary diffusion (with entries in $\mathbb{R}, \mathbb{C}$ or $\mathbb{K}$) and letting the dimension of these blocks to tend to infinity. We divide our study into two parts: in the first one, we extract square blocks while in the second one we allow rectangular blocks. In both cases, free probability theory and operator-valued free probability appear as the natural framework in which the limiting distributions are most accurately described.
\end{abstract}
\maketitle

\tableofcontents
\newpage
\section{Introduction}
In this work, we study convergence in non-commutative distribution of random matrices extracted from a unitary Brownian motion in high dimensions. We consider three cases: Brownian motions with real, complex and quaternionic entries and denote by $\mathbb{U}(\mathbb{K},N)$ the group of unitary matrices with entries in the division algebra $\mathbb{K}$. Brownian motion on $\mathbb{U}(N,\mathbb{K})$, as a non-commutative process, is studied for quite a long time. We make a short and non-exhaustive list of available results.
\par The story begins with the work of Wigner on Hermitian random matrices having, up to symmetries, independent and identically distributed entries. Under mild assumptions (satisfied for Gaussian distributed entries with a variance that scales as the square root inverse of the dimension), Wigner proved the convergence of this random matrix's moments (the non-commutative distribution) as the dimension tends to infinity. Later, the question of the convergence  in high dimensions of not only one random matrix, but of a process of Hermitian random matrices, the Hermitian Brownian motion  was addressed. The result of Wigner implies convergence of the one-dimensional marginals of this process. The convergence of multi-dimensional marginals is most easily expressed and understood using non-commutative probability theory and a new notion of independence between random variables, which is Voiculescu's freeness.
The notion of freeness (within the framework of operators algebras) was introduced by Voiculescu, but aroused probabilists' interest with the work of Speicher. Freeness is a non-commutative counterpart of classical independence between two random variables; it is a property of two non-commutative random variables that allows computation of the joint distribution from the individual distribution of the random variables. One significant result appearing in the work of Voiculescu (see \cite{voiculescu1991limit}) is the asymptotic freeness of two (classically) independent random matrices. This theorem implies asymptotic freeness of a Hermitian Brownian motion's increments, which leads in turn to the convergence of multidimensional marginals of this process. The limiting process was named semi-circular Brownian motion.
In the 1980's, Biane got interested in the stochastic exponential of a Hermitian Brownian motion, the unitary Brownian motion. He proved a similar result for this integrated version of Hermitian Brownian motion, namely the asymptotic freeness of the increments and convergence of the one-dimensional marginals. The limiting non-commutative process was named free unitary Brownian motion and is a solution of a free stochastic differential equation.
\par In the sequel, we will see that the processes mentioned above can be considered as being one dimensional, meaning that they can be seen to be members of a family of processes indexed by integers that can be approximated, in distribution, by random matrices. These processes are called higher dimensional version of the free unitary Brownian motion. Let us explain, with more details, this point. An approach for the construction of these higher dimensional free unitary Brownian motions is to extract square matrices and let the dimension tend to infinity while maintaining the number of such extractions constant. In other words, the ratio between the dimension of a block and the total dimension is kept constant. Our first result can be informally stated as follows.
\begin{theorem}
As the dimension tends to infinity, the non-commutative distribution of square blocks extracted from a unitary Brownian motion converges to the distribution of a free process.
\end{theorem}
See Theorem \ref{maintheoremsquare} for a more precise statement.
The second part deals with a generalisation the readers may have already guessed. Why settle for square extractions? The answer may be that the framework of free non-commutative probability is not the good one to study asymptotics of products of rectangular blocks. Going from the square case to the rectangular case corresponds to an algebraic move, from the category of stellar algebras to the category of bi-module stellar algebras. We will not develop this point at the moment; we only underline that for rectangular extractions processes, the right framework for studying asymptotic as the dimension of the blocks tend to infinity is amalgamated free probability theory or rectangular free probability theory.
\begin{theorem}
Under the assumption that the ratios between the dimensions of the extracted blocks and the total dimension of the matrix tend to positive real numbers, we prove that normalised traces of product of rectangular extractions from an unitary Brownian motion at a fixed time converge. In addition, the time parametrized family of distributions is an amalgamated free semi-group.
\end{theorem}
 We have been vague on a point: which products between these rectangular blocks are allowed? Do we consider all products that have a meaning, regarding the dimensions of the blocks? At this point, our work splits in two and the last theorem holds for the two possibilities.

\par We mention that the question related to the convergence of square extractions of a Brownian motion has already been addressed by Michael Ulrich in \cite{ulrich2015construction}. The present work extends this initial investigation in two directions. The first concerns the division algebras the matricial coefficients belong to, considering the three cases: complex, real and quaternionic, while the author in \cite{ulrich2015construction} focuses only on the complex case. The second direction of generalisation we explored concerns convergence of rectangular extractions.

\par We end this introduction with an indication on the method we used. It a variation of the one developed by Levy in \cite{levy} to study Browian motion on the orthogonal, unitary, and symplectic compact group. We start from the algebra of Brauer diagrams and add colors to the vertices of a diagram. We obtain a coloured Brauer algebra that proves helpful for the first part of our work concerning square extractions. This coloured Brauer algebra is however too small for the investigation we conduct in the second part: we need a central extension of this algebra.

\subsection{Outline}

In the first part of this work (Section \ref{unitary_matrices}), we make a brief reminder on unitary groups $\mathbb{U}(N,\mathbb{K})$ with $\mathbb{K}$ a division algebra equal either to the field of real numbers, complex numbers or to algebra of quaternions. We then introduce Brownian diffusions on such groups. A multidimensional counterpart of the free Brownian motion is introduced
in Section \ref{freemultbrown}. Our main combinatoric tool, the algebra of coloured Brauer diagram is introduced in Section \ref{schur_weyl}. Convergence in non-commutative distribution of square blocks extractions of a unitary, symmetric and symplectic Brownian motion is studied in Section \ref{squareextractions}. The case were rectangular blocks are allowed for extractions is exposed in Section \ref{rectangularextractions}.
\section{Brownian diffusions on unitary matrices}
\label{unitary_matrices}
\subsection{Unitary matrices over the three finite dimensional division algebras}
We let $\mathbb{K}$ be one of the three associative algebras $\mathbb{C},\mathbb{H}$ and $\mathbb{R}$. We denote by $\lmss{i},\lmss{j}$ and $\lmss{k}$ the linear real basis of $\mathbb{H}$:
$$
\lmss{i}^{2} = \lmss{j}^{2} = \lmss{k}^{2} = -1, \quad \lmss{ij}=\lmss{k}, \lmss{jk}=\lmss{i}, \lmss{ki} = \lmss{j}.
$$
The adjoint of element $x \in \mathbb{K}$ is denoted $x^{\star}$ and the adjoint of a matrix $M = \left(M_{ij}\right)_{1 \leq i,j\leq N}\in \mathcal{M}_{N}(\mathbb{K}),~N \geq 1$ is $M^{\star} = \left(M^{\star}_{ij}\right)_{ 1 \leq i,j \leq N} = \left( M^{\star}_{ji}\right)_{1 \leq i,j \leq N}$.
The group of unitary matrices with entries in $\mathbb{K}$ is the connected subgroup of $\mathcal{M}_{N}(\mathbb{K})$, which depends on an integer $N \geq 1$ and defined by
\begin{equation*}
	\mathbb{U}(N,\mathbb{K}) = \{M \in \mathcal{M}_{N}\left(\mathbb{K}\right),~MM^{\star} = M^{\star}M = 1 \}^{0}.
\end{equation*}
where the exponent ${ }^{0}$ means that we take the connected component of the identity (it is needed for the real case). If $K = \mathbb{R}$ the group $\mathbb{U}(N,\mathbb{R})$ is the group of special orthogonal matrices $SO(N,\mathbb{R})$ and for $K = \mathbb{C}$, $\mathbb{U}(N,\mathbb{C})$ is the group of unitary matrices. The Lie algebra $\mathfrak{u}(N,\mathbb{K})$ is given by
\begin{equation*}
	\mathfrak{u}(N,\mathbb{K}) = \{H \in \mathcal{M}_{N}(\mathbb{K}) : H^{\star} + H = 0 \}.
\end{equation*}
The real Lie algebra  of skew-symmetric matrices of size $N\times N$ is denoted $a_{N}$ and the vector space of symmetric matrices of size $N\times N$ is denoted $s_{N}$. As real Lie algebras, one has  the decompositions:
\begin{equation}
	\label{liealgebras}
	\mathfrak{so}_{N} = \mathfrak{a}_{N}, \quad \mathfrak{u}_{N} = \mathfrak{a}_{N} + \lmss{i}\mathfrak{s}_{N}, \quad \mathfrak{sp}_{N} = \mathfrak{a}_{N} + \lmss{i}\mathfrak{s}_{N} + \lmss{j} \mathfrak{s}_{N} + \lmss{k}\mathfrak{s}_{N},~ N \geq 1.
\end{equation}
It follows that, with $\beta = \textrm{dim}_{\mathbb{R}}(\mathbb{K})$,
$$
\textrm{dim}(\mathfrak{u}_{N}(\mathbb{K})) = \frac{N(N-1)}{2} + (\beta - 1) \frac{N(N + 1)}{2}, N \geq 1.
$$
Amongst the groups $\mathbb{U}(N,\mathbb{K}), ~\mathbb{K} = \mathbb{R},\mathbb{C}$ or $\mathbb{H}$, only $\mathbb{U}(N,\mathbb{C})$ has a non trivial center and is thus non simple. We shall add to the list the group of special unitary matrices $SU(N)$ defined as the subgroup of unitary matrices with complex entries with trace equal to one. The Lie algebra $\mathfrak{s}\mathfrak{u}(N,\mathbb{C})$ is the subalgebra of $\mathcal{M}_{N}\left(\mathbb{C} \right)$ of anti-Hermitian matrices with null trace.
Let $N \geq 1$. To define a Brownian motion on the group $\mathbb{U}(N,\mathbb{K})$ one needs to pick first a scalar product on the Lie algebra $\mathfrak{u}\left(N,\mathbb{K} \right)$. Since $U\left(N,\mathbb{K}\right)$ is simple, there exists up to a multiplication by a positive scalar only one scalar product on $\mathfrak{u}\left(N,\mathbb{K}\right)$ which is invariant by the adjoint action of $\mathbb{U}(N,\mathbb{K})$ on its Lie algebra
\par As the group $\mathbb{U}(N,\mathbb{K})$ is compact, the negative of the Killing form is an invariant scalar product. Since we are going to let the dimension $N$ tend to $+\infty$, we care about the normalization of the Killing form. Let $\langle \cdot, \cdot \rangle_{N}$ be the scalar product
\begin{equation*}
	\langle X,Y \rangle_{N} = \frac{\beta N }{2}\mathcal{R}e(\lmss{Tr}(X^{\star}Y)),~ X,Y \in \mathfrak{u}\left(N,\mathbb{K}\right)
\end{equation*}
The direct sums in the equations \eqref{liealgebras} are decompositions into mutually orthogonal summands for $\langle \cdot, \cdot \rangle_{N}$. Let $\{H^{N}_{k}\}$ be an orthonormal basis of $\mathfrak{u}\left(N,\mathbb{K}\right)$, the Casimir element $C_{\mathfrak{u}_{N}(\mathbb{K})}$ is a bivector in $ \mathfrak{u}_{N}(\mathbb{K}) \otimes_{\mathbb{R}} \mathfrak{u}_{N}(\mathbb{K})$ defined by the formula:
\begin{equation*}
	C_{\mathfrak{u}_{N}(\mathbb{K})} = \sum_{k = 1}^{\beta} H_{k} \otimes H_{k}.
\end{equation*}
We can cast the last formula for the Casimir element into a more concrete form by setting first
\begin{equation*}
	\lmss{P} = \sum_{ab} E_{ab} \otimes E_{ab} \in \mathcal{M}_{N}(\mathbb{K}) \otimes  \mathcal{M}_{N}(\mathbb{K}) \quad \lmss{T} = \sum_{a,b} E_{ab} \otimes E_{ba} \in \mathcal{M}_{N}(\mathbb{K}) \otimes  \mathcal{M}_{N}(\mathbb{K}),
\end{equation*}
then a simple calculation we shall not reproduce here for brevity shows that:
\begin{equation*}
C_{\mathfrak{a}_{N}} = -\lmss{T} + \lmss{P} \quad C_{\mathfrak{s}_{N}} = \lmss{T} + \lmss{P}.
\end{equation*}
The letters $\lmss{T}$ and $\lmss{P}$ stand for \emph{transposition} and \emph{projection}. Eventually, put $\lmss{I}(\mathbb{K}) = \{1,\lmss{i}, \lmss{j}, \lmss{k}\} \cap \mathbb{K}$ and for the needs of the quaternionic case, define
\begin{equation*}
\lmss{Re}^{\mathbb{K}} = \sum_{\gamma \in \lmss{I}(\mathbb{K})} \gamma \otimes \gamma^{-1} \in \mathbb{K} \otimes_{\mathbb{R}} \mathbb{K},~ \lmss{Co}^{\mathbb{K}} = \sum_{\gamma \in \lmss{I}(\mathbb{K})} \gamma \otimes \gamma \in \mathbb{K} \otimes_{\mathbb{R}} \mathbb{K}.
\end{equation*}
For the complex case, formulae for $\lmss{Re}^{\mathbb{C}}$ and $\lmss{Co}^{\mathbb{C}}$ are given below if these quantities are seen in the tensor product over the complex field of $\mathcal{M}_{N}(\mathbb{C})$ with itself, not over the real field as stated in the last equation.
\begin{lemma}
  The Casimir element of the real Lie algebra $\mathfrak{u}_{N}(\mathbb{K})$ is given by
  \begin{equation*}
	C_{\mathfrak{u}_{N}(\mathbb{K})} = \frac{1}{\beta N} \left( -\lmss{T} \otimes_{\mathbb{R}}\lmss{Re}^{\mathbb{K}} + \lmss{P} \otimes_{\mathbb{R}} \lmss{Co}^{\mathbb{K}} \right)
  \end{equation*}
\end{lemma}
\noindent
We agree with the author in \cite{levy}, for the complex case it is more natural to take tensor products over the complex field, not over the real field. In the sequel, the symbol $\otimes$ stands for the symbol $\otimes_{\mathbb{C}}$ if taking the tensor product of complex vector spaces and $\otimes_{\mathbb{R}}$ otherwise. With this convention, we can give simple formulae for $C_{\mathfrak{u}(N,\mathbb{C})}$ and $C_{\mathfrak{u}(N,\mathbb{R})}$:
\begin{equation*}
  C_{\mathfrak{u}_{N}} = -\frac{1}{N} T \in \mathcal{M}_{N}(\mathbb{C}) \otimes \mathcal{M}_{N} (\mathbb{C}),
 \quad C_{\mathfrak{so}_{N}} = -\frac{1}{N}\left( T-P \right) \in \mathcal{M}_{N}(\mathbb{R}) \otimes \mathcal{M}_{N} (\mathbb{R}).
\end{equation*}
Let $m : \mathcal{M}_{N}(\mathbb{K}) \otimes \mathcal{M}_{N}(\mathbb{K}) \rightarrow \mathcal{M}_{N}(\mathbb{K})$ be the multiplication map and let $c_{u_{N}(\mathbb{K})} = m(C_{u_{N}(\mathbb{K})})$.
\begin{lemma}
  If $\mathfrak{g}$ is one of the three Lie algebras at hand, then:
  \begin{equation*}
    c_{u_{N}(\mathbb{K})} = -1 + \frac{2-\beta}{\beta N}I_{N}.
  \end{equation*}
\end{lemma}
\subsection{Brownian motion on unitary groups.}
Let $N \geq 1$ be an integer and $\mathbb{K}$ one of the three division algebra $\mathbb{R}, \mathbb{C}$ and $\mathbb{H}$. Let $(B_{k})_{k \leq \textrm{dim}(\mathfrak{u}_{N}(\mathbb{K}))}$ be a $\lmss{dim}_{\mathbb{R}}(\mathfrak{u}_{N}(\mathbb{K}))$ dimensional Brownian motion and let $\left(H_{k}^{N}\right)_{1 \leq k \leq N}$ an orthonormal basis for $\mathfrak{u}(N,\mathbb{K})$, a Brownian motion $K = \left(K(t)\right)_{t\geq 0}$ with values in the Lie algebra $\mathfrak{u}(N,\mathbb{K})$ is
\begin{equation}
  \label{eq:brownianmotion_alg}
  K(t) = \sum_{k=1}^{\lmss{dim}_{\mathbb{R}}(\mathfrak{u}(N,\mathbb{K}))} B_{k}(t)H^{N}_{k},~ t \geq 0.
\end{equation}
If $d \leq 1$ and $n \geq 1$ are two integers such that $N = nd$, a matrix $A \in \mathcal{M}_{N}(\mathbb{K})$ is seen as an element of $\mathcal{M}_{N}(\mathbb{K}) \otimes_{\mathbb{R}}(\mathbb{K})$ through the identification $M \mapsto E^{i}_{j}\otimes M(i,j)$ if $M(i,j)$ is the matrix of size $d \times d$ in place $i,j$ in the matrix $M$, $1\leq i,j \leq N$.
The Brownian motion $\munitaryfd = (\munitaryfd(t))_{t \geq 0} $ on the unitary group $\mathbb{U}(N,\mathbb{K})$ is the solution of the following stochastic differential equation with values in the tracial algebra $\left( \mathcal{M}_{N}(\mathbb{K}), \lmss{tr} \right)$:
\begin{equation}
  \label{eqn:brownianmotion_gr}
\left\{
\begin{array}{l}
\textrm{d}\munitaryfd(t) =  \munitaryfd(t)\textrm{d}K(t) + \frac{c_{\mathfrak{u}_{N}(\mathbb{K})}}{2}\munitaryfd(t)\textrm{dt} \\[+2pt]
\munitaryfd(0) = I_{N}.
  \end{array}
  \right.
\end{equation}

For all $t\geq 0$, $\munitaryfd(t)$ is an unitary matrix, a random variable with values in the dual Voiculescu group $\mathcal{O}\langle n \rangle$ is defined by:
\begin{equation*}
\unitaryfd(t):
\left\{
	\begin{array}{ccc}
		\mathcal{O}\langle nd \rangle & \rightarrow & L^{\infty}\left(\Omega, \mathcal{A}, \mathcal{M}_{d}(\mathbb{C}), \mathbb{P} \otimes \lmss{tr}\right) \\
			u_{ij} & \mapsto & \munitaryfd(t)(i,j)\\
			u^{\star}_{ij} & \mapsto & \left(\munitaryfd(t)(i,j)\right)^{\star}.
	\end{array}
	\right.
\end{equation*}
In Section \ref{squareextractions}, we study the convergence in non-commutative distribution of $\unitaryfd$ as the dimension $d \rightarrow +\infty$ to the free unitary Brownian motion. A crucial step toward this goal is to give formulae for mean of polynomials in the matrix $\munitaryfd(t)$, $t\geq 0$.

\par In the following proposition, let $i,j,k$ be three integers such that $i,j\leq k$, to a tensor $A \in \mathcal{M}_{N}\left(\mathbb{K} \right) \otimes \mathcal{M}_{N}\left(\mathbb{K} \right)$ we associate the endomorphism $\iota_{ij}(A) \in \mathcal{M}_{N}(\mathbb{K})^{\otimes k}$ that acts as:
$\iota_{ij}(A)(v_{1}\otimes v_{i}\otimes \cdots \otimes v_{j} \otimes v_{k}) = v_{1} \otimes \cdots A^{(1)}(v_{i})\otimes \cdots \otimes A^{(2)}(v_{j}) \otimes \cdots \otimes v_{k}$, $v_{1}\cdots\otimes v_{k} \in (\mathbb{R}^{N})^{\otimes k}$ if we use the Sweedler notation.
For the complex case, mean of tensor product of $\munitaryfdc$ and its conjugate are also needed. We denote by $\bar{\lmss{M}}_{1}$ the free monoid generated by $x_{1}$ and $\bar{x}_{1}$. If $A\in \mathcal{M}_{N}(\mathbb{C})$ and $w \in \bar{\lmss{M}}_{1}$, then $w^{\otimes}(A)$ denotes the monomial in $\mathcal{M}_{N}(\mathbb{C})^{\otimes k}$ obtained via the substitution $x_{1} \to A$ and $\bar{x}_{1} \to \bar{A}$.
\begin{proposition}[\cite{levy}]
\label{meantensorlevy}
Let $\mathbb{K}$ be one the three division algebra $\mathbb{R}, \mathbb{C}$ or $\mathbb{H}$. Let $k \geq 1$ be an integer and $t \geq 0$ a time. We have
\begin{equation*}
	\label{formula_tensor_process}
	\mathbb{E}\left[\left(\munitaryfd(t)\right)^{\otimes k}\right] = \exp \left( kt\frac{c_{\mathfrak{g}}}{2} + t\sum_{1 \leq i<j \leq n} \iota_{ij}(C_{\mathfrak{g}}) \right).
\end{equation*}
For the complex case, let $w\in \bar{\lmss{M}}_{2}$, then:
\begin{equation}
\label{formula_tensor_process_c}
	\mathbb{E}\left[w^{\otimes}\left(\munitaryfdc\right)(t) \right] = \exp\left(-\frac{kt}{2} + \sum_{\substack{1 \leq i, j \leq k, \\ w_{i} \neq w_{j}}} \iota_{i,j}\left(\lmss{P}\right) - \sum_{\substack{1 \leq i,j\leq k, \\ w_{i} = w_{j}}} \iota_{i,j}\left(\lmss{T}\right) \right).
\end{equation}
\end{proposition}

\section{Higher dimensional free Brownian motion}
\label{freemultbrown}
Let $n$ be an integer greater than one. Let $(w_{ij}^{\lmss{i}})_{1 \leq i < j \leq n}, \left\{ w_{i},~i \leq n \right\}$ with $~\lmss{i}\in\{1,2\}$ be three mutually free families of free Brownian motions on a tracial von Neumann algebra $(\mathcal{A}, \tau)$. We define the algebra $\mathcal{H} \langle n \rangle $ as the real unital algebra freely generated by $n(n+1)$ elements $(h_{ij})_{1\leq i \leq j \leq n}$ and $(h^{\star}_{ij})_{1\leq i \leq j \leq n}$. We turn $\hermitianalg$ into a $\star$ algebra by defining the involutive antimorphism $\star$ as $(h_{ij})^{\star} = h^{\star}_{ij}$.
The complexification of $\hermitianalg$ is denoted $\mathcal{H}^{\mathbb{C}}\langle n \rangle$. The involution $\star$ is extended as an anti-linear anti-morphism of $\mathcal{H}^{\mathbb{C}}\langle n \rangle$. We prefer to work with the real algebra $\hermitianalg$ since the random variables we are interested in are valued into real algebras (and we do not want to complexify those algebras).
For each time $t\geq 0$, we define a free noise process, that is a quantum process $\noisefree$ by setting for each time $t\geq 0$, the random variable $\noisefree(t) : \hermitianalg \rightarrow (\mathcal{A}, \tau)$ equal to:
\begin{equation*}
  \begin{split}
	  \noisefree_{ii}{}(t) = w_{i}(t), \, \noisefree_{ij}{}(t) = \frac{1}{\sqrt{2}}\left( w_{ij}^{1}(t) + \lmss{i}w_{ij}^{2}(t) \right), \noisefree_{ji}{}(t) = \left(\noisefree_{ij}{}(t)\right)^{\star} \quad i < j,
\end{split}
\end{equation*}
where $\noisefree_{ij}(t) = \noisefree(t)(h_{ij})$. We refer to the matrix with non-commutative entries $\noisefree$ as the \textit{free}$n$-\textit{dimensional Hermitian Brownian motion}. We define further the dual unitary group $\udualgroup$, in the sense of Voiculescu. As a real algebra $\udualgroup$ is generated by $2n^{2}$ variables $(o_{ij})_{1 \leq i,j \leq n}$ and $(o^{\star}_{ij})_{1 \leq i,j \leq N} $ subject to the relations:
\begin{equation}
\label{defining_relation}
\sum_{k=1}^{n}o_{ik}o^{\star}_{jk} = \sum_{k=1}^{n}o_{ki}^{\star}o_{kj} = \delta_{ij},~1 \leq i,j \leq n.
\end{equation}
To define morphisms on $\udualgroup$, it is convenient to introduce the matrix $O$ with entries in $\udualgroup$ and defined by $O_{ij} = o_{ij}$,~$i,j \leq n$.
With the help of these notations, the dual Voiculescu group is turned into a free bialgebra by defining a coproduct $\Delta$ that takes values into the free product $\udualgroup \sqcup \udualgroup$ and satisfies the equation:
\begin{equation*}
\Delta(O) = O_{|1} O_{|2}.
\end{equation*}
The associated counit $\varepsilon: \udualgroup \rightarrow \mathbb{R}$ is subsequently defined by $\varepsilon(O) = I_{n} \in \udualgroup \otimes \mathcal{M}_{n}(\mathbb{R})$.
In addition, the morphism $S$ of $\mathcal{O}$ that takes the values $S(o_{ij}) = o^{\star}_{ji}$ on the generators of $\udualgroup$ is an antipode for the free bi-algebra $\udualgroup$ : $(S \sqcup 1)(o_{ij}) = (1 \sqcup S)(o_{ij}) = o_{ij}$ for all $i,j \leq n$. The tuple $(\udualgroup, \Delta, \varepsilon, S)$ is a Zhang algebra in the category of involutive algebra.
\par Our goal now is to define the higher dimensional analog of the free unitary Brownian motion. This process is a one parameter family of random variables on the Voiculescu dual unitary group denoted by $\freeunitary$ which values $U_{ij}(t) = U(t)(o_{ij})$ $i,j \leq n$ on the generators of $\udualgroup$ satisfy the following free stochastic differential system:
\begin{equation}
  \label{free_unitary_n}
  \left\{
  \begin{array}{l}
	  \textrm{d}\mfreeunitary(t)(i,j) = \frac{\lmss{i}}{\sqrt{n}} \sum_{k=1}^{n}(\textrm{d}\noisefree_{t}(i,k))\mfreeunitary(t)(k,j) - \frac{1}{2}\mfreeunitary(t)(i,j)\textrm{d}t \\[2pt]
	   \mfreeunitary(0) = I_{n}
\end{array}
\right.
\end{equation}
The following lemma states that for each time $t\geq 0$, the entries of the matrix $\mfreeunitary$ satisfy the defining relations of $\udualgroup$. In the next lemma, we use the fact the algebra $\mathcal{A} \otimes \mathcal{M}_{n}(\mathbb{R})$ is an involutive algebra if endowed with the tensor product of the star anti-morphisms of $\mathcal{A}$ and $\mathcal{M}_{n}(\mathbb{R})$ (the transposition).
\begin{lemma}
\par For each time $t\geq 0$, the matrix $\mfreeunitary = \left(\mfreeunitary(t)(i,j)\right)_{1 \leq i,j\leq n}$ is an unitary element of $\mathcal{A}\otimes \mathcal{M}_{n}(\mathbb{C})$.
\end{lemma}
\begin{proof}
The defining relations of the algebra $\udualgroup$ admit the following compact form : $OO^{\star} = O^{\star}O = I_{n}$.
We compute the derivative of $t \mapsto \mfreeunitary\mfreeunitary^{\star}(t)$. Let $t\geq 0$ be a time and $i,j \leq n$ integers.
	\begin{equation*}
		\label{equation_lemma3_1}
		\tag{$\star$}
		\begin{split}
			\textrm{d}\left(\mfreeunitary(t)\mfreeunitary(t)^{\star}\right)(i,j) &= -\left(\textrm{d}\mfreeunitary(t)\mfreeunitary(t)^{\star}\right)(i,j) + \left(\mfreeunitary(t)\textrm{d}\mfreeunitary(t)^{\star}\right)(i,j) \\
			&= -\left(\mfreeunitary(t)\mfreeunitary(t)^{\star}\right)(i,j)\textrm{d}t \\ &\hspace{3.5cm}+\frac{1}{n}\sum_{k,l,q}\textrm{d}\noisefree_{t}(i,k)\mfreeunitary(t)(k,q)\mfreeunitary(t)^{\star}(q,l)\textrm{d}\noisefree_{t}(l,j) \\ &= -\mfreeunitary(t)\mfreeunitary(t)^{\star}  + \lmss{tr}(\mfreeunitary(t)^{\star}\mfreeunitary(t))I_{n}.
		\end{split}
	\end{equation*}
By applying the linear form $\lmss{tr}$ on the left and right hand sides of equation \eqref{equation_lemma3_1},we prove that $\lmss{tr}(\mfreeunitary(t)^{\star}\mfreeunitary(t)) \\= 1$ for all time $t\geq 0$. Inserting this last relation in equation $\eqref{equation_lemma3_1}$ gives
\begin{equation*}
\begin{split}
&\textrm{d}\mfreeunitary(t)\mfreeunitary(t)^{\star} = -\mfreeunitary(t)^{\star}\mfreeunitary(t) + I_{n},~\mfreeunitary_{0} = 1.
\end{split}
\end{equation*}
By using unicity of the solution of equation \eqref{equation_lemma3_1}, we prove that $\mfreeunitary(t)\mfreeunitary^{\star}_{t} = 1$. A similar argument shows that the relation $\mfreeunitary^{\star}_{t}\mfreeunitary(t) = I_{n}$ also holds.
\end{proof}
We use the symbol $\freeunitary$ for the process on the dual Voiculescu group which values on the generators of $\udualgroup$ are given by $\mfreeunitary$. It is not gard to prove, by using equation \eqref{free_unitary_n}, that $\freeunitary$ is actually a free Levy process on $\udualgroup$. We compute the derivative of the non cmmutative-distribution at time $t=0$ of $\freeunitary$. We define an operator $L_{n}$ acting on $\udualgroup$, that will be called the generator of $\freeunitary$, such that
\begin{equation*}
  \left. \frac{d}{dt} \right|_{t=0} \tau \circ \freeunitary(t)(u)=L_{n}(u).
\end{equation*}
We define in the next section what a generator on a bi-algebra is in the next Section and compute an associated Sch\"urmann triple. We will not say that much on what a Schürmann triple is, we only indicate that it is a non-commutative analogue of the famous L\'evy triple that is associated to every classical L\'evy process. As such, Sch\"urmann triples classify laws of non-commutative L\'evy processes and allow for the definition of what Gaussian, Poisson or pure drift quantum processes are. We provide, to put it informally, a combinatorial formula for $L_{n}$ by using the algebra of coloured Brauer diagrams we introduced in the last Section. This formula is important because it will allow for a comparison between the limiting distribution of Brownian motions on unitary groups in high dimension and the free process $\freeunitary$.
We fix an element $u = u_{i_{1},j_{1}}\cdots u_{i_{p},j_{p}} \in \udualgroup$ in the Voiculescu dual group and compute the derivative of $t \mapsto \tau(\freeunitary_{t}(u))$ at time $t=0$. By using the free It\^o formula, we find
\begin{equation}
\label{compgenun}
\begin{split}
&\textrm{d}\left(U_{t}(u^{\varepsilon(1)}_{i_{1}j_{1}}) \cdots U_{t}(u^{\varepsilon(p)}_{i_{p}j_{p}})\right) = \sum_{k=1}^{p}U_{t}(u^{\varepsilon(1)}_{i_{1}j_{1}}) \cdots \textrm{d}(U_{t}(u^{\varepsilon(k)}_{i_{k}j_{k}})) U_{t}(u^{\varepsilon(p)}_{i_{p}j_{p}})\\
&\hspace{5cm}+\sum_{1 \leq k < l \leq p}U_{t}(u^{\varepsilon(1)}_{i_{1}j_{1}}) \cdots \textrm{d}U_{t}(u^{\varepsilon(k)}_{i_{k}j_{k}})\cdots \textrm{d}U_{t}(u_{i_{l}j_{l}}^{\varepsilon(l)})\cdots U_{t}(u^{\varepsilon(p)}_{i_{p}j_{p}})
\end{split}
\end{equation}
We insert in the last equation the stochastic differential equation \ref{free_unitary_n} that is satisfied by the matrix $\mfreeunitary$ and $\mfreeunitary^{\star}$ and apply the trace $\tau$ to both side of the resulting equation, note that:
\begin{equation*}
\textrm{d}\mfreeunitary^{\star}_{t} = \frac{\lmss{-i}}{\sqrt{n}}\textrm{d}\noisefree_{t}\, \mfreeunitary^{\star}_{t} - \frac{1}{2}\mfreeunitary^{\star}_{t}.
\end{equation*}
We obtain for the derivative of $t\mapsto \tau(\freeunitary_{t}(u))$:
\begin{equation}
\label{compgendeux}
\begin{split}
&\textrm{d}\tau\left(U_{t}(u^{\varepsilon(1)}_{i_{1}j_{1}}) \cdots U_{t}(u^{\varepsilon(p)}_{i_{p}j_{p}})\right) = -\frac{p}{2}\tau(U_{t}(u^{\varepsilon(1)}_{i_{1}j_{1}}) \cdots U_{t}(u^{\varepsilon(k)}_{i_{k}j_{k}}) U_{t}(u^{\varepsilon(p)}_{i_{p}j_{p}}))\textrm{dt}\\
&\hspace{4cm}+\sum_{1 \leq k < l \leq p}\tau\left(U_{t}(u^{\varepsilon(1)}_{i_{1}j_{1}}) \cdots \textrm{d}U_{t}(u^{\varepsilon(k)}_{i_{k}j_{k}})\cdots \textrm{d}U_{t}(u_{i_{l}j_{l}}^{\varepsilon(l)})\cdots U_{t}(u^{\varepsilon(p)}_{i_{p}j_{p}})\right).
\end{split}
\end{equation}
We divide the second sum in \eqref{compgenun} according to the values of $(\varepsilon(k), \varepsilon(l))$, $k,l \leq n$. First, if $\varepsilon(k)=\varepsilon(l)$ we obtain:
\begin{equation}
\stepcounter{equation}
\label{unun}
\tag{$\arabic{equation}_{1}$}
\begin{split}
&\tau\left(U_{t}(u^{\varepsilon(1)}_{i_{1}j_{1}}) \cdots \textrm{d}U_{t}(u_{i_{k}j_{k}})\cdots \textrm{d}U_{t}(u_{i_{l}j_{l}}^{})\cdots U_{t}(u^{\varepsilon(p)}_{i_{p}j_{p}})\right)\\ &\hspace{5cm}= -\frac{1}{n}\tau\left(U_{t}(u^{\varepsilon(1)}_{i_{1}j_{1}}) \cdots \tau(U_{t}(u_{i_{l}j_{k}})\cdots )U_{t}(u_{i_{k}j_{l}})\cdots U_{t}(u^{\varepsilon(p)}_{i_{p}j_{p}})\right)
\end{split}
\end{equation}
\begin{equation}
\label{starstar}
\tag{$\arabic{equation}_{2}$}
\begin{split}
&\tau\left(U_{t}(u^{\varepsilon(1)}_{i_{1}j_{1}}) \cdots \textrm{d}U_{t}(u^{\star}_{i_{k}j_{k}})\cdots \textrm{d}U_{t}(u^{\star}_{i_{l}j_{l}})\cdots U_{t}(u^{\varepsilon(p)}_{i_{p}j_{p}})\right)\\ &\hspace{5cm}= -\frac{1}{n}\tau\left(U_{t}(u^{\varepsilon(1)}_{i_{1}j_{1}}) \cdots U_{t}(u^{\star}_{i_{k}j_{l}})\tau(\cdots U_{t}(u^{\star}_{i_{l}j_{k}}))\cdots U_{t}(u^{\varepsilon(p)}_{i_{p}j_{p}})\right)
\end{split}
\end{equation}
Secondly, if $\varepsilon(k) \neq \varepsilon(l)$:
\begin{equation}
\label{unstar}
\tag{$\arabic{equation}_{3}$}
\begin{split}
&\tau\left(U_{t}(u^{\varepsilon(1)}_{i_{1}j_{1}}) \cdots \textrm{d}U_{t}(u^{}_{i_{k}j_{k}})\cdots \textrm{d}U_{t}(u^{\star}_{i_{l}j_{l}})\cdots U_{t}(u^{\varepsilon(p)}_{i_{p}j_{p}})\right)\\ &\hspace{4.5cm}= \frac{1}{n}\tau\left(\delta_{i_{k}j_{l} }U_{t}(u^{\varepsilon(1)}_{i_{1}j_{1}}) \cdots \tau(U_{t}(u_{i_{k}i_{l}})\cdots U_{t}(u^{\star}_{j_{k}j_{l}}))\cdots U_{t}(u^{\varepsilon(p)}_{i_{p}j_{p}})\right)
\end{split}
\end{equation}

\begin{equation}
\label{starstar}
\tag{$\arabic{equation}_{4}$}
\begin{split}
&\tau\left(U_{t}(u^{\varepsilon(1)}_{i_{1}j_{1}}) \cdots \textrm{d}U_{t}(u^{}_{i_{k}j_{k}})\cdots \textrm{d}U_{t}(u^{\star}_{i_{l}j_{l}})\cdots U_{t}(u^{\varepsilon(p)}_{i_{p}j_{p}})\right)\\ &\hspace{4.5cm}= \frac{1}{n}\tau\left(\delta_{j_{l}j_{k} }U_{t}(u^{\varepsilon(1)}_{i_{1}j_{1}}) \cdots U_{t}(u^{\star}_{i_{k}j_{k}})\tau(\cdots) U_{t}(u^{}_{i_{l}i_{k}})\cdots U_{t}(u^{\varepsilon(p)}_{i_{p}j_{p}})\right)
\end{split}
\end{equation}
In the sequel, we use the coloured Brauer algebra $\brauer(\underset{n}{\underbrace{1,\ldots,1}})$ and its representation $\rho_{(1,\ldots,1)}$ to write formulae \eqref{unun}-\eqref{starstar}. We set $c_{p} = (1,\ldots,p)$ and consider $c_{p}$ alternatively as a permutation or as a non-coloured Brauer diagram. The non-coloured diagram $\ncc{b}_{\varepsilon}$ is obtained by twisting $c_{p}$ at positions $i's$ such that $\varepsilon(i)=\star$:
\begin{equation*}
\ncc{b}_{\varepsilon} = \big(\prod_{i:\varepsilon(i)=\star}^{n} \lmss{Tw}_{i}\big)(c_{p}).
\end{equation*}
We colourize $\ncc{b}_{\varepsilon}$ with the colourization $c$ defined by : $c(k) = i_{k},~c(k^{\prime})=j_{k}$ if $\varepsilon(k)=1$ and $c(k)=j_{k},~c(k^{\prime})=i_{k}$ if $\varepsilon(k)=\star$ to obtain a coloured Brauer diagram denoted $\ncc{b}_{\varepsilon}$.
We claim that each of the equations $\eqref{unun}-\eqref{starstar}$ can be put in the following form:
\begin{equation*}
\begin{split}
&\tau\left(U_{t}(u^{\varepsilon(1)}_{i_{1}j_{1}}) \cdots \textrm{d}U_{t}(u_{i_{k}j_{k}}^{\varepsilon(k)})\cdots \textrm{d}U_{t}(u_{i_{l}j_{l}}^{\varepsilon(l)})\cdots U_{t}(u^{\varepsilon(p)}_{i_{p}j_{p}})\right)\\
&\hspace{5cm}=\frac{-1}{n}\left(\tau \otimes \lmss{Tr}^{\otimes n}\right)\left[\rho_{(1\ldots,1)}(\ncc{\tau}_{kl}\circ b_{\varepsilon,\lmss{i},\lmss{j}})\circ\mfreeunitary(t)^{\otimes k}\right]\text{dt},\text{ if } \varepsilon(k)=\varepsilon(l)\\
&\tau\left(U_{t}(u^{\varepsilon(1)}_{i_{1}j_{1}}) \cdots \textrm{d}U_{t}(u^{\varepsilon(k)}_{i_{k}j_{k}})\cdots \textrm{d}U_{t}(u^{\varepsilon(l)}_{i_{l}j_{l}})\cdots U_{t}(u^{\varepsilon(p)}_{i_{p}j_{p}})\right)\\
&\hspace{5cm}=\frac{1}{n}\left(\tau \otimes \lmss{Tr}^{\otimes n}\right)\left[\rho_{(1,\ldots,1)}(\ncc{e}_{kl}\circ b_{\varepsilon,\lmss{i},\lmss{j}})\circ\mfreeunitary(t)^{\otimes k}\right]\text{dt}, \text{ if } \varepsilon(k)\neq\varepsilon(l).
\end{split}
\end{equation*}
A combinatorial formula for the generator of the process $\freeunitary$ follows readily from this last four formulae. In fact, by using equation \eqref{compgendeux} and the characterisation of the sets $\trp{b^{\varepsilon}}$ and $\prp{b^{\varepsilon}}$ we gave in Section \ref{schur_weyl}, we get:

\begin{flushleft}
\resizebox{\hsize}{!}{
\begin{minipage}{\hsize}
\begin{align}
\left.\frac{d}{dt}\right|_{t=0} \tau\left(U_{t}(u^{\varepsilon(1)}_{i_{1}j_{1}}) \cdots U_{t}(u^{\varepsilon(p)}_{i_{p}j_{p}})\right) &= -\frac{p}{2}\delta_{\lmss{i},\lmss{j}}-\frac{1}{n}\sum_{\substack{\ncc{\tau} \in \trp{b_{\varepsilon,\lmss{i},\lmss{j}}}}} \delta_{\Delta}(\ncc{\tau} \circ b_{\varepsilon,\lmss{i},\lmss{j}}) + \frac{1}{n} \sum_{\substack{\ncc{e} \in \prp{b_{\varepsilon,\lmss{i},\lmss{j}}}}}\delta_{\Delta}(\ncc{e}\circ b_{\varepsilon,\lmss{i},\lmss{j}})\nonumber\\
&=\mathcal{L}_{n}(u^{\varepsilon(1)}_{i_{1}j_{1}} \cdots u^{\varepsilon(p)}_{i_{p}j_{p}}),\label{generatorfree}
\end{align}
\end{minipage}}
\end{flushleft}
with $\delta_{\Delta}$ the support function of the set of diagonally coloured Brauer diagrams: $\delta_{\Delta}(b) = 1 \Leftrightarrow c_{b}(i)=c_{b}(i^{\prime})$.

\subsection{Sch\"urmann triple for the higher dimensional free unitary Brownian motion}
\label{subsec:Schurmann}
\par For a detailed introduction to the theory of quantum stochastic calculus and Sch\"urmann triple, the reader is directed to \cite{franz2004theory} and \cite{schurmann}.

\begin{definition}
\label{generator}
	Let $B$ be a unital associative complex free bi-algebra. A \emph{generator} is a complex linear functional $L : B \to\mathbb{C}$ satisfying the properties:
	\begin{enumerate}[\indent 1.]
			\item $L(1_{B}) = 0$,~$L(b^{\star}) = \overline{L(b)}$,~\textrm{for all} $b \in B$,
			\item $L(b^{\star}b) \geq 0$ for all $b \in B$ with $\varepsilon(b) = 0$.
	\end{enumerate}
\end{definition}
Generators show up naturally if differentiating a convolution semi-group of states. In the present work, we are interested in two types of (convolution) semi-groups, which are tensor semi-groups and free semi-groups (see \cite{franz2004theory}):
\begin{equation}
\label{semigroups}
\alpha_{s+t}=\alpha_{s} \,\hat{\otimes}\, \alpha_{t} = (\alpha_{s} \otimes \alpha_{t}) \circ \Delta~\text{(tensor)},~\alpha_{s+t}(\alpha_{s} \,\hat{\sqcup}\,\alpha_{t}) = \left(\alpha_{s}\sqcup\alpha_{t}\right) \circ \Delta~\text{(free)}~,
\end{equation}
and the continuity condition,
\begin{equation*}
  \lim_{s\to 0^{+}} \alpha_{s} = \varepsilon
\end{equation*}
where $s,t \geq 0$ are two times. The word convolution is used to indicate the similitude of the products $\hat{\otimes}$ and $\hat{\free}$ with the usual convolution product of functions on a group.
If $j$ is a free L\'evy process on a bi-algebra $B$, its one dimensional marginals, $\left(\tau \circ j_{t}\right)_{t\geq 0}$ constitute a free semi-group.
Let $B$ be a free bi-algebra. Below we state the Schoenberg correspondence, that relates precisely some type semi-groups to generators. In Proposition \ref{schoenberg}, we use the symbol $\exp_{\star}$ to denote the exponentiation with respect to a convolution product on $B^{\star}$ denoted $\star$, equal either to the tensor convolution product or the free convolution product? The correspondence holds also for what is called the boolean convolution product, obtained by replacing the free product of states in equation \eqref{semigroups} by boolean product of states (see \cite{nico1}).
\begin{proposition}[Schoenberg correspondence]
\label{schoenberg}
Let $\star$ be a convolution product on $B^{\star}$.
\begin{enumerate}
\item Let $\psi:B\to\mathbb{C}$ be a linear functional on $B$, then the series:
\begin{equation*}
\exp_{\star}(\psi)(b) = \sum_{k=1}^{\infty}\frac{\psi^{\star n}}{n!}(b)
\end{equation*}
converges for all $b \in B$.
\item Let $\left(\phi_{t}\right)_{t\geq 0}$ be a convolution semi-groups, with respect to the product $\star$ on $B^{\star}$, then
\begin{equation*}
L(b) \lim_{t\to 0} \frac{1}{t}\left(\phi_{t}(b)-\varepsilon(b) \right) \text{ exists }
\end{equation*}
for all $b\in B$. Furthermore, $\exp_{\star}(tL)(b) = \phi_{t}(b)$ and the two following statements are equivalent:
	\begin{enumerate}[\indent a.]
		\item L is a generator,
		\item $\phi_{t}$ is a state for all $t\geq 0$: $\phi_{t}(bb^{\star}) \geq 0$ and $\phi_{t}(bb^{\star}) = 0 \Leftrightarrow b=0$, $b \in B$.
	\end{enumerate}
\end{enumerate}

\end{proposition}
In the next definition, we introduce the central object of this section, the Sch\"urmann triple. Let $(D,\langle~\cdot~|~\cdot~ \rangle)$ be a pre-Hilbert space, we denote by $\mathcal{L}(D)$ the vector space of all linear operators on $D$ that have an adjoint defined everywhere on $D$:
\begin{equation*}
\mathcal{L}(D) = \left\{ A:D\to D:\exists A^{\star}:D \to D,~\langle A(v),w \rangle = \langle v, A^{\star}(w)\rangle ,~v,w \in D\right\}.
\end{equation*}
\begin{definition}[Sch\"urmann triple]
\label{schurmanntriple}
 A Sch\"urmann triple on $(B, \Delta, \varepsilon)$ is triple $(\pi, \eta, L)$ with
	\begin{enumerate}[\indent\indent 1.]
		\item a unital $\star$-representation $\pi : \mathcal{B} \rightarrow \mathcal{L}(D)$ on a pre-Hilbert space $D$,
		\item a linear map $\eta : B  \rightarrow D$ verifying
				\begin{equation}
				\label{cocyle1}
				\tag{1CC}
					\eta(ab) = \pi(a)\eta(b) + \eta(a)\varepsilon(b),
				\end{equation}
			\item a generator $L$ such that:
			\begin{equation}
			\label{coboundary}
			\tag{2CB}
				- \langle \eta(a^{\star}), \eta(b)\rangle = \varepsilon(a)L(b) - L(ab) + L(a)\varepsilon(b).
			\end{equation}
	\end{enumerate}
\end{definition}

A map $\eta: B \to D$ satisfying $\eqref{cocyle1}$ is called a $\pi$-$\varepsilon$ cocyle and a map $L:B\to\mathbb{C}$ satisfying condition \eqref{coboundary} is called a $\varepsilon$-$\varepsilon$ coboundary.
A schurman triple is said to be $\emph{surjective}$ if the cocycle $\eta$ is a surjective map.
\begin{proposition}
	With the notation introduced so far, there is a one-to-one correspondence between surjective Sch\"urmann triples, generators and convolution semi-groups $\{ \phi_{t}, t \geq 0 \}$.
\end{proposition}
Gaussian and drift generators can be classified using Sch\"urmann triples. In particular, the next definition introduces the notion of Gaussian processes on $\mathcal{O}\langle n \rangle$
\begin{definition} Let $\left( j_{t} \right)_{t \geq 0}$ be a quantum L\'evy process on $\mathcal{O}\langle n \rangle$. Let $(\pi,\eta,L)$ be a Sch\"urmann triple associated with $j$. The process $j$ is said to be \emph{Gaussian} if one of the following equivalent conditions hold:
	\begin{enumerate}[\indent 1.]
		\item For each $a,b,c \in \lmss{Ker}(\varepsilon)$, we have $L(abc) = 0$,
		\item For each $a,b,c \in \lmss{Ker}(\varepsilon)$, we have $L(b^{\star}a^{\star}ab)=0$,
		\item For each $a,b,c \in \mathcal{O}\langle n \rangle$, the following formula holds
			\begin{equation*}
				\begin{split}
					L(abc) &= L(ab)\varepsilon(c) + L(ac)\varepsilon(b) + \varepsilon(a)L(bc) - \varepsilon(a)\varepsilon(b)L(c) \\
					&\phantom{=}- \varepsilon(a)\varepsilon(c)L(b) - L(a)\varepsilon(b)\varepsilon(c)
				\end{split}
		\end{equation*}
	\item The representation $\pi$ is zero on $\lmss{Ker}(\varepsilon)$,
	\item For each $a\in \mathcal{O}\langle n \rangle, \pi(a) = \varepsilon(a)1$
	\item For each $a,b \in \lmss{Ker}(\varepsilon)$, we have $\eta(ab) = 0$
	\item For each $a,b \in \mathcal{O}\langle n \rangle$, $\eta(ab) = \varepsilon(a)\eta(b) + \eta(a)\varepsilon(b)$
	\end{enumerate}
\end{definition}
\begin{proposition}[\cite{ulrich2015construction}]
	\label{ulrich_triple}
	Take $D = \mathcal{M}_{N}(\mathbb{C})$. We define a Sch\"urmann triple for $\freeunitary$ by setting
	\begin{equation}
	\label{defschurmann}
		\begin{split}
		&\eta(u_{ij}) = \varepsilon_{ij},~\eta(u^{\star}_{ij}) = \varepsilon_{ij}, \\
		&\pi(u_{ij})  = \delta_{ij}1, \\
		& \mathcal{L}_{n}(u_{ij})  = -\frac{1}{2}\delta_{ij} = \frac{1}{2} \sum_{r=1}^{n}\langle \eta(u^{\star}_{ij}), \eta(u_{ij}) \rangle
		\end{split}
	\end{equation}
In the last equation, $\langle~ \cdot~,~\cdot~\rangle = \lmss{tr}(~\cdot^{\star}~\cdot~)$
\end{proposition}
\begin{proof}
	\par Let $n\geq 1$ an integer. The operator $\mathcal{L}_{n}$ is denoted $L$. First we prove that the three operators $\left(\pi,\eta,L\right)$ defined by their values given in equation \eqref{defschurmann} on the generators of $\mathcal{O}\langle n\rangle$ exist. This is trivial for the representation $\pi$. For $\eta$ and $L$ we have to check that \eqref{defschurmann} is compatible with the defining relations of the algebra $\mathcal{O}\langle n \rangle$. Denote by $\mathcal{F}(n^{2})$ the free algebra with $n^{2}$ generators. Let $\eta:\mathcal{F}(n^{2}) \to D $ be the extension of the values $\eqref{defschurmann}$ by using the cocycle property $\eqref{cocyle1}$. Denote also by $L: \mathcal{F}(n^{2}) \to \mathbb{C}$ the operator extending the values \eqref{defschurmann} by using point \eqref{coboundary} in definition \ref{schurmanntriple}. For the maps $\eta$ and $L$ to descend to the quotient of the free algebras $\mathcal{F}(n^{2})$ by the defining relations of $\mathcal{O}\langle n\rangle$, we have to check
	\begin{equation*}
	\eta(\sum_{r}u^{\star}_{ri}u_{rj}) = 0,~L(\sum_{r}u^{\star}_{ri}u_{rj}) = 0.
	\end{equation*}
First, let $1 \leq i,j,r \leq n$,
	\begin{equation*}
		\begin{split}
			L(u^{\star}_{ri}u_{rj}) &= \langle \eta(u_{ri}),\eta(u_{rj})\rangle + \varepsilon(u^{\star}_{ri})L(u_{rj}) + L(u_{rj}^{\star}) \varepsilon(u_{rj}) \\
			&= \langle \varepsilon(ri), \varepsilon(rj)\rangle + \varepsilon(u^{\star}_{ri})L(u_{rj}) + L(u^{\star}_{ri})\varepsilon(u_{rj}) \\
		&= \frac{1}{n} \delta_{ir}\delta_{rj} -\frac{1}{2}\delta_{ri}\delta_{rj} - \frac{1}{2}\delta_{ri}\delta_{rj}.
		\end{split}
	\end{equation*}
By summing the last equation over $1 \leq r \leq n$, we obtain $L(\sum_{r=1}^{n}u^{\star}_{ri}u_{rj}) = 0$. Also, using property $\eqref{cocyle1}$,
 \begin{equation*}
	 \eta(\sum_{r=1}^{n} u_{ri}^{\star}u_{rj}) = -\sum_{r=1}^{n}\delta_{ri}\varepsilon_{rj} + \sum_{r}\varepsilon_{ir}\delta_{rj} = 0
 \end{equation*}
 By construction, $\left(L,\eta,\pi \right)$ is a Sch\"urmann triple. It is easy to show by induction the following formula for the cocycle $\eta$, $\varepsilon_{i} \in \{1,\star\}$, $1 \leq a_{i},b_{i}\leq n$, $~1 \leq i\leq p$,
	\begin{equation}
	\label{formulaeta}
	\tag{$\star$}
	\eta\left(u_{a_{1},b_{1}}^{\varepsilon_{1}} \cdots u_{a_{p},b_{p}}^{\varepsilon_{p}}\right)= \sum_{k\leq p}\delta_{a_{1},b_{1}} \cdots (-1)^{(\varepsilon_{k} ~ = ~ \star)}\delta_{a_{k}b_{k}} \cdots \delta_{a_{p},b_{p}}
	\end{equation}
On the generators, the linear functional $L$ and $L_{n}$ agree: $L(u_{ij}^{\varepsilon}) = L_{n}(u_{ij}^{\varepsilon})$, $i,j \leq n$. We prove by induction on length of words on generators of $\mathcal{O}\langle n \rangle$, that $L$ and $L_{n}$ agree. Assume that $L$  and $L_{n}$ agree on words of length less than $m$ and let $w$ a word on the alphabet $\{u_{i,j},u^{\star}_{i,j}~ 1 \leq i,j \leq n \}$ of length $m+1$, $w = \tilde{w}u_{i_{m+1}j_{m+1}}^{\varepsilon_{m+1}}$ with $\tilde{w}$ a word of length $m$.
\begin{equation*}
L_{n}(w) = L(\tilde{w})\varepsilon(u_{i_{m+1}j_{m+1}}^{\varepsilon_{m+1}}) -\frac{1}{n}\sum_{\tau_{i,p+1}\in\trp{b^{\varepsilon}_{\lmss{i},\lmss{j}}}} \delta_{\Delta}(\tau_{i,p+1}\circ b^{\varepsilon}_{\lmss{i},\lmss{j}}) + \frac{1}{n} \sum_{e_{i,p+1}\in\prp{b^{\varepsilon}_{\lmss{i},\lmss{j}})}}\delta_{\Delta}(e_{i,p+1} \circ b^{\varepsilon}_{\lmss{i},\lmss{j}}).
\end{equation*}
We compute next the remaining two sums. Let $1 \leq k \leq p$ be an integer, owing to:
\begin{equation*}
\delta_{\Delta}(\tau_{k,p+1} \circ b^{\varepsilon}_{\lmss{i},\lmss{j}}) = \delta_{i_{1},j_{1}}\cdots \delta_{i_{k}j_{p+1}} \cdots \delta_{i_{p+1},j_{k}},~ \delta_{\Delta}(e_{k,p+1} \circ b^{\varepsilon}_{\lmss{i},\lmss{j}}) = \delta_{i_{1},j_{1}}\cdots \delta_{j_{k}j_{p+1}} \cdots \delta_{i_{p+1},i_{k}},
\end{equation*}
we can write the last formula for $L_{n}(w)$ in a more explicit form as:
\begin{equation*}
L_{n}(w)=L(\tilde{w})\varepsilon(u_{i_{m+1}j_{m+1}}^{\varepsilon_{m+1}})+\frac{1}{n}\sum_{k}(-1)^{1+(\varepsilon_{p+1}~ = \varepsilon_{k}} \delta_{\lmss{i}_{1},\lmss{j}_{1}} \cdots
		\left\{
			\begin{array}{cc}
				\varepsilon_{p+1} = \varepsilon_{k} & \delta_{\lmss{i}_{p+1},\lmss{i}_{k}} \delta_{\lmss{j}_{p+1},\lmss{j}_{k}} \\
				\varepsilon_{p+1} \neq \varepsilon_{k} & \delta_{\lmss{i}_{p+1},\lmss{j}_{k}} \delta_{\lmss{j}_{p+1},\lmss{i}_{k}}
			\end{array}
		\right\}
		\cdots \delta_{\lmss{i}_{p},\lmss{j}_{p}}.
\end{equation*}
On the other hand, formula \eqref{formulaeta} implies:
\begin{equation*}
\resizebox{\hsize}{!}{$
	\begin{split}
		\langle \eta(u^{\neg \varepsilon_{p+1}}_{\lmss{i}_{p+1},\lmss{j}_{p+1}}), \eta(u^{\varepsilon_{1}}_{\lmss{i}_{1},\lmss{j}_{1}} \cdots u^{\varepsilon_{p}}_{\lmss{i}_{p},\lmss{j}_{p}} )\rangle &
		= \frac{1}{n}\sum_{k=1}^{p}(-1)^{1+(\varepsilon_{p+1}=\varepsilon_{k})} \delta_{\lmss{i}_{1},\lmss{j}_{1}} \cdots \lmss{Tr}\left(\varepsilon^{\neg \varepsilon_{p+1}}_{\lmss{i}_{p+1},\lmss{j}_{p+1}} \varepsilon^{\varepsilon_{k}}_{\lmss{i}_{k}\lmss{j}_{k}}\right) \cdots \delta_{\lmss{i}_{p},\lmss{j}_{p}}\\
		&=\frac{1}{n}\sum_{k=1}^{p}(-1)^{1+(\varepsilon_{p+1}=\varepsilon_{k})} \delta_{\lmss{i}_{1},\lmss{j}_{1}} \cdots
		\left\{
			\begin{array}{ll}
				\varepsilon_{p+1} = \varepsilon_{k} & \delta_{\lmss{i}_{p+1},\lmss{i}_{k}} \delta_{\lmss{j}_{p+1},\lmss{j}_{k}} \\
				\varepsilon_{p+1} \neq \varepsilon_{k} & \delta_{\lmss{i}_{p+1},\lmss{j}_{k}} \delta_{\lmss{j}_{p+1},\lmss{i}_{k}}
			\end{array}
		\right\}\delta_{\lmss{i}_{p},\lmss{j}_{p}}
	\end{split}$}
\end{equation*}
This achieves the proof of Proposition $\ref{ulrich_triple}$.
\end{proof}
\subsection{Some cumulants of higher dimensional free Brownian motions} In this section we compute some cumulants functions of higher dimensional Brownian motion's distribution. Our main result is contained in Proposition \ref{cumulant}. Computing mixed cumulants of $\{\freeunitary,\freeunitary^{\star}\}$ is rather difficult task, see for example \cite{demni2015star} in which formulae for only two types of mixed cumulants $\{\freeunitary,\freeunitary^{\star} \}$ are proved. We will not address this question for the process $\freeunitary$, although it would be interesting to.
\subsubsection{Free cumulants}
Let $n\geq 1$, we consider the set $\lmss{NC}(\{1,\ldots,n\}) = \lmss{NC}(n)$ of all non-crossing partitions of $\{1,\ldots,n\}$. A generic partition in $\lmss{NC}(n)$ will be denoted $\pi$ (sometimes $\rho$). A notation for $\pi = \{V_{1},\ldots,V_{k}\}$ where $V_{1},\ldots,V_{k}$ are called the blocks of $\pi$.
\par On \lmss{NC}(n), we consider the partial order given by reversed refinement, where for $\pi, \rho \in \lmss{NC}(n)$ we have $\pi \leq \rho$ if and only if every blocks of $\rho$ is a union of blocks of $\pi$. The minimal partition for the order $\leq$ is the non crossing partitions having $n$ blocks (denoted $\mathbf{0}_{n}$ and the maximal element is the partition having only one block (denoted $\mathbf{1}_{n}$).
The M\"obius function on $\lmss{NC}(n)$ will be denoted $\mu$. This function is defined on $\{(\pi,\rho)~|~\pi,\rho \in \lmss{NC}(n), \pi \leq \rho\}$. We will use only the Moebius function restricted to set of pairs $(\mathbf{0}_{n}, \pi)$ with $\pi \in \lmss{NC}(n)$ for which we have
\begin{equation*}
  \mu(\mathbf{0}_{n}, \pi) = \displaystyle\pi_{W \in \pi} (-1)^{\sharp W}C_{\sharp W -1}
\end{equation*}
where for $k \in \mathbb{N}$,
$$
C_{k}=\frac{(2k)!}{k!(k+1)!}
$$
is the $k^{th}$ Catalan number.
\par Let $(A,\phi)$ be a non commutative probability space. The $n^{th}$ moment functional of $(A,\phi)$ is the multilinear functional $\phi_{n}:A^{n}\to\mathbb{C}$ defined by $\phi_{n}(a_{1},\ldots,a_{n})=\phi(a_{1}\cdots a_{n})$, $a_{1},\ldots,a_{n} \in A$.
\par The $n^{rh}$ cumulant functional of $(A,\phi)$ is the multilinear functional $k_{n}:A^{n}\to \mathbb{C}$ defined by
\begin{equation}
  \label{momentcumulant}
  k_{n}(a_{1},\ldots,a_{n})=\sum_{\pi \in \lmss{NC}(n)}\mu(\pi,\mathbb{1}_{n})\cdot \displaystyle\pi_{\{i_{1}<\ldots<i{k}\}\in\pi}\phi_{k}(a_{i_{1}},\ldots,a_{i_{k}})
\end{equation}
We refer to the equation \eqref{momentcumulant} as the \emph{moment-cumulant formula}.
The cumulants functionals satisfy some important properties:
\begin{enumerate}[\indent 1.]
  \item Invariance under cyclic permutations of the entries
    \begin{equation*}
      k_{n}(a_{1},\ldots,a_{n})=k_{n}(a_{m},\ldots a_{n},a_{1},\ldots,a_{1}),~\forall 1 \leq m \leq n\text{ and }a_{1},\ldots,a_{n} \in A,
    \end{equation*}
  \item If $\mathcal{C}\subset A$ is a commutative algebra, then
  \begin{equation*}
    k_{n}(c_{1},\ldots,c_{n})=k_{n}(c_{n},\ldots,c_{2},c_{1}),~\forall n\geq 1\text{ and } c_{1},\ldots,c_{n} \in \mathcal{C}.
  \end{equation*}
  \item $k_{n}(a_{1},\ldots,a_{n})=\overline{k_{n}(a_{n}^{\star},\ldots,a_{1}^{\star})}$.
\end{enumerate}
\subsection{Computations of cumulants of the higher dimensional free Brownian motion}
Let $p,n\geq 1$ two integers. We use the symbol $\mathcal{C}_{2p}$ for the set comprising all sequences of the form $((i_{1},j_{1}),\ldots,(i_{p},j_{p}))$ with $i_{l},j_{l} \in \{1,\ldots,n\}$ and refer to an element of $\mathcal{C}_{2p}$ as a colourization. For each time $t\geq 0$, we denote by $\phi_{t} \in \mathcal{O}\langle n \rangle^{\star}$ the distribution of the process $\freeunitary$. Let $p\geq 1$ an integer, $\pi$ a non-crossing partition in $\lmss{NC}_{p}$ and define $\phi_{t}(\pi): \mathcal{O}\langle n \rangle^{p} \to \mathbb{C}$,  $\phi_{t}(\pi): \mathcal{C}_{2p} \to \mathbb{C}$ by
\begin{equation}
\phi_{t}(\pi)(a_{1},\ldots,a_{p}) = \prod_{b\in \pi}\phi_{t}\Big(\,\overset{\rightarrow}{\prod_{k\in b}}a_{k}\,\Big),~
\phi_{t}(\pi)(\lmss{i},\lmss{j}) = \phi_{t}^{\pi}(u_{i_{1},j_{1}},\ldots,u_{i_{p},j_{p}}).
\end{equation}
The function $\phi(pi)$ taking as argument a colourization is introduced for easing exposition.
The group $\mathfrak{S}_{p}$ of permutations acts on a finite sequence $\lmss{i} \in \{1,\ldots,n\}^{p}$ in a canonical way:
\begin{equation*}
s \cdot \sigma = (s_{\sigma(1)},~\ldots,~s_{\sigma(p)}),~s\in \{1,\ldots,n\}^{p},~\sigma \in \mathfrak{S}_{p}.
\end{equation*}

The following lemma is a downward consequence of equation \eqref{generatorfree} for the derivative of $\phi$.
\begin{lemma}
	With the notation above, for each time $t\geq 0$, non crossing partition $\pi$ and colourization $(\lmss{i},\lmss{j})\in \mathcal{C}_{2p}$,
\begin{equation}
	\label{diff_equation_3}
	\frac{d}{dt}\phi_{t}^{\pi}(\lmss{i}, \lmss{j}) = -\frac{p}{2} \phi^{\pi}_{t}(\lmss{i}, \lmss{j})- \frac{1}{n}\sum_{\tau \in \lmss{T}^{+}(\sigma_{\pi})} \phi_{t}^{\pi_{\sigma_{\pi} \circ \tau}}(\lmss{i} \cdot \tau, \lmss{j}).
\end{equation}
\end{lemma}
We fix once for all a non-crossing partition $\pi$ and a colourization $(\lmss{i},\lmss{j}) \in \mathcal{C}$. We solve the differential equation \eqref{diff_equation_3}. We introduce the normalized function:
\begin{equation*}
L(\lmss{i},\lmss{j},\pi)(s) = e^{\frac{ps}{2}}\phi_{s}^{\pi}(\lmss{i},\lmss{j}),~s \in \mathbb{R}_{+}.
\end{equation*}
For each integer $k \in \{0,\ldots,p-1\}$, we denote by $P^{k}(\lmss{i}, \lmss{j},\pi) \in \mathbb{C}$ $k^{th}$ coefficient of the Taylor expansion of $\mathbb{R}^{+} \ni s \mapsto L(\lmss{i}, \lmss{j},\pi)(s)$ (we prove in a moment this function is polynomial in its time variable). Owing to formula \eqref{diff_equation_3},
\begin{equation}
\label{recurrencerelation1}
  \begin{split}
	  \frac{d}{ds} L(\lmss{i}, \lmss{j}, \pi)(s) = -\frac{1}{n}\sum_{\tau \in \trp{\sigma_{\pi}}} L(\tau \cdotp \lmss{i}, \lmss{j},\pi_{\sigma_{\pi} \circ \tau})(s),~s\in\mathbb{R}_{+}.
  \end{split}
\end{equation}
Frm the definition of $\phi(\pi,\lmss{i}, \lmss{j})(s)$ as a product over the blocks of $\pi$, we prove that:
\begin{equation}
	L_{\lmss{i},\lmss{j}}^{\pi} (s) = \prod_{V \in \pi} L_{\lmss{i}_{|V},\lmss{j}_{|V}}^{1_{\sharp V}}(s),~s \geq 0,~ (\lmss{i},\lmss{j})\in \mathcal{C}_{p},~ \pi \in \lmss{NC}(p)
	\label{multiplicativeL}
\end{equation}
\par We gave now the argument to prove that $s \mapsto L_{\lmss{i}, \lmss{j}}^{\pi}(s)$ is polynomial. Let $N$ be the operator acting on functions of non-crossing partition and colourizations such that
\begin{equation*}
  \frac{d}{ds}L(\pi,\lmss{i},\lmss{j}) = N(L(s,\cdot,\cdot))(\pi,\lmss{i},\lmss{j}).
\end{equation*}
The operator $N$is a nilpotent operator of order $p$. In fact, if $\mathbf{1}$ is the constant function equals to $1$ on $\lmss{NC}(p) \times \mathcal{C}_{2p}$, then
$
\lmss{L}^{s}(\mathbf{1})((p\ldots,1), \lmss{i}, \lmss{j})
$
is the number of minimal factorisations of the cycle $(p\ldots1)$ of length $s$. Thus $\lmss{L}^{p-1}(\mathbf{1}) \neq 0$ and for all $s \geq p$, $\lmss{L}^{s}(\mathbf{1}) = 0$. Since $L(f) \leq L(\mathbf{1}) \sup(f)$, we have $L^{s}(f) = 0,~\forall s \geq p$. Hence $\lmss{L}$ is nilpotent of order $p$ which implies that the exponential of $\lmss{L}$ is a finite sum and the function $s \mapsto L^{\pi}_{\lmss{i},\lmss{j}}(s)$ is indeed a polynomial of degree stricly less than $p$.
\par Owing to equation \eqref{recurrencerelation1}, the following inductive relationship for the Taylor coefficients holds:
\begin{equation}
	\label{recurelationP1}
	\tag{R}
	\begin{split}
		&P_{\lmss{i}, \lmss{j}}^{\pi, k} = \frac{1}{n}\sum_{\tau \in \lmss{T}^{+}(\sigma)}P_{\tau \cdotp \lmss{i}, \lmss{j}}^{\pi_{\sigma_{\pi} \circ \tau}, k-1},\\
		&P_{\lmss{i}, \lmss{j}}^{\pi, 0}= L_{\lmss{i}, \sigma}(0) = \prod_{i = 1}^{n}\delta_{\lmss{i}_{i}, \lmss{j}_{i}},\quad k\leq p-1,~\pi \in \lmss{NC}_{p},~(\lmss{i},\lmss{j}) \in \mathcal{C}_{p}.
	\end{split}
\end{equation}
In particular, the coefficients $P^{\pi,0}_{\lmss{i},\lmss{j}}$ are independent of the non-crossing partition $\pi$.
Recall that the type of a permutation $\sigma \in \mathfrak{S}_{p}$ is a sequence $\lmss{t}(\sigma) = \left(n_{i}\right)_{1 \leq i \leq p}$ of $p$ integers with $n_{i}$ the number of cycles of $\sigma$ of length $i$, $1 \leq i\leq p$.
The geodesic distance $\lmss{d}(\sigma)$ of a permutation $\sigma \in \mathfrak{S}_{p}$ to the identity $\mathrm{id} \in \mathfrak{S}_{p}$ is the minimal numbers of transpositions needed to write $\sigma$ as a product of transpositions:
\begin{equation*}
\lmss{d}(\sigma) = \min\{\ell \geq 0 : \sigma = \tau_{1}\cdots\tau_{\ell},~\tau_{i} \in \lmss{T}_{p},~ 1 \leq i \leq p\}.
\end{equation*}
The geodesic distance can be computed as $\lmss{d}(\sigma) = p-\sharp(\sigma)$ with $\sharp(\sigma)$ the number of cycles of $\sigma$, $\sigma \in \mathfrak{S}_{p}$.
\begin{lemma}
	\label{lemmaun}
	Let $(\lmss{i},\lmss{j}) \in \mathcal{C}_{2p}$ be a colourization and $\pi \in \lmss{NC}(p)$ a non-crossing partition, it holds that:
  \begin{equation}
	  \label{lemmaunformula}
	  P_{\lmss{i}, \lmss{j}}^{\pi, k} = \frac{1}{n^{k}} \sum_{\tau_{1},\ldots,\tau_{k} \in \lmss{T}^{k}_{+}(\sigma_{\pi})} \delta_{\lmss{i} \cdot \tau_{1} \ldots \tau_{k},\lmss{j}}, \, \quad \forall k \geq 1.
	  \end{equation}
  The set $\lmss{T}_{+}^{k}(\sigma_{\pi})$ is defined by:
  \begin{equation*}
	  \lmss{T}_{+}^{k}(\sigma_{\pi}) = \left\{ \tau_{1},\ldots,\tau_{k} \in \lmss{T}_{p}^{\times k} : \sigma_{\pi} \circ \tau_{1} \circ \cdots \circ \tau_{k} \textrm{ has exactly } k+ \sharp \sigma_{\pi} + 1 \textrm{ cycles}\right\}.
  \end{equation*}
  In particular if the type $t(\sigma_{\pi})$ of $\sigma_{\pi}$ is $(n_{1},\cdots,n_{p})$, then
  \begin{equation*}
	  P_{\lmss{i}, \lmss{j}}^{\pi, \lmss{d}(\sigma_{\pi})} = \left( \frac{p-\sum_{i=1}^{p}n_{i}}{0!^{n_{1}}\ldots(p-1)^{n_{p}}} \prod_{i=1}^{p} i^{n_{i}(i-2)}\right) \delta_{\lmss{i}_{\sigma_{\pi}^{-1}(1)},\, \lmss{j}_{1}} \cdots\, \delta_{\lmss{i}_{\sigma^{-1}_{\pi}(p)},\, \lmss{j}_{p}}.
  \end{equation*}
\end{lemma}
Before we prove the last proposition, a simple consequence of the relation \eqref{corollary_1} is the nullity of the Taylor coefficients $P^{k}(\pi,\lmss{i},\lmss{j}$ of order $k$ larger than the geodesic distance of $\sigma_{\pi}$ to the identity in $\mathfrak{S}_{p}$:
$$\forall k \in \mathbb{N}, \quad k > \lmss{d}(\sigma_{\pi}) \Rightarrow P_{\lmss{i}, \lmss{j}}^{\pi, k} = 0.$$
We can be more precise. If $\sigma$ is a non-crossing permutation in $\mathfrak{S}_{p}$, define $\mathfrak{S}(\sigma,i,j) \subset \mathfrak{S}_{p}$ as the set comprising all permutations $\rho$ lying on a geodesic between the identity permutation and $\sigma$ and satisfying $\lmss{i}\cdot \rho=\lmss{j}$. The last proposition implies:
\begin{equation*}
  k > \max\{d(\rho), \rho \in S(\sigma_{\pi},\lmss{i},\lmss{j})\} \implies P^{\pi, k}_{\lmss{i},\lmss{j}} = 0.
\end{equation*}
\begin{proof}
Let $k\geq 1$ be an integer, let $\pi$ a non-crossing partition of size $p\geq 1$, and $(\lmss{i},\lmss{j}) \in \mathcal{C}_{2p}$ a colourization. A simple application of the recurrence relation \eqref{recurelationP1} shows that:
	\begin{equation*}
		P^{\pi,k}_{\lmss{i}, \lmss{j}} = \frac{1}{n^{k}}\sum_{\substack{\tau_{1},\ldots,\tau_{k} \in \lmss{T} \\ \tau_{i} \in \lmss{T}_{+}(\sigma_{\pi} \circ \tau_{1} \ldots \tau_{i-1})}} P^{\sigma_{\pi} \circ \tau_{1} \cdots \tau_{k},0}_{\lmss{i} \cdot \tau_{1} \ldots \tau_{k},~\lmss{j}}.
	\end{equation*}
	The sum runs over $k$-tuple $(\tau_{1},\cdots,\tau_{k})$ such that $\tau_{i} \in \lmss{T}^{+}(\sigma_{\pi} \circ \tau_{1} \cdots \circ \tau_{i-1})$ for all $i \leq k$. By definition such a $k$-tuple of transpositions belongs to $\lmss{T}^{k}_{+}(\sigma_{\pi})$. The first relation is proved.
	A minimal factorisation of a permutation $\pi$ is a tuple $(\tau_{1},\ldots,\tau_{q})$ such that $\pi = \tau_{1} \cdots \tau_{q}$. Such a tuple is of length $p-\sharp{\pi}$. A result of D\'enes (see \cite{denes1959representation}) on the number of minimal factorisations of a cycle of length $s$ assesses that there are $s^{s-2}$ such factorisations. Hence the number of minimal factorisations of a permutation $\sigma$ of type $n_{1},\cdots,n_{p}$ is
 \begin{equation}
   \label{minimalfactor}
 \frac{p-\sum_{i=1}^{p}n_{i}}{0!^{n_{1}} \cdots (p-1)^{n_{p}}}\prod_{i=1}^{p}i^{n_{i}(i-2)}
  \end{equation}
 where the factor $\frac{p-\sum_{i=1}^{p}n_{i}}{0!^{n_{1}} \cdots (p-1)!^{n_{p}}}$ accounts for the number of shuffling of a given minimal factorisations of $\sigma_{\pi}$.
\end{proof}
We use the notation $\lmss{mf}(\sigma)$ for the number of minimal factorisation of a non-crossing partition $\sigma$, see equation \eqref{minimalfactor}. Before going further into the analysis of the sum \eqref{lemmaunformula}, in Lemma \ref{lemmaun}, there are other simple consequences of Lemma \ref{lemmaun}. First, for the coefficient $P_{\lmss{i},\lmss{j}}^{k,\pi}$, the sequences $\lmss{i}$ and $\lmss{j}$ must contain the same number of different colours with same number of occurences:
\begin{equation}
	\label{corollary_1}
	\phi_{t}((\lmss{i},\lmss{j}),\pi) \neq 0 \Rightarrow \left(\forall i \in \{1,\cdots,p\}, \hspace{1mm} \sharp \{k \in \{1,\cdots,p\} : \lmss{i}_{k} = i\} = \sharp \{k \in \{1,\cdots,p\} : \lmss{i}_{k} = i\}\right)
\end{equation}
Given two sequences $\lmss{i}$ and $\lmss{j}$ in $\{1,\ldots,n\}^{p}$, it seems a rather difficult task to decide whether there exists a non-crossing permutation $\sigma$ such that $\sigma(\lmss{i})=\lmss{j}$. Related questions have drawn our interest but we did not succeed to make progress on them:
\begin{enumerate}
  \item Give sufficient conditions on the sequence $\lmss{j}$ such that there exists a non-crossing permutation $\sigma$ satisfying $\lmss{i}\cdot \sigma = \lmss{j}$,
  \item if such a permutation exists, give a way to construct all of them (or at least one), and finally
  \item compute the maximal distance and minimal distance to the identity of such permutations.
\end{enumerate}
This last relation implies in turn that for all pairs of integers $1 \leq i \neq j \leq p$ and all integers $n \geq 1$ that $\phi_{t}(\left(u_{ij}\right)^{n}) = 0$.
One should emphasize that the last relation does not imply $u_{ij} = 0$, because $u_{ij}$ is not self-adjoint and thus $\phi_{t}(u_{ij}\left(u_{ij}\right)^{\star}) \neq \phi_{t}(u_{ij}^{2})$.
Let $\sigma$ be a permutation of $\{1,\cdots n \}$. An other simple consequence of Lemma \ref{lemmaun} is independence of the state $\phi_{t}$ with respect to permutations of blocks:
\begin{equation*}
	P^{\pi, k}_{\sigma \cdot \lmss{i}, \sigma \cdot \lmss{j}} = P^{\pi, k}_{\lmss{i}, \lmss{j}},~ \forall k \geq 0,~ \sigma \in \mathfrak{S}_{p} \text{ and }\phi_{t}(\lmss{i},\lmss{j},\pi) = \phi_{t}(\sigma \cdot \lmss{i}, \sigma \cdot \lmss{j}, \pi).
\end{equation*}
This last property is related to an invariance of the non-commutative distribution of the driving noise $\noisefree$ of equation $\eqref{free_unitary_n}$. In fact, the group of permutations $\mathfrak{S}$ is injected into the group of unitary elements of $\mathcal{A} \otimes \mathcal{M}_{n}(\mathbb{C})$ by setting for the matrix $\left[ \sigma \right]$ corresponding to a permutation $\sigma$, $\left[\sigma\right]_{ij} = \delta_{i,\sigma(j)}$, $1 \leq i,j \leq n$. With this definition, because we chose for the entries of $\noisefree$ circular brownian motion with same covariance:
\begin{equation*}
\left[\sigma \right]\noisefree(t) \left[\sigma \right]^{-1} \overset{\text{nc. dist.}}{=} \noisefree(t),~\text{ for all time } t\geq 0.
\end{equation*}
We reformulate equation $\eqref{lemmaunformula}$ of Lemma \ref{lemmaun} by rewriting the right hand side as a sum over non-crossing partitions. In fact, for a pair of non-crossing partitions $\rho,\pi$ in $\lmss{NC}_{p}$, we denote by $[\rho,\ldots,\pi]$ the set of all non-crossing partitions that are greater than $\rho$ and smaller than $\pi$. It is a simple fact that
\begin{equation*}
	\gamma \in [\hat{0}_{[1,\ldots,p]}, \pi] \Leftrightarrow \exists k > 0, \exists \tau_{1} \cdots \tau_{k} \in \lmss{T}^{+}(\sigma_{\pi}) : \sigma_{\gamma} = \tau_{1}\cdots\tau_{k} \circ \sigma_{\pi},
\end{equation*}
Thus,
\begin{equation}
	\label{formuladeuxtaylor}
	P_{\lmss{i}, \lmss{j}}^{\pi, k} = \frac{1}{n^{k}} \sum_{\gamma \leq \pi} \sum_{\tau_{1},\ldots,\tau_{k} = \gamma } \delta_{\gamma \cdotp \lmss{i}, \lmss{j}} = \frac{1}{n^{k}}\sum_{\gamma \leq \pi} \lmss{mf}(\sigma_{\gamma}) \delta_{\sigma_{\gamma} \cdot \lmss{i}, \lmss{j}}, \quad \forall k \geq 1.
\end{equation}
Let $(u_{t})_{t}$ be a unitary Brownian motion in a tracial algebra $(\mathcal{A}, \tau)$. For each time $t\geq 0$, let $\nu_{t} \in \mathbb{C}[u,u^{\star}]^{\star}$ be the non commutative distribution of $u_{t}$. A formula due to Nica (see for example \cite{demni2015star}) for some of the cumulants $k_{t}^{p}, p\leq 1$ of the distribution $\nu_{t}$ is the following:
\begin{equation*}
	\label{nica_cumulant}
\kappa^{q}_{t}(u,\ldots,u) = e^{-\frac{qt}{2}}\frac{(-qt)^{q-1}}{q!}, \quad q \geq 1,
\end{equation*}
We briefly recall how this formula can be obtained. It can be shown (see \cite{biane1997free}) that:
\begin{equation*}
	\phi_{t}(u_{t}^{p}) = e^{-\frac{pt}{2}} \sum_{k = 0}^{n-1} \frac{(-t)^{k}}{k!} P_{k}.
\end{equation*}
with $P_{k} = \sharp \{\left(\tau_{1},\cdots,\tau_{k}\right) \in \lmss{T}_{p} : \tau_{1} \cdots \tau_{k} \leq (1,\cdots,n),~\sharp \pi_{\tau_{1}\ldots\tau_{k}} = p-k\}$.
According to a result of Denes, there are $\ell^{\ell-2 }$ minimal factorisations of a cycle of length $\ell$. Thus, if the partitions $p_{\pi} \in \lmss{NC}(p, p-k)$ of a non-crossing permutation define a partition $(s_{1},\cdots,s_{p})$ of the integer $p$, taking into account all the shuffles of a minimal factorisation leads to the following formula for the number \lmss{mf}$(\pi)$ of minimal factorisations of $\pi$:
This leads to the desired formula for the cumulants $\kappa^{q}_{t}, \, q \geq 1$. The next proposition is downright implication of equations \eqref{minimalfactor} and \eqref{formuladeuxtaylor}.
\begin{proposition}
	\label{cumulant}
	Let $p\geq 1$ an integer and $t\geq 0$ a time. Denote by $k_{t}^{p}$ the $p^{th}$ cumulant function of the distribution of $\freeunitary(t)$. With the notation introduced so far, if $u_{i_{1},j_{1}}\ldots,u_{i_{p},j_{p}}$ are elements of $\udualgroup$ taken amongst the generators,
\begin{equation*}
	\kappa^{p}_{t}(u_{\lmss{i}_{1}, \lmss{j}_{1}},\ldots, u_{\lmss{i}_{p}, \lmss{j}_{p}}) = e^{-\frac{pt}{2}} \left(\frac{-t}{n}\right)^{p-1} \frac{p^{p-2}}{(p-1)!} \delta_{c_{p} \cdotp \lmss{i}, \lmss{j}}.
\end{equation*}
\end{proposition}

\section{Coloured Brauer diagrams and Schur--Weyl dualities}
\label{schur_weyl}

\subsection{Coloured Brauer diagrams}
\par In that section, we introduce the set of coloured Brauer diagrams and the algebra they generate. We take the opportunity to make a brief reminder on Brauer diagrams, see \cite{barcelo1999combinatorial} for a detailed review on these combinatorial objects. We strive to motivate all definitions that are introduced. However, we are aware that the combinatoric developed here may seem to be quite raw but is absolutely fundamental for our work.
Let $k\geq 1$ an integer. We use the notation $i^{\prime} = 2k + i$ for $1 \leq i \leq k$ and denote by $\left\{1,\ldots,k,1^{\prime},\ldots,k^{\prime}\right\}$ the interval of integers $\llbracket 1,2k \rrbracket$.
\begin{definition}
A Brauer diagram of size $k$ is a fixed point free involution of the set of cardinal $2k$, $\{1,\ldots,k,1^{\prime},\ldots,k^{\prime}\}$.
\end{definition}
We denote by $\ncc{\Brauer}$ the set of all Brauer diagrams. (The superscript $\bullet$ is used to make clear the difference between Brauer diagrams and the notion of coloured Brauer diagram we introduce below). A Brauer diagram of size $k$ may alternatively be seen as a partition of the set $\{1,\ldots,k,1^{\prime},\ldots,k^{\prime}\}$: two integers $i,j$ are related if and only if $i = \sigma(j)$. Such a partition associated with a Brauer diagram is depicted as follows. We draw first $k$ vertices on a line labelled by the integers in ${1^{\prime},\ldots,k^{\prime}}$ from left to right and $k$ other vertices on an other line under the first one and labelled by the integers in $\{1,\ldots,k \}$. We add strands that connect two integers if one is the other image by the Brauer diagram (in the same block for the associated partition).
In the sequel, we make the identification without mentioning it between a Brauer diagram, a partition which has blocks of cardinal two, and the picture that depicts it.
We perform operations on Brauer diagrams which are naturally defined on the set $\mathcal{P}_{k}$ of all partitions of the set $\{1,\ldots,k,1^{\prime},\ldots,k^{\prime}\}$. These operations are the following ones and are related to the lattice structure of the set $\mathcal{P}_{k}$. Let $p_{1}$ and $p_{2}$ be two partitions. We write $p_{1} \prec p_{2}$ and say that $p_{1}$ is a refinement of $p_{2}$ if each block of $p_{1}$ is included in a block of $p_{2}$. We denote by $p_{1} \vee p_{2}$ the smallest partition which is greater than $p_{1}$ and $p_{2}$:
\begin{equation}
p_{1} \vee p_{2} = \cup_{V\in p_{1}} \cup_{W \in p_{2} : V \cap W \neq \emptyset} V \cup W.
\end{equation}
The greatest partition which is smaller than $p_{1}$ and $p_{2}$ is denoted $p_{1}\wedge p_{2}$:
\begin{equation*}
p_{1} \wedge p_{2} = \cup_{V\in p_{1}, W\in p_{2}} V \cap W.
\end{equation*}
The block number of a partition $p$ is denoted $\lmss{nc}(p)$. Of course, the function $\lmss{nc}$ is constant on the set $\brauerz$ and equal to $k$.
We denote by the symbol $\mathbf{1}_{k}$ the Brauer diagram that is pictured as in Fig. \ref{identity}.
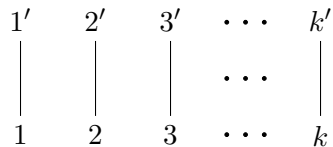
\begin{figure}[!htb]
	\centering
	\resizebox{0.3\linewidth}{!}{
\begin{tikzpicture}[scale = 1.0]
	\draw[black] (-3,0) -- (-3,-1);
	\draw[black] (-2,0) -- (-2,-1);
	\draw[black] (-1,-1) -- (-1,0);

	\node[inner sep=0.5pt,draw,circle,fill,black] at (-0.25,-0.5){};
	\node[inner sep=0.5pt,draw,circle,fill,black] at (0,-0.5){};
	\node[inner sep=0.5pt,draw,circle,fill,black] at (0.25,-0.5){};

	\node[inner sep=0.5pt,draw,circle,fill,black] at (-0.25,-1.25){};
	\node[inner sep=0.5pt,draw,circle,fill,black] at (0,-1.25){};
	\node[inner sep=0.5pt,draw,circle,fill,black] at (0.25,-1.25){};

	\node[inner sep=0.5pt,draw,circle,fill,black] at (-0.25,0.25){};
	\node[inner sep=0.5pt,draw,circle,fill,black] at (0,0.25){};
	\node[inner sep=0.5pt,draw,circle,fill,black] at (0.25,0.25){};

	\draw[black] (1,0) -- (1,-1);

   \node  [above] at ( -3,0) {$1^{\prime}$};
   \node  [above] at ( -2,0) {$2^{\prime}$};
   \node  [above] at ( -1,0) {$3^{\prime}$};
   \node  [above] at ( 1,0) {$k^{\prime}$};

   \node  [below] at ( -3,-1) {$1$};
   \node  [below] at ( -2,-1) {$2$};
   \node  [below] at ( -1,-1) {$3$};
   \node  [below] at ( 1,-1) {$k$};

\end{tikzpicture}
}
\caption{\label{identity}\small The identity element of the algebra of coloured Brauer diagram}
\end{figure}
A $\emph{cycle}$ of a Brauer diagram $\ncc{b}$ is a block of the partition $\ncc{b}\vee \mathbf{1}$.
\par Let $n\geq 1$ an integer and let $\lmss{d}=(d_{1},\ldots,d_{n})$ a finite sequence of positive integers.
\par A \emph{colouring} of $\{1,\ldots,k,1^{\prime},\ldots,k^{\prime}\}$ is a function $c : \{1,\ldots,k,1^{\prime},\ldots,k^{\prime}\} \rightarrow \llbracket 1, n\rrbracket$. We use the symbol $\mathcal{C}_{2k}^{n}$ to denote the set of colourings. The \textit{dimension} function $\lmss{d}_{c}$ associated with a colouration $c$ and the finite sequence $\lmss{d}$ is defined by:
\begin{equation*}
	\begin{array}{lcll}
		\lmss{d}_{c}: &\left\{1,\ldots,k,1^{\prime},\ldots,k^{\prime}\right\} & \longrightarrow & \{d_{1},\ldots,d_{n}\} \\
		&i &\mapsto & d_{c_{i}}
	\end{array}
	.
\end{equation*}
We define the object of interest for this section. For the rest of this section, we fix a dimension function $\lmss{d}$.
\begin{definition}[Coloured Brauer diagrams]
A \textit{coloured Brauer diagram} is a pair $\left(\ncc{b},c\right)$ with $\ncc{b}$ a Brauer diagram and $c$ a colouring which dimension function $\lmss{d}_{c}$ satisfies $\ncc{b} \circ d_{c} = d_{c}$.
\end{definition}
\par A coloured Brauer diagram is conveniently depicted as in Fig. \ref{bicolourbrauer}. The set of all coloured Brauer diagrams is denoted by $\brauerd$. If there is no risk of misunderstanding, we drop the superscript $\lmss{d}$, that indicates the dependence of the set of Brauer diagrams toward the sequence $\lmss{d}$. This sequence is also named \emph{dimension function} in the following.
The set of coloured Brauer diagram $\brauerd$ depends solely on the partition $\lmss{Ker}(\lmss{d})$ of $\llbracket 1,n\rrbracket$ that is the set of all level sets of $\lmss{d}$. We would thus write for a partition $\pi$ of $\llbracket 1,n \rrbracket$ $\mathcal{B}_{k}^{\pi}$ for the set of Brauer diagrams which links are coloured by two integers in the same blocks of $\pi$. If $\pi^{\prime}$ is a partition finer than $\pi$ then $\mathcal{B}_{k}^{\pi^{\prime}} \subset \mathcal{B}_{k}^{\pi}$.

\par Let $\mathbb{K}$ the field of real numbers or the field of complex numbers. We use $\mathbb{K}\left[\Brauer\right]$ for the $\mathbb{K}$ vector space with basis $\brauerd$. If $b$ is a coloured Brauer diagram, $\ncc{b}$ stands for its underlying Brauer diagram and $c_{b}$ is the colouring.
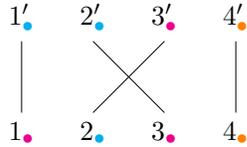
\begin{figure}[!h]
	\centering
	\resizebox{0.23\linewidth}{!}{
\begin{tikzpicture}[scale = 1.0]
	\draw[black] (-2,0) -- (-1,-1);
	\draw[black] (-2,-1) -- (-1,0);
	\draw[black] (-3,0) -- (-3,-1);
	\draw[black] (-0,0) -- (-0,-1);

   \node  [above] at ( -3,0) {$1^{\prime}_{\bcyan}$};
   \node  [above] at ( -2,0) {$2^{\prime}_{\bcyan}$};
   \node  [above] at ( -1,0) {$3^{\prime}_{\bmage}$};
   \node  [above] at ( -0,0) {$4^{\prime}_{\boran}$};

   \node  [below] at ( -3,-1) {$1_{\bmage}$};
   \node  [below] at ( -2,-1) {$2_{\bcyan}$};
   \node  [below] at ( -1,-1) {$3_{\bmage}$};
   \node  [below] at ( -0,-1) {$4_{\boran}$};

\end{tikzpicture}
}
\caption{\small A coloured Brauer diagram with $n=3$, and $d_{\color{cyan} \bullet} = d_{\color{magenta} \bullet} \neq d_{\color{orange} \bullet}$.
\label{bicolourbrauer}}
\end{figure}

\par We define on the vector space $\mathbb{R}\left[ \brauerd \right]$ an algebra structure. Let $b_{1},b_{2} \in \Brauer$ two coloured Brauer diagrams.

\par We begin with reviewing the definition of the composition law on the real span of $\brauerz$. Let $\ncc{b}_{1}$ and $\ncc{b}_{2}$ be two Brauer diagrams. We stack $\ncc{b}_{1}$ over $\ncc{b_{2}}$ to obtain a third diagram that may contain closed connected components. If so, we remove these components to obtain the concatenation $\ncc{b}_{1}\circ \ncc{b}_{2}$ of $\ncc{b}_{1}$ and $\ncc{b}_{2}$. Let $\mathcal{K}(\ncc{b}_{1},\ncc{b}_{2})$ be the number of components that were removed. The product $\ncc{b}_{1}\ncc{b}_{2}$ is defined by the formula:
\begin{equation*}
\ncc{b}_{1} \ncc{b}_{2} = n^{\mathcal{K}(\ncc{b}_{1}, \ncc{b}_{2})}\ncc{b}_{1}\circ \ncc{b}_{2}.
\end{equation*}
The unit is the Brauer diagram $1_{k} = \{\{i,i^{\prime}, i\leq k \}$.
To define the product of $b_{1}$ and $b_{2}$, we define first the composition $b_{1}\circ b_{2}$. We stack $b_{1}$ over $b_{2}$ to obtain a diagram $c$ which contains, eventually, closed connected components. The diagram $c$ contains links that may be coloured with two different colours; if it happens we set the Brauer diagram $b_{1} \circ b_{2}$ to $0$. Otherwise, $b_{1} \circ b_{2}$ is the diagram $c$ with the closed components removed. If $b_{1} \circ b_{1} \neq 0$, for each $d \in \{\lmss{d}\}$, we let $\mathcal{K}_{d}(b_{1},b_{2})$ be the number of closed connected components coloured with an integer $1 \leq i\leq n$ such that $d_{i}=d$ that were removed of $c$ to obtain $b_{1} \circ b_{2}$. Finally, The product $b_{1}b_{2}$ is defined by the formula
\begin{equation*}
	b_{1}b_{2} = \prod_{d\in\lmss{d}}d^{\mathcal{K}_{d}(b_{1},b_{2})} b_{1} \circ b_{2}.
\end{equation*}
Endowed with this composition law, $\mathbb{R}\left[\brauerd\right]$ is an associative complex unital algebra with unit
\begin{equation}
e=\sum_{\substack{c \in \mathcal{C}_{2k}^{n}\\c(i)=c(i^{\prime})}} (\mathbf{1}_{k}, (c,c^{\prime})).
\end{equation}
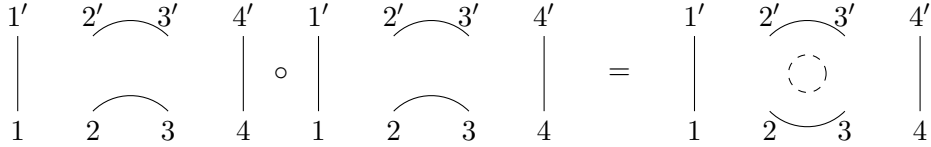
\begin{figure}[!h]
\resizebox{.8\textwidth}{!}{
\begin{tikzpicture}
\begin{scope}
\draw[black] (-3,0) -- (-3,-1);
\draw[black]  (-2,0)  to [bend left = 45](-1,0);
\draw[black] (-2,-1) to [bend left = 45](-1,-1);
\draw[black] (0,0) -- (0,-1);
\node  [above] at ( -3,0) {$1^{\prime}$};
\node  [above] at ( -2,0) {$2^{\prime}$};
\node  [above] at ( -1,0) {$3^{\prime}$};
\node  [above] at ( -0,0) {$4^{\prime}$};
\node  [below] at ( -3,-1) {$1$};
\node  [below] at ( -2,-1) {$2$};
\node  [below] at ( -1,-1) {$3$};
\node  [below] at ( -0,-1) {$4$};
\node at (0.5,-0.5) {$\circ$};
\end{scope}
\begin{scope}[shift={(4,0)}]
\draw[black] (-3,0) -- (-3,-1);
\draw[black] (-2,0) to   [bend left = 45] (-1,0);
\draw[black] (-2,-1) to  [bend left=45](-1,-1);
\draw[black] (0,0) -- (0,-1);
\node  [above] at ( -3,0) {$1^{\prime}$};
\node  [above] at ( -2,0) {$2^{\prime}$};
\node  [above] at ( -1,0) {$3^{\prime}$};
\node  [above] at ( -0,0) {$4^{\prime}$};
\node  [below] at ( -3,-1) {$1$};
\node  [below] at ( -2,-1) {$2$};
\node  [below] at ( -1,-1) {$3$};
\node  [below] at ( -0,-1) {$4$};
\node at (1,-0.5) {=};
\end{scope}
\begin{scope}[shift = {(9,0)}]
\draw[black] (-3,0) -- (-3,-1);
\draw[black] (-2,0) to  [bend left = 45] (-1,0);
\draw[black] (-2,-1) to  [bend right=45](-1,-1);
\draw[dashed]  (-1.5,-0.5) circle (0.25cm);
\draw[black] (0,0) -- (0,-1);
\node  [above] at ( -3,0) {$1^{\prime}$};
\node  [above] at ( -2,0) {$2^{\prime}$};
\node  [above] at ( -1,0) {$3^{\prime}$};
\node  [above] at ( -0,0) {$4^{\prime}$};
\node  [below] at ( -3,-1) {$1$};
\node  [below] at ( -2,-1) {$2$};
\node  [below] at ( -1,-1) {$3$};
\node  [below] at ( -0,-1) {$4$};
\end{scope}
\end{tikzpicture}
}
\caption{Concatenation of two Brauer diagrams.}
\end{figure}
The product we just defined on $\mathbb{R}\left[\brauerd\right]$ is relevant for studying distribution of square blocks extracted from an unitary Brownian motion in high dimension. However, in Section \ref{rectangularextractions} we consider the more general problem of rectangular extractions. To tackle this question, we need a central extension of $\mathcal{R}\left[\brauerd\right]$ that allows us to track the loops, and the dimension of their colourings that are possibly created if two Brauer diagrams are multiplied together. In short, this central extension is constructed by considering diagrams that may have closed connected components.
We should make an intensive use of the following fundamental relation, which is easily proved by a drawing:
\begin{equation}
\lmss{nc}((\ncc{b}_{1}\circ\ncc{b}_{2}) \vee 1) = \lmss{nc}(b_{1}\vee b_{2}) + \mathcal{K}(\ncc{b}_{1},\ncc{b}_{2}).
\end{equation}

If $\ncc{b}$ is a non-coloured Brauer diagram, $C(b)$ is the set of all colourizations of $\ncc{b}$ so as to $(\ncc{b},c)$ is a coloured Brauer diagram which each block is coloured with only one integer. We end this section by defining an injection of the algebra of non-coloured Brauer diagram $\brauer\left(\sum_{i=1}^{n}\lmss{d}(i)\right)$ into the algebra of coloured Brauer diagram $\brauer(\lmss{d})$ that will be used, without mentioning it, in computations,
\begin{equation*}
\begin{array}{cccc}
\Delta:&\brauer(\sum_{i=1}^{n}\lmss{d}(i)) &\to&\brauer(\lmss{d}) \\
& \ncc{b} & \mapsto & \displaystyle\sum_{c \in C(b)}(\ncc{b},c).
\end{array}
\end{equation*}
\subsection{Representation}
\par Let $k\geq 1$ an integer, if $\lmss{i} = \left(i_{j}\right)_{1 \leq j \leq k}$ is a $k$-tuple of integers, we denote by $\lmss{ker}(\lmss{i})$ the partition of $\{1,\ldots,k\}$ equal to the set of all level sets of \lmss{i}. Also, if $\lmss{i}, \lmss{j}$ are two integer sequences of length $k$, $\lmss{ker}(\lmss{i},\lmss{j})$ is the partition equal to the set of all level sets of the function defined on $\{1,\ldots,k,1^{\prime},\ldots,k^{\prime}\}$ equal to $\lmss{i}$ on $\{1,\ldots,k\}$ and $\lmss{j}$ on $\{1^{\prime},\ldots,k^{\prime}\}$.
\par Let $N\geq 1$ an integer. A representation $\ncc{\rho}_{N}$ of the algebra $\brauerz(N)$ is defined by setting:
\begin{equation*}
		\begin{array}{llll}
			\ncc{\rho}_{N}:&\brauerz(N) & \rightarrow & \lmss{End}(\mathbb{R}^{N}) \\
			&\ncc{b} & \rightarrow & \underset{\substack{\lmss{i},\lmss{j} \in \{1,\ldots,p+q\}^{k}, \\ \lmss{ker}(\lmss{i},\lmss{j}) \geq b}}{\sum} E_{i_{1},j_{1}} \otimes \cdots \otimes E_{i_{k},j_{k}}.
		\end{array}
\end{equation*}
If $N \geq 1$ is sufficiently large, it can be shown that $\ncc{\rho}_{N}$ is injective.
\par Let $n\geq 1$ an integer and $\lmss{d}=(d_{1},\ldots,d_{n})$ a sequence of positive integers of length $n$, set $N=d_{1}+\ldots+d_{n}$. A representation $\rho_{\lmss{d}}$ of the algebra $\mathbb{R}\left[\brauerd\right]$ on the $k$-fold tensor product $\left(\mathbb{R}^{N}\right)^{\otimes k}$ is defined by setting:
\begin{equation*}
		\begin{array}{llll}
			\rho_{\lmss{d}}:&\Brauer(\lmss{d}) & \rightarrow & \lmss{End}(\mathbb{R}^{N}) \\
			&\genbrauerz & \mapsto & \underset{\substack{\lmss{i},\lmss{j} \in \{1,\ldots,p+q\}^{k}, \\ \lmss{ker}(\lmss{i},\lmss{j}) \geq b}}{\sum} E^{c(1^{\prime}),c(1)}_{i_{1},j_{1}} \otimes \cdots \otimes E^{c(k^{\prime}),c(k)}_{i_{k},j_{k}}
		\end{array}.
\end{equation*}
With the definition of the injection $\Delta$ we gave in the previous section, simple computations show that $\rho_{\lmss{d}}\circ\Delta=\ncc{\rho}_{N}$ if $N=\sum_{i=1}^{n}d_{i}$.
\par We turn our attention to the definition of three real representations, $\rho^{\mathbb{R}}_{\lmss{d}},\rho^{\mathbb{C}}_{\lmss{d}}$ and $\rho^{\mathbb{H}}_{\lmss{d}}$ that will be used later to define statistics of the unitary Brownian motions.
For the real and complex case, we set $\rho^{\mathbb{R}}_{\lmss{d}} = \rho^{\mathbb{C}}_{\lmss{d}} = \rho_{\lmss{d}}$.
\par The representation $\rho^{\mathbb{C}}_{p,q}$ of the real algebra $\Brauer(\lmss{d})$ defines a representation, denoted by the same symbol, of the complex algebra $\Brauer(\lmss{d})\otimes\mathbb{C}$.
A real linear representation $\rho^{\mathbb{H}}$ of the algebra of Brauer diagrams $\brauerz(-2N)$ on $\left(\mathbb{H}^{n}\right)^{\otimes k}$ is defined in \cite{levy}, equation $(36)$ as a convolution of two representations: the representation $\ncc{\rho}_{N}$ of $\brauerz(N)$ and a representation $\gamma$ of $\brauerz(-2)$ which commutes with $\ncc{\rho}_{N}$.
Let us explain how a representation of the coloured Brauer algebra $\Brauer(-2\lmss{d})$ is defined similarly as a convolution product of two representations.

\par Let $m$ be the multiplication map of endomorphisms in $\lmss{End}(\mathbb{H}^{n})^{\otimes k})$ and let $s,t > 0$ be two positive real numbers. The key observation is the existence of a morphism $\Delta_{st}^{s,t} : \brauerz(st) \rightarrow \brauerz(s) \times \brauerz(t)$ which is the real linear extension of a function defined on the set of Brauer diagrams $\brauerz$ by $\Delta_{st}^{s,t}(b)= b \otimes b$, $b \in \brauerz$.
\par The representation $\rho^{\mathbb{H}}_{N}$ of the algebra $\brauerz(-2N)$ defined by L\'evy in \cite{levy} is the convolution product:
\begin{equation*}
\rho^{\mathbb{H}}_{N} = m \circ \left( \rho^{\mathbb{R}} \otimes \gamma \right) \circ \Delta_{-2N}^{N,-2}.
\end{equation*}
 The definition of a coloured version of the representation $\rho_{N}^{\mathbb{H}}$ is ensured by the existence of coloured version $ \Delta_{-2\lmss{d}}^{\lmss{d},-2}:\Brauer(-2\lmss{d}) \rightarrow \Brauer(\lmss{d}) \times \brauerz(-2)$ of $ \Delta^{N,-2}_{-2N}$, namely, for $b \in \Brauer$:
\begin{equation*}
	\Delta^{\lmss{d},-2}_{-2\lmss{d}}(b) = b \otimes \ncc{b}.
\end{equation*}
There are no difficulties in checking that the map $\Delta^{\lmss{d},-2}_{-2\lmss{d}}$ is a morphism from $\brauer(-2\lmss{d})$ to $\brauer(\lmss{d}) \otimes \brauerz(-2) $.
Finally, the representation $\rho^{\mathbb{H}}_{p,q}$ is defined by the equation:
\begin{equation}
\rho^{\mathbb{H}}_{p,q} = m \circ \left(\rho_{p,q}^{\mathbb{R}} \otimes \gamma \right) \circ \Delta_{-2(p+q)}^{-2p,-2q}.
\end{equation}
\subsection{Orienting and cutting a Brauer diagram}
\par Let $b=\genbrauerz$ a coloured Brauer diagram.
\par To the partition $\ncc{b}$ we associate a graph $\Gamma_{b}$: the vertices are the points $\{1,\ldots,k,1^{\prime},\ldots,k^{\prime}\}$ and the edges are the links of the partition $\ncc{b}$ together with the vertical edges $\{x,x^{\prime}\}$, $x\leq k$. Each of the connected components of this graph is a loop and we pick an orientation of these loops. To that orientation of $\Gamma_{b}$, we associate a function $s : \{1,\ldots,k \} \rightarrow \{-1,1\}$ defined as follows. Let $i \in \{1,\ldots,k\}$ an integer, we set $s(i) = 1$ if the edge that belongs to $b$ which contains $i$ is incoming at $i$ in the chosen orientation of $\Gamma_{b}$ and $-1$ otherwise. Of course an orientation of $\Gamma_{b} $ is completely known through its associated sign function $s$, thus we will in the sequel freely identify these last two objects.
\par We use the notation $b^{s} = (b,s)$ for an oriented Brauer diagram with sign function $s$ and the set of oriented Brauer diagrams is denoted $\mathcal{O}\brauerd$. To each oriented Brauer diagram $(b,s)$ there are two associated permutations $\Sigma_{(b,s)}$ and $\sigma_{(b,s)}$ defined as follows. An oriented Brauer diagram $(b,s)$ is naturally a permutation $\Sigma_{(b,s)}$ of the set $\left\{1,\ldots,k,1^{\prime},\ldots,k^{\prime}\right\}$. The cycles of the permutation $\sigma_{(b,s)}$ are the traces on $\{1,\ldots,k\}$ of the cycles of $\Sigma_{(b,s)}$.
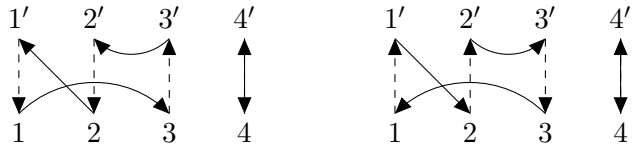
\begin{figure}[h!]
\begin{tikzpicture}
\usetikzlibrary{arrows}
\begin{scope}
	\draw[black,-triangle 45] (-2,-1) -- (-3,-0);
	\draw[black,-triangle 45]  (-1,0)  to [bend left = 45](-2,0);
	\draw[black,-triangle 45] (-3,-1) to [bend left = 45](-1,-1);
	\draw[black] (0,0) -- (0,-1);

	\draw[black, dashed,-triangle 45] (-3,0) --  (-3,-1);
	\draw[black, dashed,-triangle 45]  (-2,0)  -- (-2,-1);
	\draw[black,dashed,-triangle 45] (-1,-1) -- (-1,0);
	\draw[black,dashed,-triangle 45] (0,0) -- (0,-1);
	\draw[black,dashed,-triangle 45] (0,-1) -- (0,0);

   \node  [above] at ( -3,0) {$1^{\prime}$};
   \node  [above] at ( -2,0) {$2^{\prime}$};
   \node  [above] at ( -1,0) {$3^{\prime}$};
   \node  [above] at ( -0,0) {$4^{\prime}$};

   \node  [below] at ( -3,-1) {$1$};
   \node  [below] at ( -2,-1) {$2$};
   \node  [below] at ( -1,-1) {$3$};
   \node  [below] at ( -0,-1) {$4$};

\end{scope}

\begin{scope}[shift={(5,0)}]
	\draw[black,-triangle 45] (-3,-0) -- (-2,-1);
	\draw[black,-triangle 45]  (-2,0)  to [bend right = 45](-1,0);
	\draw[black,-triangle 45] (-1,-1) to [bend right = 45](-3,-1);
	\draw[black] (0,-1) -- (0,0);

	\draw[black, dashed,-triangle 45] (-3,-1) --  (-3,0);
	\draw[black, dashed,-triangle 45]  (-2,-1)  -- (-2,0);
	\draw[black,dashed,-triangle 45] (-1,0) -- (-1,-1);
	\draw[black,dashed,-triangle 45] (0,-1) -- (0,0);
	\draw[black,dashed,-triangle 45] (0,0) -- (0,-1);

   \node  [above] at ( -3,0) {$1^{\prime}$};
   \node  [above] at ( -2,0) {$2^{\prime}$};
   \node  [above] at ( -1,0) {$3^{\prime}$};
   \node  [above] at ( -0,0) {$4^{\prime}$};

   \node  [below] at ( -3,-1) {$1$};
   \node  [below] at ( -2,-1) {$2$};
   \node  [below] at ( -1,-1) {$3$};
   \node  [below] at ( -0,-1) {$4$};

\end{scope}
\end{tikzpicture}
\caption{\label{exorienteddiagram} \small Two orientations of the same Brauer diagram.}
\end{figure}

\par We denote by $\mathcal{O}\brauer$ the set of oriented uncoloured Brauer diagrams. We were not able to endow the real vector space with basis $\mathcal{O}\brauerd$ with an algebra structure that would turn the canonical projection from $\mathbb{R}\left[\mathcal{O}\brauerd\right]$ into a morphism. Also, we introduce Brauer algebras and related for two reasons : to represent quantities that are of interest for us in Section \ref{rectangularextractions} and to define operators that will ease computations. As we should see, these operators act on the Brauer component of a oriented Brauer diagram by multiplication, hence we need, somehow, to associate to an unoriented coloured Brauer diagram and to an oriented Brauer diagram a third Brauer diagram. There is no canonical way in doing that. We may just simply pick a section  $\brauerd \to \mathcal{O}\brauerd$ and use it to give orientation to a Brauer diagram if needed.
Such a section can be defined, for example, by choosing for the orientation of a diagram the sign function that is equal to \emph{one} on the minimum of each cycle.
We prefer, given an oriented Brauer $(b_{1},s)$ diagram and a Brauer diagram $b$, to define a orientation of $b\circ b_{1}$ in the following way. We choose for $s_{b_{1}\circ b_{2}}$ the sign function that is equal to $s_{b_{1}}$ on the minimum of each cycle of $b_{1}\circ b$.
We denote by $b \diamond b_{1}$ the oriented Brauer diagram obtained in this way.
\par The subset of permutations $\mathcal{S}_{k} \subset \Brauer$ is defined as the subset of Brauer diagrams that represent a permutation. In symbols, the Brauer diagram $b_{\sigma}$ associated with a permutation $\sigma$ is equal to $\{\{i,\sigma(i)^{\prime}\},~i\leq k\}$. It is easily seen that a Brauer diagram $b$ is a permutation diagram if and only if any orientation of $b$ is constant on the cycles, meaning that the vertical edges of $\Sigma_{b}$ belonging to the same cycle have the same orientation.
We denote by $\ncc{s}_{b}$ the orientation of a Brauer diagram that is equal to one on the minimum of each cycle of $b$.
\par For $i \in \{1,\ldots,k,1^{\prime},\ldots,k^{\prime}\}$, denote by $i^{k}$ the integer $i+k$ if $i \leq k$ and $i-k$ if $i >k$. The \textit{transposition} of a diagram $b = \left(\ncc{b}, c_{b}\right)$
is the diagram $b^{t} = \left(\ncc{{b^{t}}}, c_{b^{t}} \right)$ defined by the equation
\begin{equation*}
\ncc{b^{t}} = \cup_{l \in \ncc{b}} \{i^{\star}, i \in l\},~c_{b^{t}}(t) = c(i \mathrm{~mod~} k).
\end{equation*}
See Figure \ref{transpositiondiagram}.
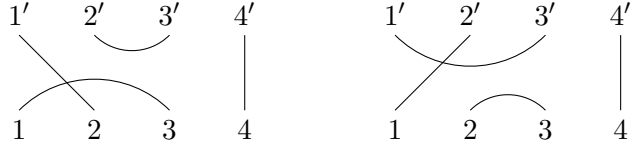
\begin{figure}[!h]
\resizebox{0.55\linewidth}{!}{
\begin{tikzpicture}
\begin{scope}
	\draw[black,] (-2,-1) -- (-3,-0);
	\draw[black,]  (-1,0)  to [bend left = 45](-2,0);
	\draw[black,] (-3,-1) to [bend left = 45](-1,-1);
	\draw[black] (0,0) -- (0,-1);

   \node  [above] at ( -3,0) {$1^{\prime}$};
   \node  [above] at ( -2,0) {$2^{\prime}$};
   \node  [above] at ( -1,0) {$3^{\prime}$};
   \node  [above] at ( -0,0) {$4^{\prime}$};

   \node  [below] at ( -3,-1) {$1$};
   \node  [below] at ( -2,-1) {$2$};
   \node  [below] at ( -1,-1) {$3$};
   \node  [below] at ( -0,-1) {$4$};
\end{scope}
\begin{scope}[shift={(5,0)} ]
	\draw[black,] (-2,0) -- (-3,-1);
	\draw[black,]  (-1,-1)  to [bend right = 45](-2,-1);
	\draw[black,] (-3,0) to [bend right = 45](-1,0);
	\draw[black] (0,-1) -- (0,0);

   \node  [above] at ( -3,0) {$1^{\prime}$};
   \node  [above] at ( -2,0) {$2^{\prime}$};
   \node  [above] at ( -1,0) {$3^{\prime}$};
   \node  [above] at ( -0,0) {$4^{\prime}$};

   \node  [below] at ( -3,-1) {$1$};
   \node  [below] at ( -2,-1) {$2$};
   \node  [below] at ( -1,-1) {$3$};
   \node  [below] at ( -0,-1) {$4$};
\end{scope}
\end{tikzpicture}
}
\caption{Transposition of a non-coloured Brauer diagram.}
\label{transpositiondiagram}
\end{figure}

We define now the twist operators $\twist_{i},~ i \leq k$ that act on (non-coloured) partitions of $\{1,\ldots,k,1^{\prime}$ $,\ldots,k^{\prime}\}$. The set of (non-coloured partitions) of $\{1,\ldots,k,1^{\prime},\ldots,k^{\prime}\}$ is denoted $\mathcal{P}_{k}$.
\begin{definition}[Twist operators]
\label{deftwist}
Let $i \leq k$. The twist operator $\twist_{i}:\mathcal{P}_{k} \to \mathcal{P}_{k}$ is the complex linear extension to $\mathbb{C}\left[\mathcal{P}_{k}\right]$ of the set function defined on $\mathcal{P}_{k}$ by the equation
\begin{equation*}
\twist_{i}(p) = \cup_{l \in p} \{i^{\star},j, j\in l\backslash {i} \}
\end{equation*}
\end{definition}
To put it in words, the twist operator $\twist_{i}$ exchanges the integer $i$ and $i^{\star}$ in their own blocks. The subset of Brauer diagrams is stable by the Twists operators $\twist_{i},~ i\leq k$.
\par Let us recall that the set $\mathcal{P}_{k}$ is a lattice. The minimum $p_{1} \wedge p_{2}$ of two partitions $p_{1}$ and $p_{2}$ is the partition which blocks are the intersection of the blocks of $p_{1}$ and $p_{2}$. The maximum $p_{1} \vee p_{1}$ of two partitions is the partition which blocks are union of blocks of $p_{1}$ and $p_{2}$ that have a non empty intersection.
\begin{lemma}
\label{twistmorphism}
The twist operator is a morphism of the lattice $\left(\mathcal{P}_{k}, \wedge, \vee \right)$. In addition, $\lmss{nc}(p) = \lmss{nc}(\twist(p)), p\in \mathcal{P}_{k}$. In particular, the number of cycles of a Brauer diagram is preserved by twisting.
\end{lemma}
\begin{proof}
A simple drawing of the diagrams does the proof. Let's nevertheless do the proof. Let $p_{1},p_{2}$ two partitions. Let $S$ be a block of $\ncc{\lmss{Tw}}_{i}(p) \vee \ncc{\lmss{Tw}}_{i}(q)$. The set $S$ enjoys the maximality property:
\begin{equation}
\label{maximal1}
U \in \ncc{\lmss{Tw}}_{i}(p_{1}) \cup \ncc{\lmss{Tw}}_{i}(p_{2}),~ U \cap S \neq \emptyset \Rightarrow U \subset S.
\end{equation}
\par Define the set $\tilde{S}$ by
\begin{itemize}
\item $\tilde{S} = S$ if $i, i^{\prime} \in S$ or $i,i^{\prime} \not\in S$,
\item $\tilde{S} = S \backslash \{i\} \cup \{i^{\prime}\}$ if $i \in S, i^{\prime} \not\in S$,
\item $\tilde{S} = S \backslash \{i^{\prime}\} \cup \{i\}$ if $i^{\prime} \in S, i \not\in S $.
\end{itemize}
In order to prove that $S$ is a block of the partition $\ncc{\lmss{Tw}}_{i}(p_{1} \vee p_{2})$, we prove that $\tilde{S}$ enjoys the maximality property:
\begin{equation}
\label{maximal}
U \in p_{1} \cup p_{2},~ U \cap \tilde{S} \neq \emptyset \Rightarrow U \subset \tilde{S}.
\end{equation}
\par Assume that $S$ does not contain nor $i$ nor $i^{\prime}$ then $S = \tilde{S}$, $S$ is an union of blocks of $p_{1}$ and $p_{2}$, and further $S \in \ncc{\lmss{Tw}}_{i}(p_{1}\vee p_{2})$.
\par Assume that $i \in S,~i^{\prime} \not\in S$. Then $\tilde{S} = S \backslash \{i\} \cup \{i^{\prime} \}$. Let $U \in p_{1} \cup p_{2}$, $U\cap S \neq \emptyset$. We distinguish four cases:
\begin{itemize}
\item $i\in U, i^{\prime} \not\in U$, then $U \backslash \{i\} \cup \{i^{\prime}\} \cap S \neq \emptyset$, $U \backslash \{i\} \cup \{i^{\prime}\} \in \ncc{\lmss{Tw}}_{i}(p_{1}) \cup \ncc{\lmss{Tw}}_{i}(p_{2})$ implies  $U \backslash \{i\} \cup \{i^{\prime}\} = S$ and then $\tilde{S} = U$.
\item $i^{\prime} \in U,~ i \in U $ then $U\in \ncc{\lmss{Tw}}_{i}(p_{1}) \cup \ncc{\lmss{Tw}}_{i}(p_{2})$ thus $\tilde{S} = U$.
\item $i\not\in U, i^{\prime}\not\in U$ then $U\in \ncc{\lmss{Tw}}_{i}(p_{1}) \cup \ncc{\lmss{Tw}}_{i}(p_{2})$ thus $\tilde{S} = U$.
\item $i \in U,~i^{\prime} \in U$, then $U \backslash \{i^{\prime}\} \cup \{i\} \cap S \neq \emptyset$, $U \backslash \{i^{\prime}\} \cup \{i\} \in \ncc{\lmss{Tw}}_{i}(p_{1}) \cup \ncc{\lmss{Tw}}_{i}(p_{2})$ implies  $U \backslash \{i^{\prime}\} \cup \{i\} = S$ and then $\tilde{S} = U$.
\end{itemize}
We conclude that $\tilde{S}$ has the maximal property \eqref{maximal}, moreover $\tilde{S}$ is an union of blocks of $p_{1}$ and $p_{2}$, it follows that $S\in p_{1} \vee p_{2}$.
\par Assume now that $i \in S$, $i^{\prime} \in S$. Let $U \in p_{1} \cup p_{2}$ such that $U \cap S \neq \emptyset$. We prove that $S$ has the maximality property $\eqref{maximal}$. If $U \in \ncc{\lmss{Tw}}(p_{1}) \cup \ncc{\lmss{Tw}}_{i}(p_{2})$, then $U = S$ so let us assume that $ U \not\in \ncc{\lmss{Tw}}(p_{1}) \cup \ncc{\lmss{Tw}}_{i}(p_{2})$. Either $i$, either $i^{\prime}$ belongs to $U$ but not both, we can make the hypothesis that $i \in U$ and $i^{\prime} \not\in U$. We have $U \backslash \{i\} \cup \{i^{\prime}\} \cap S \neq \emptyset$ and $U\backslash \{i\} \cup \{i^{\prime}\} \in \ncc{\lmss{Tw}}_{i}(p_{1}) \cup \ncc{\lmss{Tw}}_{i}(p_{2})$. It follows that $U\backslash \{i\} \cup \{i^{\prime}\} = S$ and $U = S$. Since $S$ is an union of sets in $p_{1}$ and $p_{2}$, one has $S \in p_{1}\vee p_{2}$.
\end{proof}
\begin{figure}[!h]
\resizebox{0.85\linewidth}{!}{
\begin{tikzpicture}
\begin{scope}
	\draw[black,] (-2,-1) -- (-3,-0);
	\draw[black,]  (-1,0)  to [bend left = 45](-2,0);
	\draw[black,] (-3,-1) to [bend left = 45](-1,-1);
	\draw[black] (0,0) -- (0,-1);

   \node  [above] at ( -3,0) {$1^{\prime}$};
   \node  [above] at ( -2,0) {$2^{\prime}$};
   \node  [above] at ( -1,0) {$3^{\prime}$};
   \node  [above] at ( -0,0) {$4^{\prime}$};

   \node  [below] at ( -3,-1) {$1$};
   \node  [below] at ( -2,-1) {$2$};
   \node  [below] at ( -1,-1) {$3$};
   \node  [below] at ( -0,-1) {$4$};
\end{scope}
\begin{scope}[shift={(5,0)} ]
		\draw[black,] (-2,-1) to [bend right = 45] (-3,-1);
	\draw[black,]  (-1,0)  to [bend left = 45](-2,0);
	\draw[black,] (-3,0) to [](-1,-1);
	\draw[black] (0,0) -- (0,-1);

   \node  [above] at ( -3,0) {$1^{\prime}$};
   \node  [above] at ( -2,0) {$2^{\prime}$};
   \node  [above] at ( -1,0) {$3^{\prime}$};
   \node  [above] at ( -0,0) {$4^{\prime}$};

   \node  [below] at ( -3,-1) {$1$};
   \node  [below] at ( -2,-1) {$2$};
   \node  [below] at ( -1,-1) {$3$};
   \node  [below] at ( -0,-1) {$4$};
\end{scope}

\begin{scope}[shift={(10,0)}]
	\draw[black,] (-2,-1) -- (-3,-0);
	\draw[black,]  (-1,-1)  to [](-2,0);
	\draw[black,] (-3,-1) to [](-1,0);
	\draw[black] (0,0) -- (0,-1);

   \node  [above] at ( -3,0) {$1^{\prime}$};
   \node  [above] at ( -2,0) {$2^{\prime}$};
   \node  [above] at ( -1,0) {$3^{\prime}$};
   \node  [above] at ( -0,0) {$4^{\prime}$};

   \node  [below] at ( -3,-1) {$1$};
   \node  [below] at ( -2,-1) {$2$};
   \node  [below] at ( -1,-1) {$3$};
   \node  [below] at ( -0,-1) {$4$};
\end{scope}
\end{tikzpicture}
}
\caption{The second diagram is the twist at $1$ of the first one. The third diagram is the twist at $2$ of the first one.}
\end{figure}
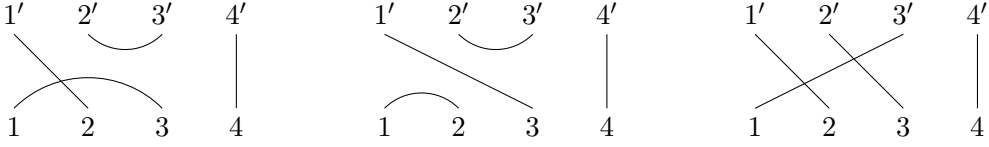

\begin{lemma}
\label{twistorientation}
Let $i\leq k$ an integer. Let $(b,s)$ be an oriented Brauer diagram. For any orientation $u$ of the diagram $\twist_{i}(\ncc{b})$, we have
\begin{equation*}
u(i)u(j) = -u(i)u(j),~ i\neq j,~ i \sim_{\ncc{b}\vee 1},~\textrm{and}~u(i)u(j) = -u(i)u(j),i\neq j,~ i \not\sim_{\ncc{b}\vee 1}j
\end{equation*}
\end{lemma}
\begin{proof}
We use the notations introduced in Lemma \ref{twistorientation}. For any $k,l \leq k$ integers the products $u(k)u(l)$ does not depend on the orientation we pick to orient the twist at $i$ of the diagram $b$.
Define an orientation $u$ of $\twist(\ncc{b})$ by setting
\begin{equation*}
u(j) = s(j) ~\textrm{if}~ j \neq i,~\textrm{and}~ u(i) = -s(i).
\end{equation*}
To prove that $u$ does indeed define an orientation of $\ncc{b}$, a simple drawing is, once again, sufficient. In fact, twisting at $i$ the diagram $\ncc{b}$ reverses the orientation of the vertical edge that connects the vertices $i$, and $i+k$.

\end{proof}

In the previous proof, we defined an orientation of the twist of a non-coloured Brauer diagram given an orientation of the diagram. This suggests that the twist operator can be lifted to the set of non-coloured oriented Brauer diagrams.

Also, extending the twist operator to the set of coloured Brauer diagrams is straightforward; if a coloured Brauer diagram is twisted at a site $i$, the colours the integers $i$ and $i^{\prime}$ are coloured with are exchanged.
In the following lemma we denote by $\langle \ncc{\lmss{Tw}}_{i},~ i \leq k \rangle$ the algebra generated by the twist operators $\{\twist_{i},~ i\leq k\}$. For an ordered set $\{i_{1} \leq i_{2} \leq \ldots \leq i_{q}\}$, we use sometimes the notation $\twist_{S} = \twist_{s_{1}} \cdots \twist_{s_{q}}$.
\begin{lemma}
\label{twistorbit}
Let $b\in\brauerz$ be an irreducible Brauer diagram ($\lmss{nc}(b))=1$). There are exactly two permutation diagrams in the orbit $\{\twist(b),~\twist \in \langle \twist_{i},~i\leq k\rangle \}$. In addition, two permutations in the orbit $\{\twist(b),~\twist \in \langle \twist_{i},~i\leq k\rangle \}$ are related by transposition.
\end{lemma}
\begin{proof}
There are exactly two possible orientations for an irreducible Brauer diagram. We pick one and denote it by $s$ (the other is $-s$). Define $\twist_{s} = \displaystyle\prod_{i : s(i) = -1} \twist_{i}$. As in the proof of Lemma \ref{twistorientation}, an orientation $u$ of $\twist_{s}(\ncc{b})$ is defined by setting
\begin{equation*}
u(i) = 1 ~\textrm{if}~ s(i) = -1~\textrm{and}~u(i) = 1 ~\textrm{if}~ s(i) = 1.
\end{equation*}
Hence, $u = 1$ and the diagram $\twist_{s}(\ncc{b})$ is a permutation diagram. Let $\sigma$ an other permutation diagram in the orbit of $\ncc{b}$. There exists a set $S \subset \{1,\ldots,k\}$ such that $\twist_{s}(\ncc{b}) = \left(\prod_{i \in S}\twist_{i}\right)(\sigma)$.
Once again, the diagram $\big(\displaystyle\prod_{i \in S}\twist_{i}\big)(\sigma)$ is oriented by mean of the sign function $v$:
\begin{equation*}
v(i) = -1 ~\textrm{if}~ i \in S,~ v(i) = 1 ~\textrm{if}~ i \not\in S.
\end{equation*}
Unless $S = \{1,\ldots,k\}$, $v$ is not constant. This achieves the proof.
\end{proof}
We defined the twist operators to prove the following Proposition, which is needed in Section \ref{squareextractions} to prove that the free unitary Brownian of dimension $n$ motion is the limit of the process extracting square blocks of a Brownian unitary matrix..
We recall the following formula which is used extensively in Section \ref{squareextractions} and is needed for the proof of the forthcoming proposition:
\begin{equation}
\label{fondform}
\lmss{nc}(\ncc{b}_{1}\circ\ncc{b}_{2} \vee 1) + \mathcal{K}(\ncc{b}_{1},\ncc{b}_{2}) = \lmss{nc}(\ncc{b}_{1} \vee \ncc{b}_{2}).
\end{equation}
\begin{proposition}
\label{projtrmult}
Let $\ncc{b} \in \brauerz$ an irreducible Brauer diagram and $\ncc{e}$ a projector. Then $\lmss{nc}(\ncc{e}\vee \ncc{b})-\lmss{nc}(\ncc{b} \vee 1) \in \{0,1\}$ and $\lmss{nc}(\ncc{e}\vee \ncc{b}) = \lmss{nc}(\ncc{b}\vee 1) + 1$ if and only if $s(i)s(j) = -1$ for any orientation $s$ of $\ncc{b}$.
\end{proposition}
\begin{proof}
First, pick $T \in \langle \ncc{\lmss{Tw}}_{i},~i\leq k \rangle$ such that $T(b)$ is a permutation diagram (see Lemma \ref{twistorbit}). Let $i,j \leq k$ be two integers. Then $\lmss{nc}(\ncc{e}_{ij} \vee b) = \lmss{nc}(T(\ncc{e}_{ij}\vee b)) = \lmss{nc}(T(\ncc{e}_{ij}) \vee T(b))$.
The diagram $T(\ncc{e}_{ij})$ is equal either to the projector $\ncc{e}_{ij}$ if $s(i)s(j) = 1$ either to the transposition $\tau_{ij}$ if $s(i)s(j)=-1$. It is easily checked that no loops nor cycles are created if multiplying an irreducible permutation diagram by a projector, thus from equation \ref{fondform} $\lmss{nc}(\ncc{b}\vee\ncc{e}_{ij}) =  \lmss{nc}(\ncc{b})$ if $s(i)s(j) = 1$. If $s(i)s(j)=-1$, we multiply by a transposition an irreducible permutation (a permutation with only one cycle). A direct calculation shows that $\lmss{nc}\left(\tau_{ij}T(\ncc{b})\right)= \lmss{nc}\left(T(\ncc{b})\right)+1$, hence $\lmss{nc}(\ncc{b} \vee \ncc{e}_{ij}) - \lmss{nc}(\ncc{b}\vee 1) = 1$.
\end{proof}
\begin{proposition}
\label{projtrmultblocks}
Let $k\geq 1$. Denote by $c$ the cycle $(1,\ldots,k)$. Let $S \subset \{1,\ldots,k\}$ a set of integers. Define the Brauer diagram $\ncc{b}$ as $\ncc{b} = \displaystyle\prod_{s \in S}\left(\ncc{\lmss{Tw}}_{s}\right)(c)$. Let $i\neq j$ be two integers, then
\begin{itemize}
\item if $i\in S,j\in S$,~ $(\ncc{\tau}_{ij}\circ \ncc{b}\vee 1 = \{\{1,\ldots,i-1,j,\ldots,k\},\{i,i+1,\ldots,j-1\}\}$,
\item if $i\not\in S,j\not \in S$,~ $(\ncc{\tau}_{ij}\circ \ncc{b})\vee 1 = \{\{1,\ldots,i,j+1,\ldots,k\},\{i+1,i+1,\ldots,j\}\}.$
\end{itemize}
In addition, for any orientation $s$ of $\ncc{b}$ and $u$ of $\tau_{ij}\circ b$, we have
$u(x)u(y) = s(x)s(y)$, $1 \leq x,y \leq k$,~ $x \sim_{\tau_{ij}\circ b} y$.

\begin{itemize}
\item if $i\in S,j \not\in S$,~ $(\ncc{e}_{ij}\circ \ncc{b})\vee 1 = \{\{1,\ldots,i-1,j+1,\ldots,k\},\{i,i+2,\ldots,j\}\}$,
\item if $i\not\in S,j\in S$,~ $(\ncc{e}_{ij}\circ \ncc{b})\vee 1 = \{\{1,\ldots,i-1,j,\ldots,k\},\{i+1,\ldots,j-1\}\}.$
\end{itemize}
In addition, for any orientation $s$ of $\ncc{b}$ and $u$ of $e_{ij}\circ b$, we have
$u(x)u(y) = s(x)s(y)$, $1 \leq x,y \leq k$,~ $x \sim_{e_{ij}\circ b} y$.
\end{proposition}
\begin{proof}
Let $S$ and $b$ be as in Proposition \ref{projtrmult}. As shown in Lemma \ref{twistorientation}, an orientation $s$ of $b$ is defined by setting
\begin{equation*}
s(i) = -1~\textrm{if}~ i \in S,~s(i) = 1 ~\textrm{if}~ i \not\in S.
\end{equation*}
Let $i,j \leq k$ be two integers. Assume that $i\in S$ and $j \in S$.
Easy computations show that
\begin{equation*}
(\tau_{ij} \circ b)(i^{\prime})=b(j),~(\tau_{ij}\circ b )(j') = b(i),~ (\tau_{ij}\circ b) (k) = b(k), k\neq i^{\prime},j^{\prime}.
\end{equation*}
By using Proposition \ref{projtrmult}, we prove that $\lmss{nc}(\tau_{ij} \circ b \vee 1) = 2$. Recall that $b$ is involution of $\{1,\ldots,k,1^{\prime},\ldots,k^{\prime}\}$, for each $i\in \{1,\ldots,k,1^{\prime},\ldots,k^{\prime}\}$, $b(i)$ is the integer that lies in the same block of $b$ as $i$. Recall that $\star(x)$ denotes $x^{\prime}$ if $x\leq k$ and $i-k$ if $x > k$. As shown, the partition $\tau_{ij} \circ b \vee 1$ has two blocks and if $x$ is in one of this block, so is $x^{\star}$.
The block that contains $\{i,i^{\prime}\}$ is equal to the set of alternate products of $b\circ\tau_{ij}$ and $\star$ applied to $i$:
\begin{equation*}
\{i,\star(i),\left((b\circ\tau_{ij})\circ \star\right)(i), \left(\star \circ (b\circ\tau_{ij})\circ \star\right)(i), \left((b\circ \tau_{ij}) \circ \star \circ (b\circ\tau_{ij})\circ \star\right)(i),... \}
\end{equation*}
which is equal to $\{i,i^{\prime},b(j),\star(b(j)),(b\circ \tau_{ij})(\star(b(j)),... \}$. We have $\star(b(j))[k] = j-1$ (from the definition of $b$). Thus if $i \neq j-1$, we have $(\tau_{ij}\circ b)(\star(b(j))) = b(\star(b(j)))$. Continuing in the same manner, we find
\begin{equation*}
\{i,i^{\prime},b(j),\star(b(j)), b(\star(b(j)), \star(b(\star(b(j))),\star(b(\star(b(j)))...i, (\star){i} \} = \{i,i^{\prime},j-1,j^{\prime}-1,j-2,j^{\prime}-2,...,i,i^{\prime}\}.
\end{equation*}
We do the same for the case $i\in S$ and $j \not\in S$, the details are left to the reader.
Assume now that $i \in S$ and $j \not\in S$. The partition $e_{ij}\circ b \vee 1$ has two blocs (this follows from Proposition \ref{projtrmult}).
Since $(e_{ij}\circ b)(i^{\prime}) = j^{\prime}$, the set $\{i^{\prime},j^{\prime}\}$ is contained within a block of $(e_{ij}\circ b)$. Once again to compute the blocks that contains the set $\{i^{\prime},j^{\prime}\}$, we have to compute the set
\begin{equation}
\label{set2}
\{i^{\prime},i,(e_{ij}\circ b)(i), (\star \circ (e_{ij}\circ b))(i),((e_{ij}\circ b) \circ \star \circ (e_{ij}\circ b))(i), \ldots,\}
\end{equation}
We remark that for any integer $x$ in the interval $\llbracket i,j-1 \rrbracket$, $\{(e_{ij} \circ b)(x), \star((e_{ij} \circ b)(x)\} = \{x+1,x^{\prime}+1\}$. We have $(e_{ij}\circ b)(i) [k] = i+1$, thus we find the set \ref{set2} is equal to:
\begin{equation*}
\{i^{\prime},i,i+1, i^{\prime}+1,i^{\prime}+2,i+2,\ldots,... j-1,j-1,j,j^{\prime}\}
 \end{equation*}
The case $i\not\in S$ and $j \in S$ is left to the reader.
\end{proof}
Conjugation of a Brauer diagram by a permutation $\alpha$ results in a Brauer diagram that has the same number of cycles, and
\begin{equation*}
\alpha \circ b \circ \alpha^{-1} = \bigcup_{l\in b} \{\alpha(i),\alpha(j), i,j \in l\}.
\end{equation*}
Hence, $\alpha\circ e_{ij} \circ \alpha^{-1} = e_{\alpha(i),\alpha(j)}$ and $\alpha\circ\tau_{ij}\circ\alpha^{-1} = \tau_{\alpha(i),\alpha(j)}$ with $i,j \leq k$ two integers.
In addition, orientation and conjugation enjoy a remarkable property. For any oriented Brauer diagram $(b,s_{b})$ and permutation $\alpha$, the sign function  $s_{\alpha\circ b \circ \alpha^{-1}} = s_{b} \circ \alpha^{-1}$ defines orientation of $\alpha\circ b \circ \alpha^{-1}$.
\par The twists operators are also equivariant with respect to conjugation action:
\begin{equation*}
\twist_{S} \left(\alpha \circ b \circ \alpha^{-1} \right) = \alpha \circ \twist_{\alpha^{-1}(S)}(b)\circ \alpha^{-1},~ S \subset \{1,\ldots,k\},~\twist_{S} = \prod_{s\in S}\twist_{s},~\alpha \in \mathcal{S}_{k}.
\end{equation*}
Let $S\subset \{1,\ldots,k\}$. Let $\alpha$ a permutation in $\mathcal{S}_{k}$.
The proposition is easily generalised to twists of the cycle $\alpha \circ c \circ \alpha^{-1} = (\alpha(1),\ldots,\alpha(k))$. Define $b=\twist_{S}(\left(\alpha(1),\ldots,\alpha(k) \right))$.
In fact, $(\alpha(1),\ldots,\alpha(k))$ = $ \alpha \circ (1,\ldots,k) \circ \alpha^{-1}$ and
\begin{equation*}
\left(e_{ij} \circ \twist_{S}(\alpha \circ c \circ \alpha^{-1})\right) \vee 1 = \alpha \circ \left(\left(e_{\alpha^{-1}(i),\alpha^{-1}(j)}\circ\twist_{\alpha^{-1}(S)}(c)\right)\vee 1\right)\circ\alpha^{-1}
\end{equation*}
We apply Proposition \ref{projtrmultblocks} to find
\begin{itemize}
\item if $i\in S,j\in S$,~$
(\tau_{ij}\circ b)\vee 1 = \{\{\alpha(1),\ldots,\alpha(\alpha^{-1}(i)-1),j,\alpha(\alpha^{-1}(j+1),\ldots,\alpha(k)\},$ \\ $
\{\alpha^{-1}(i),\alpha(\alpha^{-1}(i)+1),\ldots,\alpha(\alpha^{-1}(j)-1)\}\}$,
\item if $i\not\in S,j\not \in S$,~
$(\tau_{ij}\circ b)\vee 1 = \{\{\alpha(1),\ldots,\alpha(\alpha^{-1}(i)-1),\alpha(\alpha^{-1}(j)+1),\ldots,\alpha(k)\},$
 \\
$\{\alpha(\alpha^{-1}(i)),\alpha(\alpha^{-1}(i)+1),\ldots,\alpha(\alpha^{-1}(j)-1), j\}\}.$
\end{itemize}
In addition, or any orientation $s$ of $\ncc{b}$ and $u$ of $\tau_{ij}\circ b$, we have
$u(x)u(y) = s(x)s(y)$, $1 \leq x,y \leq k$,~ $x \sim_{\tau_{ij}\circ b} y$.
\begin{itemize}
\item if $i\in S,j \not\in S$,~ $
(e_{ij} \circ b) \vee 1 = \{ \{\alpha(1),\ldots, \alpha(\alpha^{-1}(i)-1),\alpha(\alpha^{-1}(j)+1,\ldots,\alpha(k)\}$, \\ $
\{i,\alpha(\alpha^{-1}(i)+1),\ldots,j\}\}$,
\item if $i\not\in S,j\in S$,~ $
(e_{ij}\circ b) \vee 1 = \{\{\alpha(1),\ldots,\alpha(\alpha^{-1}(i)-1),j,\alpha(j+1),\ldots,\alpha(k)\},$\\$\{\alpha(\alpha^{-1}(i)+1),\ldots,\alpha(\alpha^{-1}(j)-1)\}\}.$
\end{itemize}
In the same way, for any orientation $s$ of $\ncc{b}$ and $u$ of $e_{ij}\circ b$, we have
$u(x)u(y) = s(x)s(y)$, $1 \leq x,y \leq k$,~ $x \sim_{e_{ij}\circ b} y$.

\par We have seen in Proposition \ref{projtrmult} that multiplication of an irreducible Brauer diagram by a transposition or a projector produces at most one cycle. Given a non-necessarily irreducible Brauer diagram $\ncc{b}$, the following proposition specifies how many cycles are deleted or created if we multiply $\ncc{b}$ by a transposition or a projector.
\begin{proposition}
\label{mainpropbrauer}
Let $b$ be a non-coloured Brauer diagram. Let $i\neq j \in \{1,\ldots,k\}$ two integers. If $i,j$ do not lie in the same cycle of $\ncc{b}$ ($i\not\sim_{\ncc{b}\vee 1}j$) then $\lmss{nc}(\ncc{e}_{ij}\vee\ncc{b}) = \lmss{nc}(e_{ij}\vee 1)-1$.
If $i$ and $j$ are in the same cycle of $\ncc{b}$, we have for any orientation $s$ of $\ncc{b}$:
\begin{equation}
\label{mainpropequation}
\begin{split}
&\textrm{if}~s(i)s(j) = -1,~\lmss{nc}(\ncc{b}\vee \ncc{e}_{ij}) = \lmss{nc}(\ncc{b}\vee1)+1, \\
&\textrm{if}~s(i)s(j) = 1,~ \lmss{nc}(\ncc{b} \vee \ncc{\tau}_{ij}) = \lmss{nc}(\ncc{b}\vee 1)+1.
\end{split}
\end{equation}
\end{proposition}

\subsection{Central extension of the algebra of coloured Brauer diagrams}
As it will appear in Section \ref{rectangularextractions}, it will be necessary to keep track of the colouration of the loops that are created if two Brauer diagrams are multiplied together. In fact, as of now, it is not possible to do so: from the definition of the algebra structure on $\mathcal{B}_{k}(\lmss{d})$, a loop that is produced by multiplication of two diagrams multiply by a positive scalar the concatenation of the two diagrams. It there are least two loops that are created, it not possible to find back the colourizations from this multiplication factor. That is the reason why we introduce a central extension.
\par The central extension $\cbrauer(\lmss{d})$ is, as a vector space, equal to the direct sum of vector spaces $\mathbb{R}\left[\brauerd\right]\oplus \mathbb{R}\left[\{\lmss{o}_{d},d \in \{\lmss{d}\}\} \right]$. The set $\lmss{o} = \{\lmss{o}_{d},d \in \{\lmss{d}\}\}$ of commuting variables is referred to as the set of \emph{loops variables} or \emph{ghost variables}. Two elements $b \oplus P(\lmss{o})$ and $b^{\prime}\oplus Q(\lmss{o})$ in $\brauerd \oplus \mathbb{R}\left[\lmss{o} \right]$ are multiplied as follows:
\begin{equation*}
(b \oplus P(\lmss{o})) \cdot (b^{\prime}\oplus Q(\lmss{o})) = \left(b \circ b^{\prime}, PQ \times \prod_{d} \lmss{o}_{d}^{\mathcal{K}_{d}(b,b^{\prime})}\right).
\end{equation*}
We indexed the loops variable by the set $\{\lmss{d}\}$, we could have equivalently indexed it by the blocks of the partition $\lmss{ker}(\lmss{d})$. In the the sequel, we will mainly deal with operators that are defined on a subalgebra of $ \cbrauer(\lmss{d})$. If $(\alpha_{i})_{i\in\lmss{d}}$ is a multi-index, we denote by $\lmss{o}^{\alpha}_{\{\lmss{d}\}}$ the monomial $\lmss{o}_{d_{1}}^{\alpha_{d_{1}}} \cdots \lmss{o}_{d_{p}}^{\alpha_{d_{p}}}$ if $\{\lmss{d}\} = \{d_{1},\ldots,d_{p}\}$. We set $\cbrauerd = \{b \oplus \lmss{o}_{\{\lmss{d}\}}^{\alpha},~ b \in \brauerd, \alpha \in \mathbb{N}^{\{\lmss{d}\}}\}$. An element of $\cbrauerd$ is named a diagram with loops and is pictured as in Fig. \ref{exextendeddiagrams}.
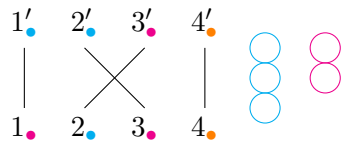
\begin{figure}[!h]
\begin{tikzpicture}[scale = 0.8]
	\draw[black] (-2,0) -- (-1,-1);
	\draw[black] (-2,-1) -- (-1,0);
	\draw[black] (-3,0) -- (-3,-1);
	\draw[black] (-0,0) -- (-0,-1);

   \node  [above] at ( -3,0) {$1^{\prime}_{\bcyan}$};
   \node  [above] at ( -2,0) {$2^{\prime}_{\bcyan}$};
   \node  [above] at ( -1,0) {$3^{\prime}_{\bmage}$};
   \node  [above] at ( -0,0) {$4^{\prime}_{\boran}$};

   \node  [below] at ( -3,-1) {$1_{\bmage}$};
   \node  [below] at ( -2,-1) {$2_{\bcyan}$};
   \node  [below] at ( -1,-1) {$3_{\bmage}$};
   \node  [below] at ( -0,-1) {$4_{\boran}$};

\draw[cyan] (1,-1) circle [radius=0.25cm];
\draw[cyan] (1,-0.5) circle [radius=0.25cm];
\draw[cyan] (1,0) circle [radius=0.25cm];

\draw[magenta] (2,-0.5) circle [radius=0.25cm];
\draw[magenta] (2,0) circle [radius=0.25cm];
\end{tikzpicture}
\caption{\label{exextendeddiagrams} \small A diagram with loops.}
\end{figure}
Let $\pi$ a partition of $\llbracket 1,n \rrbracket$ that is greater than $\lmss{ker}(\lmss{d})$. As mentioned, the set $\brauerd$ of coloured Brauer diagram is injected into the set $\mathcal{B}_{k}^{\pi}$. There is no injection of the vector space
$\mathbb{R}\left[\brauerd\right]\oplus \mathbb{R}\left[\{\lmss{o}_{d},d \in \{\lmss{d}\}\} \right]$ into the vector space $\mathbb{R}\left[\mathcal{B}_{k}^{\pi}\right]\oplus \mathbb{R}\left[\{\lmss{o}_{V},V \in \{\pi\}\} \right]$. The only canonical map from
$\mathbb{R}\left[\{\lmss{o}_{d},d \in \{\lmss{d}\}\} \right]$ to $\mathbb{R}\left[\{\lmss{o}_{V},V \in \{\pi\}\} \right]$ is the projection induced by the change of variable $\lmss{o}_{V} \to \lmss{o}_{W}$, $V \in \lmss{ker}(\lmss{d})$, $W \in \pi$  with $V \subset W$.

The space $\cbrauerd$ projects onto the algebra of coloured Brauer diagrams. The projection $\overset{\circ}{\pi}:\cbrauer(\lmss{d}) \to \brauer(\lmss{d})$ specializes a loop variable $\lmss{o}_{d}$ to the corresponding dimension $d$:
\begin{equation*}
\pi((b,P(\lmss{o}_{d},d\in \{d\}))) = P(d, d\in\{d\})b.
\end{equation*}
We draw in Fig. \ref{exactseq} the short exact sequence.\begin{figure}[!h]
\begin{tikzcd}
0 \arrow{r} & \mathbb{R}\left[\cbrauerd\right] \arrow["\pi"]{r} & \mathbb{R}\left[\brauerd\right] \arrow{r} & 0.
\end{tikzcd}
\caption{\label{exactseq}\small A central extension of the algebra of coloured Brauer diagrams.}
\end{figure}

If $b \in \brauerd$ is a Brauer diagram, we denote by $\overset{\circ}{b}$ the element $(b,1 )$ in the $\cbrauerd$. In this way we define a section from $\brauerd$ to $\cbrauerd$ which is not a algebra morphism.
\par We finish with the definition of the functions that justify alone the introduction of this central extension. In the last section, we gave orientation to Brauer diagrams. We will do the same for Brauer diagrams with loops in a consistent way. We recall that we denote by $\mathcal{O}\brauer(\lmss{d})$ the vector space with basis the set $\mathcal{O}\brauerd$ of all oriented coloured Brauer diagrams. We denote by $\mathcal{O}\cbrauer(\lmss{d})$ the vector space $\mathcal{O}\brauer(\lmss{d}) \oplus \mathbb{R}\left[\lmss{o}_{d}, d\in\{\lmss{d}\} \right]$.
\par Let $d \in \{\lmss{d}\}$ and $\tilde{b} = ((b,s),\prod_{d \in \{d_{N}\}}\lmss{o}_{d}^{n_{d}}) \in \mathcal{O}\cbrauer(\lmss{d})$, the function $\lmss{fnc}_{d}$ counts the number of loops variables $\lmss{o}_{d}$ and the number of cycles of $b$ whose minimum $m$ or minimum $m^{\prime}$ is coloured with the dimension $d$, depending on the orientation of the cycle. More formally:
\begin{equation*}
\begin{split}
\lmss{fnc}_{d}\left(\tilde{b}\right) = n_{d} + \sharp&\left\{\{i_{1} < \ldots <i_{k}\} \in \ncc{b}\vee 1 : d_{c_{b}(i_{1})} = d \text{ and } s(m)=1\right\} \\ &\hspace{2cm}+ \sharp\left\{\{i_{1} < \ldots <i_{k}\} \in \ncc{b}\vee 1 : d_{c_{b}(i^{\prime}_{1})} = d \text{ and } s(m)=-1\right\}.
\end{split}
\end{equation*}
In the last section, given an oriented Brauer diagram $(b_{1},s)$ and a Brauer diagram $b$, we defined the oriented Brauer diagram $b\diamond (b_{1},s)$. This operation can be lifted to $\mathcal{O}\cbrauer(\lmss{d})$:
\begin{equation*}
((b,P)\diamond ((b_{1},s),Q) = \left(b \diamond (b_{1},s), PQ\times \prod_{d\in\{d\}}\lmss{o}_{d}^{\mathcal{K}_{d}(b,b_{1})}\right),~ (b,P)\in \cbrauer(\lmss{d}),~((b_{1},s),Q) \in \mathcal{O}\cbrauer(\lmss{d}).
\end{equation*}
\subsection{Special subsets of coloured Brauer diagrams}
We introduce subsets of Brauer diagrams that will be used in Sections \ref{squareextractions} and \ref{rectangularextractions} to express generators of differential systems satisfied by statistics of the unitary Brownians motions.

\par The first of these sets is the set of non-mixing Brauer diagrams, which we denote $\mathcal{B}_{k,n}$, that is defined as a set of coloured Brauer diagrams which blocks are coloured by a single integer $1 \leq i \leq n$.
\par We denote by $\Delta_{k,n}$ the subset of $\mathcal{B}_{k,n}$ of diagonally coloured Brauer diagrams: diagrams $b$ that have a colouring $c_{b}$ satisfying $c_{b}(i)=c_{b}(i^{\prime})$ for all integers $1 \leq i \leq n$.
\par Let $i,j \leq k-1$. The projector $\ncc{e}_{ij} \in \brauerz$ and the transposition $\ncc{\tau}_{ij}$ are non-coloured Brauer diagrams that are defined by
\begin{equation}
\begin{gathered}
\ncc{e}_{ij} = \left\{ \left\{i,j \right\}, \left\{i^{\prime},j^{\prime}\right\}\} \cup \{\left\{x,x^{\prime}\right\},~ x \neq i,j,~ 1 \leq x \leq k \right\} \\
\ncc{\tau}_{ij} = \left\{ \left\{i,j^{\prime}\right\},\left\{j,i^{\prime}\right\}\} \cup \{\left\{x,x^{\prime}\right\},~x \neq i,j,~1 \leq x \leq k\right\}.
\end{gathered}
\end{equation}
The set of non-coloured Brauer diagrams comprising all non-coloured transpositions, respectively all non-coloured projectors, is denoted $\nccc{\tr}$ and $\nccc{\pr}$.
The subset of coloured Brauer diagrams in $\brauerd$ which non-coloured component is a projector, respectively a transposition, is denoted $\prd$, respectively $\trd$. The set of non-mixing transpositions, respectively projectors is denoted $\lmss{T}_{k,n}$, respectively $\lmss{W}_{k,n}$. We will frequently drop the subscript $n$ if the numbers of colours is clear from the context or the subscript $\lmss{d}$ if the partition is clear from the context. Elements of the set $\tr\cup\pr$ are called \emph{elementary diagrams}.

\par We also define the sets of \textit{exclusive transpositions} $\trexcd$ and the set of \textit{exclusive projectors} $\prexcd$ by setting
\begin{equation*}
	\trexcd = \{(\ncc{\tau}_{ij}, c_{\tau_{ij}}) \in \lmss{T}_{k}: c(i) \neq c(j) \},~ \prexcd = \{ (\ncc{e}_{ij},c_{e_{ij}})\in \lmss{W}_{k} : c(i) \neq c(i^{\prime}) \}.
\end{equation*}
The sets of \textit{diagonal transpositions} and \textit{diagonal projectors} are defined by
\begin{equation*}
	\trdiagd = \{(\ncc{\tau}_{ij}, c) \in \trd : c(i) = c(j) \},~ \prdiagd = \{ (\ncc{e}_{ij},c) \in \prd: c(i) = c(i^{\prime}) \}.
\end{equation*}
In Fig. \ref{colourexc}, we draw examples of elements of the subsets defined above.
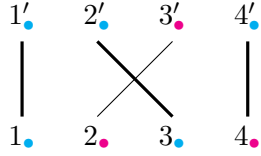
\begin{figure}[!h]
\centering
\begin{tikzpicture}[scale = 1.0]
\begin{scope}
	\draw[black,very thick] (-2,0) -- (-1,-1);
	\draw[black] (-2,-1) -- (-1,0);
	\draw[black,very thick] (-3,0) -- (-3,-1);
	\draw[black,very thick] (0,0) -- (0,-1);

   \node  [above] at ( -3,0) {$1^{\prime}_{\bcyan}$};
   \node  [above] at ( -2,0) {$2^{\prime}_{\bcyan}$};
   \node  [above] at ( -1,0) {$3^{\prime}_{\bmage}$};
   \node  [above] at ( -0,0) {$4^{\prime}_{\bcyan}$};

   \node  [below] at ( -3,-1) {$1_{\bcyan}$};
   \node  [below] at ( -2,-1) {$2_{\bmage}$};
   \node  [below] at ( -1,-1) {$3_{\bcyan}$};
   \node  [below] at ( -0,-1) {$4_{\bmage}$};

\end{scope}
\end{tikzpicture}
\caption{\label{colourexc} \small Example of an exclusive transposition, $n=2$.}
\end{figure}
\begin{figure}[!h]
\begin{tikzpicture}
\begin{scope}[shift={(8.5,0)}]
	\draw[black] (-2,0) -- (-1,-1);
	\draw[black] (-2,-1) -- (-1,0);
	\draw[black] (-3,0) -- (-3,-1);
	\draw[black] (0,0) -- (0,-1);
   \node  [above] at ( -3,0) {$1^{\prime}_{\bcyan}$};
   \node  [above] at ( -2,0) {$2^{\prime}_{\bmage}$};
   \node  [above] at ( -1,0) {$3^{\prime}_{\bmage}$};
   \node  [above] at ( -0,0) {$4^{\prime}_{\bmage}$};
   \node  [below] at ( -3,-1) {$1_{\bcyan}$};
   \node  [below] at ( -2,-1) {$2_{\bmage}$};
   \node  [below] at ( -1,-1) {$3_{\bmage}$};
   \node  [below] at ( -0,-1) {$4_{\bmage}$};
\end{scope}
\end{tikzpicture}
\caption{\label{colourdiag} \small Example of a diagonal transposition, $n=2$.}
\end{figure}
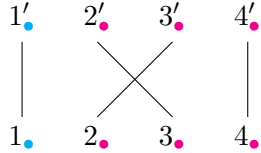

Let $\ncc{b}$ be a non coloured Brauer diagram. The sets $\trpnc{\ncc{b}}$ and $\prpnc{\ncc{b}}$ of elementary non-coloured Brauer diagrams that create a cycle if concatenated with $b$ are defined as
\begin{equation*}
	\begin{split}
		&\trpnc{\ncc{b}} = \{\ncc{\tau} \in \nccc{\tr}:~\lmss{nc}(\ncc{b}\vee\ncc{\tau}) = \lmss{nc}(\ncc{b}\vee 1)+1\}, \\
		&\prpnc{\ncc{b}} = \{\ncc{e} \in \nccc{\pr}:~\lmss{nc}(\ncc{b}\vee\ncc{e}) = \lmss{nc}(\ncc{b}\vee 1)+1\}.
	\end{split}
\end{equation*}
We end this section by defining four more subsets. Let $b \in \Brauer$, we define the subsets
\begin{equation*}
	\begin{split}
		&\trexcp{b} = \{\tau \in \trexc: \tau b \neq 0,~\lmss{nc}(\ncc{b}\vee\ncc{\tau}) = \lmss{nc}(\ncc{b}\vee 1)+1 \}, \\
		&\prexcp{b} = \{ e \in \prexc : eb \neq 0,~\lmss{nc}(\ncc{b}\vee \ncc{e}) = \lmss{nc}(\ncc{b}\vee 1)+1  \}.
	\end{split}
\end{equation*}
and
\begin{equation*}
	\begin{split}
		&\trdiagp{b} = \{\tau \in \trdiag : \tau b \neq 0,~ \lmss{nc}(\ncc{b} \vee \ncc{\tau}) = \lmss{nc}(\ncc{b}\vee 1)+1\}, \\
		&\trdiagp{b} = \{e \in \prdiag : eb \neq 0,~ \lmss{nc}(\ncc{b} \vee \ncc{e}) = \lmss{nc}(\ncc{b}\vee 1)+1 \}.
	\end{split}
\end{equation*}

\subsection{Invariant polynomials and Brauer diagrams}

\par Let $k \geq 1$ and $\lmss{d} = (d_{1},\ldots,d_{n})$ a partition of $N$ and $\{\lmss{d}\} = \{d^{1},\ldots,d^{n}\}$. Let $\mathbb{K}$ be one ot the three division algebras $\mathbb{R}$, $\mathbb{C}$ or $\mathbb{H}$. We introduce the two following compact Lie groups:
\begin{equation*}
U(\lmss{d},\mathbb{K}) = U_{d_{1}}(\mathbb{K}) \times \cdots \times U_{d_{n}}(\mathbb{K}) \text{ and } U^{\sharp}(\lmss{d},\mathbb{K}) = U_{d^{1}}(\mathbb{K}) \times \ldots \times U_{d^{p}}(\mathbb{K}),
\end{equation*}
with $\{d_{1},\ldots,d_{n}\} = \{d^{1} < \ldots < d^{p}\}$. The group $U^{\sharp}(\lmss{d},\mathbb{K})$ is injected into $U(\lmss{d},\mathbb{K})$ using the diagonal injection of each factor $U(d^{i})$ into the product $U(\lmss{d},\mathbb{K})$.

\par An $U(\lmss{d},\mathbb{K})$-invariant polynomial on $\mathcal{M}_{N}(\mathbb{K})^{k}$ is a polynomial function $p:\mathcal{M}_{N}(\mathbb{K})^{k} \rightarrow \mathbb{K}$ and invariant by the diagonal conjugacy action $\lmss{conj}^{k}_{\lmss{d}}$ of the group $U(\lmss{d},\mathbb{K})$ on $\mathcal{M}_{N}(\mathbb{K})^{\times k}$,
\begin{equation*}
	p\left(\lmss{conj}^{k}_{\lmss{d}}(U)\left(A_{1},\ldots,A_{k}\right)\right) = f(UA_{1}U^{-1},\ldots,UA_{k}U^{-1}) = f(A_{1},\ldots,A_{k}),~ \forall U \in U^{\sharp}(\lmss{d},\mathbb{K}).
\end{equation*}
The $U^{\sharp}(\lmss{d},\mathbb{K})$ invariant polynomial functions are defined similarly.
We denote by $P(k,\mathbb{K}, \lmss{d})$ (resp. $P^{\sharp}(k,\mathbb{K},\lmss{d})$) the set of all $U(\lmss{d},\mathbb{K})$ (resp. $U^{\sharp}(\lmss{d},\mathbb{K})$) (real if $\mathbb{K}=\mathbb{R}$ and complex if $\mathbb{K}=\mathbb{C}$ or $\mathbb{H}$) invariant polynomials on $\mathcal{M}_{N}(\mathbb{K})^{\times k}$. The set of invariant tensors $\text{Inv}(k,\mathbb{K},\lmss{d})$ is the set of all elements in $\mathcal{M}_{N}(\mathbb{K})^{\otimes k}$ fixed by the $k$ folded conjugacy action $\lmss{conj}_{\lmss{d}}^{k}$ of $U(\lmss{d},\mathbb{K})$ on $\mathcal{M}_{N}(\mathbb{K})^{\otimes k}$:
\begin{equation*}
	Z \in \lmss{inv}^{k}_{\lmss{d}} \Leftrightarrow \lmss{conj}^{k}_{\lmss{d}}(Z)  = Z.
\end{equation*}
In the sequel, we denote by $\lmss{Tr}_{\mathbb{K}}^{\otimes{k}}$ the $k$ folded tensor product of the matricial trace on $\mathcal{M}_{N}(\mathbb{K})^{\otimes k}$ if $\mathbb{K}=\mathbb{R}$ or $\mathbb{C}$ and $\lmss{Tr}_{\mathbb{H}}^{\otimes k} = \mathcal{R}e(\lmss{Tr})^{\otimes k}$.
The trace $\lmss{Tr}^{\otimes k}$ defines a non degenerate bilinear pairing, $\langle A,B \rangle  = \lmss{Tr}^{\otimes k}(AB)$ for $A,B \in \mathcal{M}_{N}(\mathbb{K})$ which is an $U(\lmss{d}, \mathbb{K})$ invariant polynomial of $(\mathcal{M}_{N}(\mathbb{K})^{\otimes k} \otimes \mathcal{M}_{N}(\mathbb{K})^{\otimes k})^{\star}$. Hence, any $U(\lmss{d},\mathbb{K})$-invariant tensor $Z$ defines an invariant polynomial $f_{Z}$:
\begin{equation}
	\label{inv_tensor_polynomial}
	f_{Z}(A_{1},\ldots,A_{k}) = \langle Z, A_{1} \otimes \ldots \otimes A_{k}\rangle,~ A_{1},\ldots,A_{k} \in \mathcal{M}_{N}(\mathbb{K}).
\end{equation}
We denote by $\mathbb{K}\left[X_{1},\ldots,X_{k},X_{1}^{\star},\ldots,X_{k}^{\star} \right]$ the algebra of polynomials in $k$ non commutating indeterminates. Any invariant polynomial is of the form \ref{inv_tensor_polynomial}. A set of generators for the $U(d,\mathbb{K})$ invariant polynomials with $\mathbb{K} = \mathbb{R}$ or $\mathbb{C}$, can be set equal to
\begin{equation*}
\lmss{Tr}(P(A_{1},\ldots,A_{k})), \text{ with } P \in \mathbb{K}[X_{1},\cdots,X_{k},X_{1}^{\star},\ldots,X_{k}^{\star}]\}
\end{equation*}
with $A^{\star}$ equal to the transpose of $A$ for the case $\mathbb{K}=\mathbb{R}$ and $A^{\star}$ equal to the Hermitian conjugated matrix $A$ for the case $\mathbb{K}=\mathbb{C}$.
For the compact symplectic group $U(N,\mathbb{H}) = Sp(N)$, the complex algebra generated by the set
\begin{equation*}
\lmss{Tr}(P(A,A^{\star})), P \in \mathbb{C}[X_{1},\cdots,X_{k}]\}
\end{equation*}
with $\star$ the quaternionic conjugation being a proper subalgebra of $P(\mathbb{H},k,d)$. Simple results of invariant theory we recall below lead us to the determination of the $U^{\sharp}(\lmss{d},\mathbb{K})$ and $U(\lmss{d},\mathbb{K})$ tensor invariants in $\mathcal{M}_{N}(\mathbb{K})^{\otimes k}$.
\par In fact, let $d \geq 1$ an integer and $\rho^{i}_{d},~ 1 \leq i\leq 2$ two representations of $U(d,\mathbb{K})$ on $\mathbb{K}^{d}$. The space invariants of the sum $\rho^{1}\oplus\rho^{2}$ of the two representations is the sum of the invariants of $\rho_{1}$ and $\rho^{2}$ since
\begin{equation*}
	\forall U \in U(d,\mathbb{K}),~Z_{1},Z_{2} \in \mathbb{K}^{d},~\left(\rho^{1} \oplus \rho^{2}\right)(U)(Z_{1} + Z_{2}) = Z_{1} + Z_{2}, \Leftrightarrow \rho^{1}(U)(Z_{1}) = Z_{1}~\mathrm{and}~\rho^{2}(Z_{2}) = Z_{2}.
\end{equation*}
Hence the space of invariants of $\left(\lmss{nat}_{d} \oplus \lmss{nat}_{d}\right) \otimes \left(\lmss{nat}^{\star}_{d} \oplus \lmss{nat}^{\star}_{d} \right)$ is the sum of four spaces. Each of these spaces is isomorphic to a space of endomorphisms $\lmss{End}(V_{1}, V_{2})$ invariant by $\lmss{nat}_{d} \otimes \lmss{nat}^{\star}_{d}$ with $V_{i},~i \leq 2$ one of the two copies of $\mathbb{K}^{d}$ in $\mathbb{K}^{d} \oplus \mathbb{K}^{d}$. A straightforward generalization of this argument for the $k$-folded action $\left(\lmss{nat}_{d} \oplus \lmss{nat}_{d}\right) \otimes \lmss{nat}_{d}^{\star} \oplus \lmss{nat}_{d}^{\star}$ proves that the polynomials invariants $P(k,\mathbb{K},\lmss{d})$ with $d_{1} = \ldots = d_{n} = d$ admit as a set of generators
\begin{equation*}
	\lmss{Tr}\left(P\left(A^{i_{1}}_{j_{1}},A^{i_{2}}_{j_{2}}, \cdots, A_{j_{p}}^{i_{p}}, \left(A_{j_{1}}^{i_{1}}\right)^{t}, \cdots, \left(A_{j_{p}}^{i_{p}}\right)^{t}\right)\right)
\end{equation*}
with $P$ a non-commutative polynomial and $A^{i}_{j}$ a square block in position $(i,j)$ in the matrix $A$.
For the group $U_{d_{1}} \times U_{d_{2}}$ that acts by $\rho_{1} \times \rho_{2}$ on $\mathbb{K}^{d_{1}}\oplus \mathbb{K}^{d_{2}}$ defined as
\begin{equation*}
	\left(\rho_{1} \times \rho_{2}\right)(U_{1} \times U_{2})(Z_{1} + Z_{2}) = \rho_{1}(U_{1})(Z_{1}) + \rho_{2}(U_{2})(Z_{2})
\end{equation*}
the space of invariant is also the sum of the spaces of invariants for $\rho_{1}$ and $\rho_{2}$. Using the expression
\begin{equation*}
	\lmss{nat}_{d_{1}} \times \lmss{nat}_{d_{2}} = \left(\lmss{nat}_{d_{1}} \times 1\right) \oplus \left(1 \times \lmss{nat}_{d_{2}}\right),
\end{equation*}
we find that the space of tensor invariants for the $k$-folded action $\scalebox{0.9}[1]{$\left( \left(\lmss{nat}_{d_{1}} \times \lmss{nat}_{d_{2}}\right)\otimes \left(\lmss{nat}^{\star}_{d_{1}} \times \lmss{nat}^{\star}_{d_{2}}\right) \right)^{\otimes k}$}$ is a sum of spaces and each term is the space of invariants for a representation of the form
\begin{equation}
	\label{space_inv}
	\lmss{nat}_{d_{i_{1}}} \otimes \lmss{nat}^{\star}_{d_{j_{1}}} \otimes \cdots \otimes \lmss{nat}_{d_{i_{k}}} \otimes \lmss{nat}^{\star}_{d_{j_{k}}},
\end{equation}
where we have written $\lmss{nat}_{d_{1}} = \lmss{nat}_{d_{1}} \times 1$ and $\lmss{nat}_{d_{2}} = 1 \times \lmss{nat}_{d_{2}}$ for brevity. We write a representation $\eqref{space_inv}$ as the tensor products of two representations $a$ and $b$ of respectively $U(d_{1})$ and $U(d_{2})$
\begin{equation*}
	\lmss{nat}_{d_{i_{1}}} \otimes \lmss{nat}_{d_{j_{1}}}^{\star} \otimes \cdots \otimes \lmss{nat}_{d_{i_{k}}} \otimes \lmss{nat}_{d_{j_{k}}}^{\star} = \left(a \otimes b \right),
\end{equation*}
by setting $a$ to be equal to the tensor product $\eqref{space_inv}$ in which $\lmss{nat}_{d_{2}}$ and $\lmss{nat}_{d_{2}}^{\star}$ have been replaced  by the trivial representation. The representation $b$ is defined similarly. The space of invariant of \eqref{space_inv} is the tensor product of the space of invariants of $a$ and $b$. One must have an equal number of the representation $\lmss{nat}_{d_{i}}$ and of its contragredient representation for the space of invariants of the representation $\eqref{space_inv}$ to be non trivial.
Let $A \in \mathcal{M}_{n}$ and $1 \leq i,j\leq n$ two integers. We denote by $A(i,j)$ the block of $A$ of dimension $d_{i}\times d_{j}$ in position $(i,j)$. We denote by $\tilde{A}^{i}_{j}$ the matrix of dimension $(d^{1}+d^{2})\times (d^{1} + d^{2})$ having the block $A(i,j)$ in position $(i,j)$ and the remaining coefficient equal to $0$. A set of generators for the polynomials invariant for the representation $U(d_{1}) \times U(d_{2})$ is given by
\begin{equation}
	\lmss{Tr}\left(P\left( \left(\tilde{A}^{i}_{j},~\left(\tilde{A}^{k}_{l}\right)^{t},~i,j,k,l \leq n\right)\right)\right)
\end{equation}
with $P$ a non commutative polynomial in $2n$ indeterminates.
Having discussed what are the polynomials invariants of $\lmss{nat}_{d_{1}} \times \lmss{nat}_{d_{2}}$ and $\lmss{nat}_{d_{1}}\oplus \lmss{nat}_{d_{2}}$, it is now straightforward to prove that a set of generators for the polynomial invariants of $U(\lmss{d})^{\sharp}$ is given by
\begin{equation*}
	\lmss{Tr}\left(P\left({A}^{i}_{j},~i,j,k,l \leq n \right)\right)
\end{equation*}
with $P$ a polynomial in the non commutative indeterminates $X^{i}_{j}, i,j\leq n$ and ${(X^{t})}^{k}_{l},k,l \leq n$ that satisfies the following condition. A monomial $X$ is in $P$ is for two consecutive indeterminates in $X$, say $X^{i}_{j}$ and $X_{k}^{l}$ in this order, we have $d_{j} = d_{l}$. The same condition holds for two consecutives indeterminates $X^{i}_{j}$ and ${X^{t}}^{l}_{k}$.
\par We state three lemmas that justify the introduction of the algebra of coloured Brauer diagrams.
We use the following convention: for a matrix $A \in \mathcal{M}_{N}(\mathbb{R})$ or in $\mathcal{M}_{N}(\mathbb{C})$, we set the expression $\left[A\right]^{-1}$ equal to the transpose of $A$ and we set $\left[A\right]^{1}$ equal to $A$. If $A$ is a matrix with quaternionic entries, we set $\left[A\right]^{-1}$ equal to $A^{\star}$ (the quaternionic transpose of $A$).
\begin{lemma}[\cite{levy}]
	\label{simple_lemma_levy}
	Let $N \geq 1$ an integer and $A_{1},\ldots,A_{k} \in \mathcal{M}_{n}(\mathbb{R})$. Let $b$ be a non-coloured oriented Brauer diagram, then
	\begin{equation*}
		\lmss{Tr}^{\otimes k}\left((A_{1} \otimes \ldots \otimes A_{k}) \circ \ncc{\rho}_{n}(b)\right) = \prod_{(i_{1},\ldots,i_{k}) \in \sigma_{(b,s)}} \lmss{Tr}\left(\left[A_{i_{1}}\right]^{s(i_{1})} \cdots \left[A_{i_{k}}\right]^{s(i_{k})}\right),
	\end{equation*}
where we have denoted by $\left( i_{1},\ldots,i_{k} \right)$ a cycle of $\sigma_{(b,s)}$.
\end{lemma}
 With the coloured version of the representation $\rho_{d}$ we defined in the previous section, the following lemma is a simple consequence of Lemma \ref{simple_lemma_levy}.
\begin{lemma}
\label{brauertrace}
	Let $N \geq 1$ an integer, and $\lmss{d}=(d_{1},\ldots,d_{n})$ a partition of $N$ into $n$ parts. Let $A_{1},\ldots,A_{k} \in \mathcal{M}_{n}(\mathbb{R})$. Let $b$ be a coloured oriented Brauer diagram, then
	\begin{equation*}
	\resizebox{\hsize}{!}{$
		\lmss{Tr}^{\otimes k}\left((A_{1} \otimes \ldots \otimes A_{k}) \circ \rho_{\lmss{d}}((b,s))\right) = \displaystyle\prod_{\substack{(i_{1},\ldots,i_{k}) \\ \in \sigma_{(b,s)}}} \lmss{Tr}\left(\left[A_{i_{1}}(c_{b}(i_{1}),c_{b}(i^{\prime}_{1}) )\right]^{s(i_{1})} \cdots \left[A_{i_{k}}(c_{b}(i_{k}),c_{b}(i^{\prime}_{k})\right]^{s(i_{k})}\right).$}
	\end{equation*}
\end{lemma}

We also defined a coloured version of the representation $\rho^{\mathbb{H}}$ defined in \cite{levy} to study large Brownian quaternionic matrices. The following lemma is a straightforward corollary of the lemma $2.6$ in \cite{levy}.
\begin{lemma}
	Let $A_{1}, \ldots, A_{n} \in \mathcal{M}_{n}(\mathbb{K})$, $(b,s)$ an oriented coloured Brauer diagram, then
	\begin{equation*}
	\resizebox{0.95\hsize}{!}{$%
		\mathcal{R}e(\lmss{Tr}^{\otimes k})\left(\left(A_{1} \otimes \cdots \otimes A_{k} \circ \rho^{\mathbb{H}}(b)\right)\right) = \displaystyle\prod_{\substack{(i_{1},\ldots,i_{k}) \\ \in \sigma_{(b,s)}}} \mathcal{R}e\lmss{Tr}\left(\left[A_{i_{1}}(c_{b}(i_{1}),c_{b}(i^{\prime}_{1}) )\right]^{s(i_{1})} \cdots \left[A_{i_{k}}(c_{b}(i_{k}),c_{b}(i^{\prime}_{k})\right]^{s(i_{k})}\right)$}
	\end{equation*}
\end{lemma}

\section{Square extractions of an unitary Brownian motion}
\label{squareextractions}
\label{square_case}

Let $\mathbb{K}=\mathbb{R}, \mathbb{C}$ or $\mathbb{H}$. If $n,d\geq 1$ are two integers, we denote by $\unitaryfd$ the quantum process on the dual Voiculescu group $\udualgroup$ extracting blocks of size $d\times d$ from a unitary Brownian motion of dimension $nd$:
\begin{equation}
	\begin{array}{cccc}
	\unitaryfd(t):&\udualgroup&\to&\mathcal{M}_{d}(L^{\infty-}(\Omega,\mathcal{F},\mathbb{P})) \\
	&u_{ij}&\mapsto&\munitaryfd(t)(i,j)
\end{array}
\end{equation}
The quantum process $\unitaryfdc$ and $\unitaryfdr$ are seen as valued in the algebraic probability spaces $\mathcal{M}_{N}(\mathbb{R}), \mathbb{E}\otimes \frac{1}{d}\lmss{Tr}$, respectively $\mathcal{M}_{N}(\mathbb{C}),\mathbb{E}\otimes \frac{1}{d}\lmss{Tr}$ while for the quaternionic case, it is mandatory to take the real part of the trace, $\unitaryfdh$ is valued into the algebraic probabilit space $\mathcal{M}_{N}(\mathbb{H}),\mathbb{E}\otimes\frac{1}{d}\mathcal{R}e\circ \lmss{Tr}$.
This section is devoted to the proof of our main theorem stated below.
\begin{theorem}
\label{maintheoremsquare}
	Let $n \geq 1$ an integer. Each of the three processes $\unitaryfd$ $\mathbb{K} = \mathbb{R},\mathbb{C}$ or $\mathbb{H}$, converges in non-commutative distribution as $d\rightarrow +\infty$ to the free $n$-dimensional Brownian motion $\freeunitary$.
\end{theorem}
Let $t\geq 0$, $\mathbb{K}$ be one of the three divisions algebra in \ref{maintheoremsquare} and $n\geq 1$ an integer, that stands for the number of blocks the random matrices will be cut into and $d\geq 1$ an other integer that is the dimension of each of these blocks. Set $N=nd$. The proof of \ref{maintheoremsquare} proceeds as follows. First, we focus on the one dimensional marginals. We begin by showing that a suitable set of statistics of the process which comprises its distribution satisfies a differential system. Secondly, the convergence of the generator as the dimension $d \rightarrow +\infty$ of that system is proved and we give a formula for the limit.
Finally, we draw a comparison between the generator of the limit of the process $\unitaryfd$ and the generator of the free $n-$dimensional Brownian motion $\freeunitary$.
Our proof of the convergence of the multidimensional marginals of the process $\unitaryfd$, heavily rely on Theorem \ref{collins_sniady} and conjugation invariance of the law of the Browian motions on unitary groups.
\par This method has already been applied by Levy in \cite{levy} to prove the convergence in non commutative distribution of the unitary Brownian motion, this corresponds to the case $n=1$ in our setting. Introducing of the algebra of coloured Brauer diagrams makes the computations for the case $n=1$ and $n>1$ very similar, which is obviously an argument in favor of the lengthy exposition made in Section \ref{schur_weyl}.
\par Concerning the outline, the complex, real and quaternionic cases are treated separately to prove the convergence of the one dimensional marginals.

\par We use the shorter notation $\rho^{\mathbb{K}}$ for one the three representations of the algebra of coloured Brauer diagrams $\brauer(\underset{n}{\underbrace{(d,\ldots,d)}})$ we defined in the last section. We use the injection $\Delta$ of the algebra $\brauerz(nd)$ of uncoloured Brauer diagrams without making mention of it.

\subsection{Convergence of the one-dimensional marginals, the complex case}
\label{complex}
Let $t \geq 0$, the convergence in distribution of the one dimensional marginals of the quantum process $\unitaryfd$ is implied by the following proposition.
\begin{proposition}
\label{convcomplexsquare}
Let $t \geq 0$, $r\geq 0$ and $\alpha,\beta \in \{1,\ldots,n\}^{r}$. The mixed moments of the family
	\begin{equation*}
\left\{ \munitaryfdc(t)(\alpha_{i},\beta_{i}), \left(\munitaryfdc(t)(\alpha_{i},\beta_{i})\right)^{t}, \overline{\munitaryfdc(t)(\alpha_{i},\beta_{i})}, \left(\munitaryfdc(t)(\alpha_{i},\beta_{i})\right)^{\star}, i \leq r \right\}
	\end{equation*}
in the tracial algebra $\mathcal{M}_{d}(L^{\infty}(\Omega,\mathcal{F},\mathbb{P},\mathbb{C}),\mathbb{E} \otimes \lmss{Tr})$ converge as $d \rightarrow +\infty$.
\end{proposition}
\label{propcomplex}
\begin{remarque}
	The last proposition does not only imply the convergence of the distribution of $\unitaryfd(t)$. In fact, the convergence stated in Proposition \ref{convcomplexsquare} is much more general and holds also for words on the matrix transpose ${\munitaryfd}^{t}$ and on the complex conjugate $\overline{\munitaryfd}(t)$ (without transposition).
\end{remarque}
Owing to equation \ref{formula_tensor_process_c} of Proposition \ref{meantensorlevy}, mean of polynomials of the matrix $\munitaryfdc(t)$ admits the following combinatorial formula:
\begin{equation}
\label{meancomplex}
	\mathbb{E}\left[w^{\otimes}({\munitaryfdc(t)}) \right] = \exp\left(-\frac{kt}{2} - \frac{1}{N}\sum_{\substack{1 \leq i,j \leq k, \\ w_{i} = w_{j}}} \rho^{\mathbb{C}}(\ncc{\tau}_{ij})  + \frac{1}{N}\sum_{\substack{1 \leq i,j \leq k, \\ w_{i} \neq w_{j}}} \rho^{\mathbb{C}}(\ncc{e}_{ij}) \right),~w\in\bar{\lmss{M}}_{2}(k).
\end{equation}
As explained previously, we are looking for a linear space of statistics that are polynomial functionals in the coefficients of a matrix (and its conjugate), which is manifestly invariant by the generator of the unitary Brownian motion and contains traces of matrix powers. To define such statistics, we use the algebra of coloured  Brauer diagrams defined in Section \ref{schur_weyl}.
The symbol $\overline{\lmss{M}}_{1}$ stands for the free monoids generated by two letters $\{x_{1},\overline{x}_{1}\}$. If $w \in \overline{\lmss{M}}_{1}$ is a word of length $k$ and $A \in \mathcal{M}_{nd}(\mathbb{C})$ a complex matrix, $w^{\otimes}(A)$ stands for the monomial in $\mathcal{M}_{nd}(\mathbb{C})^{\otimes}$ that is obtained by the substitutions rules: $x_{1}\to A, \overline{x}_{1}\to \overline{A}$.
Let $k\geq 1$ an integer. The subset of words in $\overline{M}_{1}$ of length $k$ is denoted $\overline{\lmss{M}}_{1}(k)$. Let $w \in \overline{\lmss{M}}_{1}(k)$ a word and $b$ a (coloured) Brauer diagram in $\brauer$. We set
\begin{equation}
\lmss{m}_{d}^{\mathbb{C}}(w,b,A) = \lmss{Tr}(\rho^{\mathbb{C}}(b) \circ w^{\otimes}(A)),~A \in \mathcal{M}_{N}(\mathbb{C}).
\end{equation}
We are now concerned with the derivative of the normalized statistic $\lmss{m}^{\mathbb{R}}_{s}$ defined by  $$\mathbb{m}_{d}^{\mathbb{C}}\left(w,b,t\right) = d^{-\lmss{nc}(\ncc{b} \vee \mathbf{1}_{k})}\lmss{m}^{\mathbb{C}}_{d}(w,b,\mathbb{E}\left[\munitaryfdc\right]) $$
The statistic $\mathbb{m}_{d}^{\mathbb{C}}$ is extended as a linear function on the space $\mathcal{B}_{k} \otimes \overline{\lmss{M}}_{2}(k) = \brauer(nd) \otimes \mathbb{R}\left[\overline{\lmss{M}}_{2}(k)\right]$. Note that owing to \ref{mainpropbrauer}, the range of $b \to \mathbb{m}_{d}^{\mathbb{C}}(b,t)$ comprises the distribution of $\unitaryfdc(t)$. By using formula \eqref{meancomplex}, we prove the existence of an operator $L^{\mathbb{C}}_{d}: \mathcal{F}\left(\mathcal{B}_{k} \otimes \bar{\lmss{M}}_{2}(k) \right) \rightarrow \mathcal{F}\left(\mathcal{B}_{k} \otimes \bar{\lmss{M}}_{2}(k) \right)$ such that
\begin{equation*}
	\begin{split}
		\frac{d}{dt}\mathbb{m}_{d}^{\mathbb{C}}(b\otimes w,t) &= L_{d}^{\mathbb{C}}\left(\mathbb{m}_{d}^{\mathbb{C}}(t)\right)(b\otimes w),~\mathbb{m}_{d}^{\mathbb{C}}(b \otimes w,0)= \delta_{\Delta_{k}}(b), b\otimes w \in \brauer\times\overline{\lmss{M}}_{2}(k).
	\end{split}
\end{equation*}
In the last formula, we use the notation $\Delta_{k}$ for the support function of the set $\Delta_{k}\subset \mathcal{B}_{k,n}$ of Brauer diagrams that are diagonally coloured: $\brauergen \in \brauer \Leftrightarrow c_{b}(i) =c_{b}(i^{\prime}),~\forall 1 \leq i \leq k$. An explicit expression for the operator $L^{\mathbb{C}}_{d}$ is given by
\begin{equation}
\label{computgenc}
	\begin{split}
		&L^{\mathbb{C}}_{d}(g)(w,b) = \frac{k}{2}g(w,b) -\frac{1}{nd} \sum_{\substack{1 \leq i,j \leq k \\ w_{i} = w_{j}}}d^{\lmss{nc}(\ncc{b} \vee \ncc{\tau}_{ij}) - \lmss{nc}(\ncc{b} \vee 1)}g(w,\ncc{\tau}_{ij}\circ b) \\ &\hspace{7cm}+ \frac{1}{nd} \sum_{\substack{1 \leq i,j \leq k \\ w_{i} \neq w_{j}}}d^{\lmss{nc}(\ncc{b} \vee \ncc{e}_{ij}) - \lmss{nc}(\ncc{b} \vee 1)}g(w,\ncc{e}_{ij}\circ b).
	\end{split}
\end{equation}
The coefficients $d^{\lmss{nc}(\ncc{b}\vee \ncc{r}_{ij})}-\lmss{nc}(\ncc{b}\vee \mathbf{1}_{k})$ are obtained by using the fundamental equality $d^{\lmss{nc}(\ncc{b}\vee \ncc{r}})=\lmss{nc}(r\ncc{b}\vee\mathbf{1}_{k})+ \mathcal{K}_{d}(b,r) = \lmss{nc}(b\vee r)$, with $r$ an elementary diagram and $b$ a coloured Brauer diagram. Note that the symbol $\ncc{r}$ stands for two different objects in the last equation. If acting by multiplication on a coloured Brauer diagram, $\ncc{r}$ is to be seen as an element of $\brauer$ as described in the introduction of the present section. On the other hand, in the expression $\ncc{r} \vee \ncc{b}$,  the symbol $\ncc{r}$ stands for an elementary uncoloured Brauer diagram in $\brauerz$.

Proposition \ref{mainpropbrauer} implies that $\lmss{nc}(\ncc{b} \vee \ncc{\tau}_{ij}) - \lmss{nc}(\ncc{b} \vee 1) \in \{-1,0,1\} $ and $\lmss{nc}(\ncc{b} \vee \ncc{e}_{ij}) - \lmss{nc}(\ncc{b} \vee 1)\in \{-1,0,1\}$. Hence, the two sums in the right hand side of equation \eqref{computgenc}, as $d\rightarrow +\infty$, converge to two sums over elementary diagrams $r$ such that $r\circ b$ has one more cycle than $b$, or $rb$ has a loop (which means $rb = d (\ncc{r}\circ b)$). With the notation introduced in Section \ref{schur_weyl},
\begin{equation}
\label{system_limit_complex}
\begin{split}
L_{d}^{\mathbb{C}}(g)(b\otimes w)&=\frac{k}{2}b\otimes w -\frac{1}{n}\sum_{\substack{1 \leq i<j \leq k \\ w_{i} = w_{j},\\ \ncc{\tau}_{ij} \in \trpnc{\ncc{b}}}} g(\ncc{\tau}_{ij}\circ b \otimes w) + \frac{1}{n} \sum_{\substack{1 \leq i < j \leq k \\ w_{i} \neq w_{j} \\ \ncc{e}_{ij} \in \prpnc{b}}}g(\ncc{e}_{ij}\circ b \otimes w)+o_{d,\infty}(\frac{1}{d}) \\
&= \bar{L}_{n}(g)(b\otimes w) + o_{d,\infty}(\frac{1}{d})
\end{split}
\end{equation}
for any function $g \in \mathcal{F}\left(\brauer \otimes\overline{\lmss{M}}_{2}(k)\right)$. Note that the range of the two sums in the last equation does not exhaust the sets $\trp{b}$ or $\prp{b}$ since the elementary diagrams $\ncc{r}, r \in \{e,\tau\}$ are, by definition, in the range of the injection $\Delta$ of the algebra of uncoloured Brauer diagrams, non mixing. The differential system satisfied by the functional $t\to \mathbb{m}_{d}^{\mathbb{C}}(t)$ is linear and finite dimensional. Its solutions can be expressed as the (matrix) exponential of its generator $L_{d}^{\mathbb{C}}$. Since the matrix exponential is a continuous map, convergence of the generator $L_{d}^{\mathbb{C}}$ implies the convergence of $\mathbb{m}_{d}^{\mathbb{C}}$ to the solution $\overline{\mathbb{m}}_{n}$ of:
\begin{equation*}
\frac{d}{dt}\overline{\mathbb{m}}_{n}(t) = \bar{L}_{n}(\overline{\mathbb{m}}_{n}(t)),~ \overline{\mathbb{m}}_{n}(0) = \delta_{\Delta_{k}}.
\end{equation*}
We end this section by stating the convergence of the one dimensional marginal of the process $\unitaryfd$. In fact, for each word $u$ on the generators of the dual Voiculescu group, there exists a Brauer diagram $b_{u}$ and a word $w_{u}$ such that $\lmss{m}^{\mathbb{C}}_{d}(b_{w}, u_{w},t) = \frac{1}{d}\lmss{Tr}(\unitaryfd(t)(w))$. The point-wise convergence of $\lmss{m}^{\mathbb{C}}_{d}$ implies the convergence of $\frac{1}{d}\lmss{Tr}(\unitaryfd(t)(w))$ as $d\rightarrow +\infty$ . Later, we should exploit a property satisfied by the pair $(b_{u},w_{u})$ to draw a comparison between $\overline{L}_{n}$ and the generator of the free n-dimensional unitary Brownian motion.

\subsection{Convergence of the one-dimensional marginals, the real case.}
\label{real}
Let $t \geq 0$ a time and $k \geq 0$ an integer, that will be the size of the diagrams and length of words that are considered. To treat the real case, we define the real counterpart of the statistics $\mathbb{m}_{d}^{\mathbb{C}}$ that were defined in the previous section. We are more brief in this section to expose the convergence of the one dimensional marginals of $\unitaryfdr$ since the method used for the complex case is applied here too. The representation $\rho^{\mathbb{R}}$ is defined in Section \ref{schur_weyl}.For each real matrix $A \in \mathcal{M}_{N}(\mathbb{R})$, we define the function $\lmss{m}_{d}^{\mathbb{R}}(t)$ on the set of coloured brauer diagram $\brauer$ by the equation
\begin{equation*}
	\lmss{m}_{d}^{\mathbb{R}}(b,t) = \lmss{Tr}\left(\mathbb{E}\left[A^{\otimes k}\right] \circ \rho^{\mathbb{R}}(b)\right),~ b \in \brauer
\end{equation*}
and extend it linearly to $\mathbb{R}\left[\brauer\right]$. The normalized statistics $\mathbb{m}_{d}^{\mathbb{R}}(b,t)$ of the real unitary Brownian diffusion is
$$ \mathbb{m}_{d}^{\mathbb{R}}(b,t) = d^{-\lmss{nc}(\ncc{b} \vee 1)}\lmss{Tr}\left(\mathbb{E}\left[\munitaryfdr(t)^{\otimes k}\right] \circ \rho^{\mathbb{R}}(b)\right),~ b \in \brauer.$$
By using the representation $\rho^{\mathbb{R}}$, the mean $\mathbb{E}\left[\munitaryfdr(t)^{\otimes k}\right]$ can be expressed as:
\begin{equation}
\label{meanreal}
	\mathbb{E} \left[{\munitaryfdr(t)}^{\otimes k} \right]= \exp\left(-\frac{kt(N-1)}{2N}+t\left(-\frac{1}{N}\sum_{\substack{1 \leq i,j \leq k}}\rho^{\mathbb{R}}(\ncc{\tau}_{ij}) + \frac{1}{N} \sum_{ \substack{1 \leq i < j \leq k}} \rho^{\mathbb{R}}(\ncc{e}_{ij}) \right)\right).
\end{equation}
Once again the range of $\mathbb{m}_{d}^{\mathbb{R}}(t)$ comprises the distribution of $\unitaryfdr$. We compute the derivative of $\mathbb{m}_{d}^{\mathbb{R}}$. We apply formula \eqref{meanreal} to get:
\begin{equation*}
\frac{d}{dt} \mathbb{m}_{d}^{\mathbb{R}}(b,t) = L_{d}^{\mathbb{R}}(\mathbb{m}^{\mathbb{R}}_{d}),~\mathbb{m}_{d}^{\mathbb{R}}(0) = \delta_{\Delta_{k}}.
\end{equation*}
where the operator $L_{d}^{\mathbb{R}}$ acting on linear forms $\mathbb{R}\left[\brauer \right]$ is defined by the formula:
\begin{equation*}
\resizebox{\hsize}{!}{$
	L_{d}^{\mathbb{R}}(g)(b)= -\frac{k(N-1)}{2N}g(b) -\frac{1}{nd}\left(\displaystyle\sum_{\ncc{\tau} \in \nccc{\tr}}d^{\lmss{nc}(\ncc{\tau} \vee \ncc{b}) - \lmss{nc}(\ncc{b} \vee 1)}g(\ncc{\tau} \circ b)+\displaystyle\sum_{\ncc{e}\in\nccc{\pr}} d^{\lmss{nc}(\ncc{e} \vee \ncc{b}) - \lmss{nc}(\ncc{b} \vee 1)}g(\ncc{e}\circ b) \right)$}.
\end{equation*}
with $g \in \mathbb{R}\left[\brauer\right]^{\star}$. Again, as for the complex case, the last two sums localize to sums over diagrams in, respectively $\trp{b}$ and $\prp{b}$ if we let $d \rightarrow +\infty$. If we define the operator $L_{n}$ acting on functions of Brauer diagrams, by
\begin{equation}
	\label{system_limit_real}
	L_{n}(g) = -\frac{k}{2}g(b)-\frac{1}{n}\sum_{\substack{\ncc{\tau} \in \trpnc{b}}} g(\ncc{\tau} \circ b) + \frac{1}{n} \sum_{\substack{\ncc{e} \in \prpnc{b}}}g(\ncc{e}\circ b),
\end{equation}
with $g \in \mathbb{R}\left[\brauer\right]^{\star}$, we get $L_{d}^{\mathbb{R}} = L_{n} + o_{d,\infty}(d)$. Again, as for the complex case, the convergence of the generator of the system satisfied by the statistics $\lmss{m}_{d}^{\mathbb{R}}$ implies the convergence of the solution to the function $\mathbb{m}_{n}$ that satisfies the differential system:
\begin{equation*}
\frac{d}{dt}\mathbb{m}_{n}(t) = L_{n}(\mathbb{m}_{n}(t)),~\mathbb{m}_{n}(0) = \delta_{\Delta_{k}}.
\end{equation*}

\subsection{Convergence of the one-dimensional marginals, the quaternionic case}
Let $t \geq 0$ a time. As we did for the complex and real cases, we use the representation $\rho^{\mathbb{H}}$ to define, for any matrix of size $N\times N$ with quaternionic entries, a function on the set of coloured Brauer diagrams $\brauer_{k}$ of size $k$ (we recall that the dimension function used to defined $\brauer_{k}$ is the constant dimension function on $\llbracket 1,\ldots,n\rrbracket$ equal to $d$),
$$
\lmss{m}_{d}^{\mathbb{H}}(A,b) = -2\mathcal{R}e\lmss{Tr}\left(A \circ \rho^{\mathbb{H}}(b)\right),~A\in\mathcal{M}_{N}(\mathbb{H})
$$
Further, we normalize these statistics using a slighlty different normalization factor that was used for the real and the complex case, the main reason being that $\rho^{\mathbb{H}}$ is a representation of the algebra $\brauer(-2N)$:
$$
~\mathbb{m}_{d}^{\mathbb{H}}(t,b) = \frac{1}{(-2d)^{\lmss{nc}(\ncc{b} \vee \mathbf{1}_{k})}} \lmss{m}_{d}^{\mathbb{H}}(b,\mathbb{E}\left[\munitaryfdh\right]).
$$
For the third time, the range of $\mathbb{m}_{d}^{\mathbb{H}}(t)$ comprises the distribution of $\unitaryfdh$. Owing to Proposition \ref{meantensorlevy}, the mean of tensor monomials of $\unitaryfdh(t)$ is expressed as:
\begin{equation}
\label{meanquaternion}
\mathbb{E}\left[\munitaryfdh(t)^{\otimes k} \right] = \exp\left(-\frac{tk}{2}\left(\frac{2N-1}{2N}\right) - \frac{1}{N}\sum_{\ncc{e} \in \nccc{\pr}}\rho^{\mathbb{H}}(\ncc{e}) + \frac{1}{N} \sum_{\ncc{\tau} \in \nccc{\tr}}\rho^{\mathbb{H}}(\ncc{\tau}) \right).
\end{equation}
Similar computations as done for the real and complex case lead to a differential system satisfied by the statistics $\mathbb{m}_{d}^{\mathbb{H}}$. Note that, contrary to the complex case, we do not need an extra parameter (a word in $\lmss{M}_{2}$), even if the conjugation is not trivial on $\mathbb{H}$. Let $b \in \brauer$ a Brauer diagram, there exists an on operator $L_{d}^{\mathbb{H}}$ acting on the space of functions on $\brauer$ such that:
\begin{equation*}
\frac{d}{dt} \mathbb{m}_{d}^{\mathbb{H}}(b,t) = L_{d}^{\mathbb{H}}(\mathbb{m}_{d}^{\mathbb{H}}(t))(b),~ \mathbb{m}_{d}^{\mathbb{H}}(0) = \delta_{\Delta}.
\end{equation*}
The operator $L_{d}^{\mathbb{H}}$ is given by the formula
\begin{equation*}
	\begin{split}
		L_{d}^{\mathbb{H}}(g)(b)&= -\frac{k(2N+1)}{4N}g(b) - \frac{1}{n}\sum_{\ncc{\tau} \in \tr} \left(-2d\right)^{\lmss{nc}(\ncc{b} \vee \ncc{\tau}) - \lmss{nc}(\ncc{b} \vee 1) - 1}g(\ncc{\tau} \circ b) \\
		&\hspace{5cm}+ \frac{1}{n}\sum_{\ncc{e} \in \nccc{\pr}} \left(-2d\right)^{\lmss{nc}(\ncc{b} \vee \ncc{e}) - \lmss{nc}(\ncc{b} \vee 1) - 1}g(\ncc{e}\circ b)
\end{split}
\end{equation*}
where $g\in\mathbb{R}\left[\brauer\right]^{\star}$. As $d$ tends to infinity the generator $L^{\mathbb{H}}_{d}$ converges to $L_{n}$ defined in equation \eqref{system_limit_real}. This is sufficient to prove the convergence of the one dimensional marginals of $\unitaryfdr$. In addition, $\unitaryfdr$ and $\unitaryfdh$ converge in non-commutative distribution to the same limit.
\subsection{Convergence of the one-dimensional marginals: conclusion}
Let $\mathbb{K}$ be one of the three division algebras $\mathbb{R},\mathbb{C}$ or $\mathbb{H}$. In the last section we proved the convergence of the one dimensional marginals of the process $\unitaryfd$. We exhibit differential system the limiting distributions are solution of and saw that limiting non-commutative distributions of the one dimensional marginals of the real and quaternionic processes are equal, we prove now that it also holds for the complex one dimensional marginals.

\par To that end, we define first the notion of compatible pairs in $\brauer \times \overline{\lmss{M}}_{1}(k)$. Let $b \in \brauer$ a Brauer diagram. Recall $\ncc{s}_{b}$ denotes the orientation of $b$ that is positive on the minimum of the cycles of $b$. Define the subset $\mathcal{S}(b)$ of $\{-1,1\}^{k}$ by $S(b) = \{s \in \{-1,1\}^{k} : \ncc{s}_{b}(i)\ncc{s}_{b}(j) = s_{i}s_{j}~\forall~i,j \leq k,~i \sim_{\ncc{b} \vee 1} j \}$ and the set of compatible words and diagrams by $\mathcal{C} = \{(b,w^{s}), b \in \brauer,~s \in \mathcal{S}(b)\}$.
If $w\in \overline{\lmss{M}}_{1}(k)$ is a word, $\overline{w}$ is obtained by substituting $x_{1}$ in place of $\overline{x}_{1}$ and vice-versa. The operator $L_{n}$,respectively $\overline{L}_{n}$, acts on the space of linear forms on $\mathbb{R}\left[\brauer\right]$, respectively on the space of linear forms on $\mathbb{R}\left[\overline{\lmss{M}}_{1}\right] \otimes \mathbb{R}\left[\brauer \right]$. To state the next proposition, we find convenient to consider the dual operators $L^{\star}_{n}$ and $\bar{L}^{\star}_{n}$ acting, respectively, on $\mathbb{R}\left[\brauer\right]$ and $\mathbb{R}\left[\overline{\lmss{M}}_{1}(k) \right] \otimes \mathbb{R}\left[\brauer\right]$.
\begin{lemma}
	\label{lemma_comparison}
	Let $b\in \brauer$ a Brauer diagram and $w \in \lmss{M}_{2}(k)$ a word. For any time $t\geq 0$, one has $\overline{\mathbb{m}}_{n}(b,w,t) = \overline{\mathbb{m}}_{n}(b,\overline{w},t)$. The real vector space generated by tensors of compatible words and diagrams is stable by the action $\overline{L}^{\star}_{n}$. In addition, $\bar{L}^{\star}_{n}(b \otimes w) = 	(L^{\star}_{n}(b) \otimes w)$ for any pair $(b,w) \in \mathcal{C}$.
\end{lemma}
\begin{proof}
The first assertion of Lemma \ref{lemma_comparison} is trivial. Let $r \in \prp{b} \cup \trp{b}$, then $\ncc{s}_{r_{ij}\circ b}(i) \\ \ncc{s}_{r_{ij}\circ b}(j) = \ncc{s}_{b}(i)\ncc{s}_{b}(j)$ whenever $i\sim_{\ncc{rb} \vee 1}j$. It follows that $\mathcal{C}$ is stable by $\bar{L}_{n}$. Let $b$ a Brauer diagram. Let $(b,w^{s}) \in \mathcal{C}$ with $s\in S(b)$. From Proposition \ref{mainpropbrauer}, Section \ref{schur_weyl}, $\ncc{e}_{ij} \in \prp{b}$ if and only if $\ncc{s}_{b}(i)\ncc{s}_{b}(j) = -1 = s(i)s(j)$ and $\ncc{\tau}_{ij} \in \trp{b}$ if and only if $\ncc{s}_{b}(i)\ncc{s}_{b}(j) = 1 = s(i)s(j)$. Since $w^{s}_{i} = \overline{w^{s}}(j) \Leftrightarrow s(i)s(j)=-1$ and $w^{s}_{i} = w^{s}(j) \Leftrightarrow s(i)s(j)=1$ we get $\bar{L}^{\star}_{n}(b,w) = (L^{\star}_{n}(b),w)$.
\end{proof}
If $b$ and $w$ are compatible diagrams and word, with $b$ having only one cycle, Lemma \ref{lemma_comparison} implies the equalities:
\begin{equation*}
	\overline{\mathbb{m}}_{n}(b,w,t) = \overline{\mathbb{m}}_{n}(b,\overline{w},t),~\mathrm{and}~ \overline{\mathbb{m}}_{n}(b,w) = \mathbb{m}_{n}(b),
\end{equation*}
if the diagram $b$ has more than one cycle, the word $w$ in the first of the last two inequalities can partially conjugated: we can swap exchange all letters $x_{i}$ and $\overline{x_{i}}$ which positions in the word $w$ are two integers belonging to the same cycle of $b$ and leave the other letters untouched.

Now, to each word $u \in \udualgroup$ is associated a pair $(b_{u},w_{u})$ of compatible word and diagram (we are simply requiring that transposition and conjugation occur both at a time on the matrix $\munitaryfdc(t)$) such that $\frac{1}{d}\mathbb{E}\lmss{Tr}(\unitaryfdc(t)(u)) = \mathbb{m}_{d}^{\mathbb{C}}(b_{u},w_{u},t)$. As the dimension tends to infinity, $\frac{1}{d}\mathbb{E}\left[\lmss{Tr}(\unitaryfdc(t)(u))\right] \rightarrow \overline{\mathbb{m}}_{n}(b_{u},w_{u},t) = \mathbb{m}_{n}(b_{u},t)$.

Now owing to the formula $\eqref{generatorfree}$ for the generator of the pseudo-unitary diffusion, $\mathcal{L}_{n}(u) = \delta_{\Delta_{k}}(L_{n}(b_{u}))$. This last inequalities implies
\begin{equation}
\unitaryfd(t) \overset{\text{dist.nc.}}{\rightarrow} \freeunitary(t).
\end{equation}

\subsection{Convergence of the multidimensional marginals}
To finish the proof of Theorem \ref{maintheoremsquare}, we prove convergence of the multidimensional marginals of $\unitaryfd$ by using Theorem \ref{collins_sniady}, which in turn relies on conjugation invariant property of the process's distribution; for any  unitary matrix in $\mathbb{U}(d,\mathbb{K})$ and words $u_{1},\ldots,u_{q} \in \udualgroup$ in the dual Voiculescu group, the family $\{U\unitaryfd(t_{1})U^{-1},\ldots,U\unitaryfd(t_{q})U^{-1}\}$ has the same non-commutative distribution as the family $\{\unitaryfd(t_{1}),\ldots,\unitaryfd(t_{q})\}$.
\begin{theorem}[Voiculescu; Collins,Sniady]
\label{collins_sniady}
Choose $\mathbb{K} \in \{\mathbb{R}, \mathbb{C},\mathbb{H}\}$. Let $\left(A_{N,1},\ldots,A_{N,n} \right)_{N\geq1}$ and $\left(B_{N,1},\ldots,B_{N,n} \right)$ be two sequences of families of random matrices with coefficients in $\mathbb{K}$. Let $a_{1},\ldots,a_{n}$ and $b_{1},\ldots,b_{n}$ be two families of elements of a non commutative probability space $\left(\mathcal{A},\tau \right)$. Assume that the convergence in non-commutative distribution
	\begin{equation*}
		\left(A_{N,1},\ldots,A_{N,n} \right) \rightarrow (a_{1},\ldots,a_{n}) \textrm{ and } \left(B_{1,N},\ldots,B_{N,n} \right) \rightarrow \left(b_{1},\ldots,b_{n} \right)
	\end{equation*}
hold. Assume also that for all $N$, given a random matrix $U$ distributed according to the Haar measure on $\mathbb{U}(N,\mathbb{K})$ and independent of $\left(A_{N,1},\ldots, A_{N,n},B_{N,1},\ldots,B_{N,n} \right)$, the two families $(A_{N,1},\ldots,\\ $ $A_{N,n}, B_{N,1},\ldots,B_{N,n} )$ and $\left(UA_{N,1}U^{-1}, \ldots, UA_{N,n}U^{-1},B_{N,1},\ldots,B_{N,n} \right)$ have the same distribution. Then the families $\{a_{1},\ldots,a_{n}\}$ and $\{b_{1},\ldots,b_{n}\}$ are free.
\end{theorem}
Recall that we endowed the Dual Voiculescu group $\mathcal{O}\langle n \rangle$ with a coproduct $\Delta$ with value in the free product $\mathcal{O}\langle n \rangle$, a counit $\varepsilon$ and an antipode $S$ that makes $\mathcal{O}\langle n \rangle$ into a Zhang's algebra (see \cite{nico1}).
The two parameters family $\left(\unitaryfd(s,t) \right)_{s,t \geq 0}$ of increments of the process $\left(\unitaryfd(t)\right)_{t \geq 0}$ is defined as
$$
\unitaryfd(s,t) = \left(\unitaryfd(t) \sqcup \left(\unitaryfd(s) \circ S \right)\right) \circ \Delta.
$$
and the increments $\left(U^{\langle n \rangle}(s,t) \right)_{0 \leq s \leq t \leq +\infty}$ of the free $n$ dimensional unitary Brownian motion satisfy
\begin{equation*}
	U^{\langle n \rangle}(s,t) = \left( U^{\langle n \rangle}(t) \sqcup \left( U^{\langle n \rangle}(t) \circ S \right)\right)\circ \Delta.
\end{equation*}
The $U(d)$ invariance of the finite dimensional marginals of $\unitaryfd(s,t)$ combined with the following result of asymptotic freeness stated in Theorem \ref{collins_sniady} are the last two ingredients that end the proof of Theorem \ref{maintheoremsquare}. In fact, we show by recurrence that as $d \rightarrow +\infty$, for any tuples $s_{1} < t_{1} \leq s_{2} < t_{2} \cdots s_{p} < t_{p}$, the random variables
\begin{equation*}
	\unitaryfd(s_{1},t_{1}), \ldots,\unitaryfd(s_{p},t_{p}) \text{ are asymptotically free.}
\end{equation*}
 Let $s_{0} < s_{1}$, then $\unitaryfd(s_{0},s_{1})$ has the same non-commutative distribution as $\unitaryfd(s_{1}-s_{0})$. Thus, $\unitaryfd(s_{0},s_{1})$ converges to $U_{n}(s_{0},s_{1})$. Pick a two tuples of time such that $0 < s_{0} < s_{2} < \cdots < s_{p}$. Assume that the family $\{\unitaryfd(s_{0},s_{1}),\ldots,\unitaryfd(s_{p-2},s_{p-1})\}$ converges to $\{U^{\langle n \rangle}(s_{0},s_{1}),$ $\ldots,U^{\langle n \rangle}(s_{p-2},s_{p-1})\}$ in non-commutative distribution. We proved that $\unitaryfd(s_{p-1},s_{p})$ converges to $U^{\langle n \rangle}(s_{p-1},s_{p})$. Besides, the law of $\unitaryfd(s_{p-1},s_{p})$ is invariant by conjugation by any element of $U(d)$ and $\unitaryfd(s_{p-1},s_{p})$ is independent from the random variables $\{\unitaryfd(s_{0},s_{1}), \ldots, \unitaryfd(s_{p-2},s_{p-1})\}$ thus an application of Theorem $\ref{collins_sniady}$ shows that
$$
\{\unitaryfd(s_{0},s_{1}),\ldots,\unitaryfd(s_{p-1},s_{p})\} \underset{d \rightarrow +\infty}{\overset{n.c}{\rightarrow}} \{ \{\freeunitary(s_{0},s_{1}),\ldots,\freeunitary(s_{p-1},s_{p})\}.
$$

\section{Rectangular extractions of an unitary Brownian motion}
\label{rectangularextractions}
{
In that section, we extend the result we proved in the last section stating the convergence
in non-commutative distribution of square blocks extracted from an unitary matrix by allowing these blocks to be rectangular. But first, we briefly expose amalgamated non-commutative probability theory.
}

\subsection{Operator valued probability theory}
In this section, we make an overview of operator valued probability theory. Categorical notions are used without recalling them, for brevity. The reader can refer to the first chapter in which he will find a detailed exposition on Zhang algebras, categorical independance, categorical coproduct and comodule algebras.
\subsubsection{Involutive bimodule Zhang algebras}
\par In the sequel, algebras are complex or real unital algebras. Let $A$ and $B$ two algebras. Let $R$ be a third algebra and assume that $A$ and $B$ are $R$-bi-modules; there exists a left and a right action commuting which each other such that:
\begin{equation*}
(rr^{\prime})a = r(r^{\prime}a),~a(rr^{\prime}) = ar(r^{\prime}),~ 1a=a,~r(ar^{\prime}) = (ra)r^{\prime},~ r \in R,~a \in A
\end{equation*}
In this work, we mainly deal with involutive algebras. An involutive algebra $C$ is endowed with an involutive anti-morphism $\star_{A}: C \to C$ that is linear if $C$ is a real algebra, anti-linear if $C$ is complex. We assume the three algebras $A$, $B$ and $R$ to be involutive and the following compatibility condition between the bi-module structure and the anti-morphisms $\star_{R}$ and $\star_{A}$:
\begin{equation*}
\star_{A}(r\cdot a) = \star_{A}(a) \cdot \star_{R}(r),~\star_{A}(a\cdot r) =  \star_{R}(r)\cdot \star_{A}(a),~r \in R,~a \in A,
\end{equation*}

A significant construction is the amalgamated free product of $A$ and $B$ over $R$, denoted by $A\sqcup_{R} B$. This amalgamated free product turns the category of involutive bi-module algebras into an algebraic category, more on this is explained at the end of the paragraph and in \cite{nico1}.

The amalgamated free product is, to put it in words, the free product of $A$ and $B$ in which we forget from which algebra the letters that belong to $R$ comes from. For instance, if $a\in A$, $b\in B$ and $r \in R$: $a(rb) = (ar)b \in A\sqcup_{R}B$. In symbols, the amalgamated free product is the quotient:
\begin{align*}
\hspace{0.8cm}A\sqcup_{R} B&=(R\oplus \bigoplus_{n\geq 1}T^{n}(A\oplus B))/(ar\otimes r'a'-arr'a', br\otimes r'b'-brr'b', ar\otimes b-a\otimes rb,\\
&\hspace{1.7cm} br\otimes a- b\otimes ra,r1_{A}r'-rr',r1_{B}r'-rr' : a,a'\in A, b,b'\in B,r,r'\in R).
\end{align*}

The free product $A\sqcup_{R} B$ is endowed in a canonical way with a $R$-bi-module structure, since $R \subset A\sqcup_{R}B$ so that $R$ acts by left and right multiplication on $A\sqcup_{R}B$. In addition, the two star morphisms $\star_{A}$ and $\star_{B}$ induce a morphism $\star_{A\sqcup_{R}B}$ on the amalgamated free product $A\sqcup_{R}B$ which makes the diagram Fig.\ref{stellarfreeproduct} commutative.

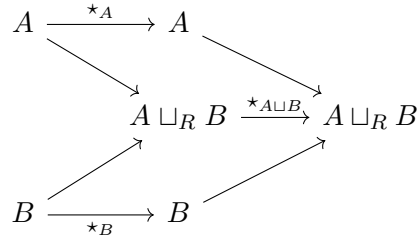
\begin{figure}[!h]
\centering
\begin{tikzcd}
A \arrow{rd} \arrow["\star_{A}"]{r} & A \arrow{rd}  \\
& A \free_{R} B \arrow["\star_{A\sqcup B}"]{r} & A\sqcup_{R}B       \\
B \arrow{ru} \arrow["\star_{B}",swap]{r} & B \arrow{ru}
\end{tikzcd}
\caption{\label{stellarfreeproduct}\small The amalgamated free product as an involutive algebra.}
\end{figure}

The category of involutive bi-modules algebras over $R$ is algebraic, with $R$ as initial object. It is therefore meaningful to define Zhang algebras in this category (see \cite{nico1}) and subsequently the notion of quantum processes in this category. Let us put this discussion into a more formal way by writing the definition of an involutive $R$-bi-module Zhang algebra.

\begin{definition}
Let $R$ be a unital algebra. A involutive $B$-bi-module Zhang algebra is a quadruplet $(H,\Delta,S,\varepsilon)$ with $H$ an involutive unital algebra that is also an involutive $R$-bi-module andv three structural maps:
\begin{equation*}
\Delta: H \to H \free_{R} H,~\varepsilon: H \to B,~S:H \to H.
\end{equation*}
Each of the maps $\Delta,S,\varepsilon$ is a morphism of unital algebra and $R$-bi-module maps. They are subject to the relations:
\begin{enumerate}[\indent 1.]
\item $\Delta(\Delta \free \textrm{id}_{H}) = (\textrm{id}_{H} \free \textrm{id}_{H})\Delta$,
\item $S \free \textrm{id}_{H} = \textrm{id}_{H} \free S = \eta \circ \varepsilon$,
\item $(\varepsilon \free 1)\Delta = (\textrm{id}_{H} \free \varepsilon) \Delta.$
\end{enumerate}
\end{definition}
\par We end this section with a remark on amalgamated tensor product $A \otimes_{R} B$ of two $R$ bimodules. The space $A \otimes_{R} B$ is the quotient:
\begin{equation*}
A \otimes_{R} B = (A \otimes B) / \{ a\otimes (r \cdot b)= (a\cdot r) \otimes b, a \in A, b \in B, r \in R \}
\end{equation*}
The projection in $A\otimes_{R} B$ of a tensor $ a\otimes b \in A\otimes_{R}B$ will be denoted $a\otimes_{R}b$. It is trivial to define a $R$-bi-module structure on $A \otimes_{R} B$, $r\cdot(a\otimes_{R}b)r^{\prime} = (ra)\otimes_{R}(br^{\prime})$.
However, even if $A$ and $B$ are algebras $R$-bimodules, the product $(a_{1}\otimes_{R} b_{1})(a_{2}\otimes_{R}b_{2}) = a_{1}a_{2} \otimes_{R}b_{1}b_{2}$ is, in general, not associative. In fact,
\begin{equation*}
a_{1} (rb) a_{2} = a_{1}a_{2} \otimes_{R} rb \neq a_{1}ra_{2} \otimes_{R} b = (a_{1}r)ba_{2},~a_{1},a_{2} \in A,~b \in B.
\end{equation*}
Since $A\otimes_{R} B$ is not a $R$ bi-module algebra, we cannot define an $R$ bi-module algebra morphism from the amalgamated free product $A\free_{R}B$ to $A\otimes_{R} B$ equal to identity on $A$ and on $B$ as solution of an universal problem. Assume that $R$ is commutative and denote by comBiModAlg$(R)$ the category of commutative $R$-bimodule algebras which left and right module structures are equal. If $A,B \in \text{comBiModAlg}(R)$, the aforementioned issue disappears: the natural product on $A \otimes_{R}B$ turns this space into an $R$-bimodule structure.
\par A probabilistic implication of the last discussion needs to be clarified. Amalgamated tensor independence, as an amalgamated counterpart of classical (from the point of view of non-commutative probability with $R=1$) tensor independence cannot be defined in the category of operator valued probability spaces (see below for the definition of such spaces).

\subsubsection{Operator valued probability spaces, rectangular probability spaces}
Let $R$ be an unital involutive algebra.
If $A$ and $B$ are two involutive $R$ bimodule that are $R$-valued probability space, an $R$-valued random variable from $A$ to $B$ is a morphism of the category of bi-module involutive algebras.
Having introduced the notion of amalgamated Zhang algebra in the last section, its now meaningful to talk about amalgamated quantum processes. However, to define the notion on non-commutative distribution, we need an appropriate definition of expectation on bimodule algebras.
\begin{definition}
Let $A$ a bi-module algebra over $R$. A conditional expectation $E$ on $A$ is a positive involutive $R$-bi-module map $E:A \to R$:
\begin{enumerate}[\indent 1.]
\item $E(bab^{\prime}) = bE(a)b^{\prime}$ ($R$-bi-module map),
\item $E(a^{\star}) = E(a)^{\star}$ (involutive map),
\item $E(aa^{\star}) \geq 0$ (positivity).
\end{enumerate}
\end{definition}

An $R$-valued probability space is the data of a bimodule algebra $A$ and a conditional expectation $E$ on $A$. We introduce as of now the class of $R$ valued probability spaces we are interested in, the rectangular probability spaces. Let $A$ a bi-module algebra and assume from now that $A$ contains a set of complete, mutually orthogonal projectors $\{p_{1},\ldots,p_{n}\}$:
\begin{equation*}
\sum_{p\in \{p_{1},\ldots,p_{n}\}} p = \mathrm{id}_{A},~p^{\star} = p, ~p^{2} = p, ~pq = 0,~ p\neq q \in \{p_{1},\ldots,p_{n}\}.
\end{equation*}
and that $R = \langle\{p_{1},\ldots,p_{n}\}\rangle$, the commutative algebra generated by the projectors. Each element $x$ of $A$ can be written as a matrix, since $x = \sum_{i,j = 1}^{n} p_{i}xp_{j}$, we adopt the notation:
$\lmss{x} = (p_{i}xp_{j})_{i,j \leq n} \in \mathcal{M}_{n}(A)$. We use the terminology \emph{compressed spaces} for the vector spaces $A_{ij} = p_{i}Ap_{j}$, $1 \leq i,j \leq n$. To construct a conditional expectation on $A$, we assume further that each the diagonal compressed algebras $A_{ii},~i\leq n$ is a (usual) probability space and denote by $\phi_{i}$ the expectation on $\mathcal{A}_{ii}$. A conditional expectation $E$ on $A$ is defined by the formula:
\begin{equation*}
E(x) = \sum_{i=1}^{n}\phi_{i}(\lmss{x}_{ii})p_{i},~ x \in A.
\end{equation*}
Having introduced the analogue notion of mean for $R$-valued probability space, we focus now on cumulants. Recall that the set of non-crossing partitions of an interval $\llbracket 1,\ldots,k \rrbracket$ is denoted $\lmss{NC}(k)$. In the sequel, we are handling multilinear bimodule maps over $A$, these maps are defined naturally on tensor products of the algebra $A$ with itself over $R$, for which we use the symbol $A\otimes_{R}\cdots \otimes_{R} A$. By definition: $a\otimes_{R}(ra^{\prime}) = (ar)\otimes_{R} a^{\prime} $, $a,a^{\prime} \in A,~r\in R$.
For $n\geq 1$ an integer, define the map $E_{n}: A^{\otimes_{R}n} \to A$, by $E_{n}(a_{1}\otimes_{R}\cdots \otimes_{R}a_{n}) = E(a_{1}\cdots a_{n})$, $a_{1},\ldots,a_{n} \in A^{\otimes_{R}n}$.
\par Let $\pi \in \lmss{NC}(k)$ a non-crossing partition. The map $E_{\pi}$ from $A\otimes_{R}\cdots\otimes_{R}A$ is defined recursively as follows.
For two blocks $V$ and $W$ of $\pi$, we write $V < W$ if there exists two integers $i,j$ in $W$ such that $V \subset [i,j]$. Denote by $V_{1},\ldots,V_{p}$ blocks of $\pi$ that are maximal for $\pi$. Let $m\leq p$ an integer and write $V_{m} = \{i_{1}^{m} < \cdots < i_{k^{m}}^{m}\}$, the partition $\pi$ restricts to a non-crossing partition $\pi^{m}_{l}$ of the interval $\rrbracket i_{l}, i_{l+1} \llbracket $. The family of functions $(E_{\pi})_{\pi \in \lmss{NC}}$ is defined recursively by the equation
\begin{equation*}
\begin{split}
E_{\pi}(a_{1} \cdots a_{k}) = E_{\sharp V_{1}}\Big(a_{i_{1}^{1}} E_{\pi_{1}^{1}}\left(a_{\rrbracket i_{1}^{1},i_{k^{1}}^{1} \llbracket}\right)&\cdots E_{\pi_{k^{1}-1}^{1}}\left(a_{\rrbracket i_{k^{1}-1}^{1},i^{1}_{k^{1}}\llbracket}\right)a_{i_{k^{1}}^{1}} \Big) \cdots \\ &E_{\sharp V_{p}}\Big(a_{i_{1}^{p}}E_{\pi_{1}^{p}}\left(a_{\rrbracket i_{1}^{p},i_{2}^{p} \llbracket}\right)\cdots E_{\pi_{k^{p}-1}^{p}}\left(a_{\rrbracket i^{p}_{k^{p}-1},i^{p}_{k^{p}}\llbracket}\right) a_{i_{k^{p}}^{p}}\Big)
\end{split}
\end{equation*}
and the initial condition $E_{\emptyset} = 1$. Recall that we use the notation $\mu$ for the M\"oebius function of $\lmss{NC}(k)$. The $R$-valued cumulants $(c_{\pi}:A\otimes_{R}\cdots\otimes_{R}A \to R)_{\pi \in \lmss{NC}(k)}$ are obtained by M\"oebius transformation:
\begin{equation*}
c_{\pi} = \sum_{\gamma \leq \pi} \mu(\gamma,\pi)E_{\gamma},~ E_{\pi} = \sum_{\gamma \leq \pi} c_{\gamma}.
\end{equation*}
As we shall see below, amalgamated freeness (defined below) is most efficiently seen on cumulants, which is the main reason for introducing them. Since cumulants and conditional expectations are obtained from each others by a linear transformation, they share a lot of properties.
\par First, for any non-crossing partition $\pi$, $c_{\pi}$ is a $R$-bi-module map. Secondly, let $k\geq 1$ an integer and $i_{1},\ldots,i_{k}$, $j_{1},\ldots,j_{k}$ two $k$-tuples of integers in $\llbracket 1,n \rrbracket$. The conditional expectation $E_{k}$ is equal to zero on the space $A_{i_{1},j_{1}} \otimes_{R} \cdots \otimes_{R} A_{i_{k},j_{k}}$ if there exists a pair $(i_{l},j_{l+1}),~l\leq n$ such that  $i_{l}\neq j_{l+1}$. It can be proved, by induction, that for any non-crossing partition $\pi$, $E_{\pi}$ evaluates to zero on $A_{i_{1},j_{1}}\otimes_{R} \cdots \otimes_{R} A_{i_{k},j_{k}}$ if there exists a block $\{l_{1},\ldots,l_{p}\} \in \pi$ such that $i_{l_{1}-1}\neq i_{l_{p}}$. This property of the conditional expectation is shared with the cumulants.
\par Finally, $c_{k}(a_{1}\otimes_{R} \cdots \otimes_{R} a_{k}) = 0$ if there exists an integer $i \leq k$ and $r\in R$ such that $a_{i} = r1$. Let us prove this property. To ease the exposition, we assume that $a_{1}=r1$ for some $r\in R$. Define $\tilde{c}_{\pi}(a_{1} \otimes_{R} \cdots \otimes_{R} a_{k}) = c_{\pi}(a_{1}\otimes_{R} \cdots \otimes_{R}a_{k})$ if $\{1\} \in \pi$ and set $\tilde{c}_{\pi}(a_{1}\otimes_{R} \cdots \otimes_{R}a_{k}) = 0$ otherwise. We claim that
\begin{equation*}
E_{\pi}(a_{1}\otimes_{R} \cdots \otimes_{R} a_{k}) = \sum_{\gamma \leq \pi} \tilde{c}_{\gamma}(a_{1}\otimes_{R}\cdots\otimes_{R}a_{k})
\end{equation*}
This last relation, obviously, holds if $\{1\} \in \pi$. Let $V=\{1 < i_{1}\cdots < i_{p}\} \in \pi$ the block of $\pi$ containing $1$. Using $R$ linearity, $E_{\pi}(a_{1}\otimes_{R}\cdots\otimes_{R}a_{k}) = E_{\tilde{\pi}}(a_{1}\otimes_{R}\cdots \otimes_{R}a_{k})$ with $\tilde{p}$ the partition obtained from $\pi$ by splitting the block $V$ of $\pi$ into the two blocks $ \{1\},\{i_{1}<\cdots<i_{p}\}$. Hence,
\begin{equation*}
\begin{split}
E_{\pi}(a_{1}\otimes_{R}\cdots\otimes_{R}a_{k}) &= E_{\tilde{\pi}}(a_{1}\otimes_{R}\cdots\otimes_{R}a_{k}) \\ &=\sum_{\gamma \leq \tilde{\pi}} c_{\gamma}(a_{1}\otimes_{R}\cdots\otimes a_{k}) = \sum_{\gamma \leq \pi} \tilde{c}_{\gamma}(a_{1}\otimes_{R}\cdots_{R}\otimes a_{k}).
\end{split}
\end{equation*}
From which if follows that $\tilde{c}_{\pi}(a_{1}\otimes_{R}\cdots\otimes_{R}a_{k}) = c_{\pi}(a_{1}\otimes_{R}\cdots\otimes_{R}a_{k})$ and finally $\tilde{c}_{k}(a_{1}\otimes_{R}\cdots\otimes_{R}a_{k})=0$ if $k\geq 2$.
\begin{proposition}[see \cite{speicher1998combinatorial}]
\label{nullmixedcumulants}
Let $A$ be an involutive algebra endowed with a bimodule action of an algebra $R$. Let $x_{1}$ and $x_{2}$ two elements in $A$. The two $R$-bimodule algebras $R\left[a\right]$ and $R\left[b\right]$ are free from each other if and only if for all $n\geq 1$ the cumulants $c_{n}(x_{i_{1}},\ldots,x_{i_{n}})$ are null if there exists two integers $1 \leq k,q \leq n$ with $i_{k}\neq i_{q}$.
\end{proposition}
\subsubsection{Amalgamated semi-groups and L\'evy processes}
Let $R$ a unital associative algebra. Equivalent notions for tensor semi-groups and free semi-groups on involutive bi-algebras can be defined in the context of amalgamated bi-algebras. Let $\mathcal{C}$ be either the algebraic category of $R$-bimodule algebras, $\left(\text{biModAlg}(R), \free_{R}, R\right)$, either the algebraic category of commutative $R$-bimodule algebras $\left(\text{comBiModAlg}(R),\otimes_{R}, B\right)$ and recall that these two categories are algebraic if endowed, respectively, with the amalgamated free product or the amalamated tensor product.
The notion of bi-algebra in $\text{comBiModAlg}(R)$ (in the case $R$ is commutative) is obtained by replacing the free amalgamated product by the tensor product in the definition of a free amalamated Zhang algebra and removing the antipode $S$ from the set of structural morphism.

Let $(B,\Delta,\varepsilon)$ be an associative bi-algebra in biModAlg$(R)$, and $\alpha:B \to R$, $\beta:B\to R$ two $R$-bimodule linear maps. The free product $\alpha\,\free_{R}\,\beta \in (B\sqcup_{R}B)^{\star}$ of $\alpha$ and $\beta$ is the unique $R$-bimodule map satisfying the condition: for any alternating word $s_{1}s_{2} \cdots s_{m} \in B\free_{R} B$,
\begin{equation*}
(\alpha\free_{R}\beta)(s_{1}s_{2} \cdots s_{m}) = 0,\alpha(s_{i}) = 0 \text{ if } s_{i} \in B_{|1} \text{ or } \beta(s_{i}) = 0 \text{ if } s_{i} \in B_{|2}, 1 \leq i\leq m.
\end{equation*}

On the commutative side, if $B$ is an object of the category comBiModAlg$(R)$, the tensor product of $\alpha$ and $\beta$ is an $R$-bimodule map on the amalgamated tensor product $A\otimes_{R}B$ defined by:
\begin{equation*}
(\alpha \,\hat{\otimes}_{R}\,\beta)(b_{1}\otimes_{R}b_{2}) = \alpha(b_{1})\beta(b_{2}),~b_{1} \otimes b_{2} \in B\otimes_{R} B.
\end{equation*}

The free convolution product $\alpha\,\hat{\free}_{R}\,\beta$ of $\alpha$ and $\beta$ is an $R$-bimodule map on $B$ and is defined by $(\alpha\,\hat{\free}_{R}\,\beta) = \left(\alpha \free_{R} \beta\right)\circ \Delta$. Again, in case $B$ is an object of comBiModAlg$(R)$, we can define the tensor convolution product $\alpha \hat{\otimes} \beta$ of $\alpha$ and $\beta$ by setting: $\alpha \,\hat{\otimes}_{R}\,\beta = (\alpha \otimes_{^R} \beta)\circ \Delta $.

An amalgamated free semi-group is a time parametrized family $(E_{t})_{t\geq 0}$ of $R$-bimodule maps on $B$ satisfying:
\begin{equation*}
E_{t+s} = E_{t}\,\hat{\free}_{R}\,E_{s},~t,s \geq 0.
\end{equation*}
The notion of amalgmated tensor semi-group is obtained be replacing the amalamated free convolution product in the last equation by the tensor amalgamated product.
We now expose the amalmagmated counterpart of free independence. Let $(A,E)$ a $R$ valued probability space. Let $B_{1}$ and $B_{2}$ two $R$-bimodule sub-algebras of $A$. We say that $B_{1}$ and $B_{2}$ are freely independent with amalgamation if:
\begin{equation}
E(b^{1}_{1}b^{1}_{2} ... b^{p}_{1}b^{p}_{2}) = 0,~\text{with } b^{i}_{1}\in B_{1},~b_{2}^{i}\in B_{2}\text{ and, } E(b_{1}^{i}) = E(b_{2}^{i}) = 0,~1\leq i \leq p.
\end{equation}
Working in the category $biModAlg(R)$,we say that two sub-$R$-bimodule algebras $B_{1}$ and $B_{2}$ of commutative operator valued probability space $(A,E)$ are \emph{amalgamated tensor independent} if
\begin{equation*}
  E(b_{1}b_{2})=E(b_{1})E(b_{2}).
\end{equation*}
Let $a \in A$ an element of $A$, we denote by $R[a]$ the $R$ bimodule algebra generated by $a$. By definition $R[a]$ is the set of all linear combinations of monomials in the element $a$ with coefficients in the algebra $R$:
\begin{equation*}
R[a] = \mathbb{K}\left[\{r_{0}a^{n_{1}}r_{1}a^{n_{2}}\cdots a^{n_{p}}r_{p}, r_{0},\ldots,r_{p} \in R,~ p \geq 1\}\right] \subset A.
\end{equation*}
We say that two elements $a \in A$ and $b\in B$ are free with amalgamation over $R$ if $R[a]$ and $R[b]$ are free sub-modules of $A$.
\par The restriction of $E$ to the algebra generated by two mutually free with amalgamation algebras $B_{1}$ and $B_{2}$ is entirely determined by the restriction of $E$ to $B_{1}$ and to $B_{2}$ and is equal to the amalgamated free product of the restrictions of $E$ to these algebras; with $\iota_{1}$ and $\iota_{2}$ the injections of, respectively, $B_{1}$ into $B$ and $B_{2}$ into $B$:
\begin{equation*}
\left( B_{1} \free_{R} B_{2}, E_{|B_{1}}\hat{\free}_{R}E_{|B_{2}} \right) \overset{ \iota_{B_{1}} \freem\, \iota_{B_{2}}}{\longrightarrow} (B,E).
\end{equation*}
We are now in position to give the definition of an \emph{amalgamated free L\'evy process}. Let $(H,\Delta,\varepsilon, S)$ be an $R$-amalgamated free Zhang algebra, that is, a Zhang algebra in $\biMod\Alg(R)$ and $(A,E)$ an $R$-valued probability space. An amalamated free L\'evy process $j = (j_{t})_{t\geq 0}$ is a time parametrized collection of homomorphisms of $\biMod\Alg{R}$ with values in $A$ that satisfy the following three conditions, with $j_{st} = j_{t} \freem j_{s} \circ s$:
\begin{enumerate}[\indent 1.]
\item for all times $s_{1} < t_{1} \leq s_{2} < t_{2} \leq \cdots \leq s_{p} < t_{p}$, $\{j_{s_{1}t_{1}},\ldots,j_{s_{p}t_{p}}\}$ is a mutually free with amalgamation family, meaning that the algebras $j_{s_{1}t_{1}}(H),\ldots,j_{s_{p}t_{p}}(H)$ are mutually free in $A$.
\item For all times $t>s$, the distribution $E\circ j_{st}$ depends only on the difference $t-s$,
\end{enumerate}
In addition, if $R$ is naturally endowed with a norm, we require also the continuity condition:
\begin{equation*}
\lim_{s\to t}j_{st}(h) = \varepsilon(h), h \in H.
\end{equation*}
To be complete, if considering the category comBiModAlg$(R)$, the defintion of a tensor amalgamated L\'evy process is obtained by requiring amalagmated tensor independence of the increments in the last definition. With the notation of the definition, the one dimensional marginals $t\mapsto E\circ j_{t}$ of a free amalgamated L\'evy process is a free semi-group (the same holds if working with tensor amalgamated tensor L\'evy process).
We do not know if an amalgamated version of the Schoenberg correspondance for free or tensor amalgamated L\'evy process holds.

\subsection{Extraction processes and their statistics in high dimensions}
\label{extractionprocesses}
\par In the Section \ref{squareextractions}, we proved the convergence in non-commutative distribution of the process on the dual Voiculescu group $\udualgroup$ that extracts square blocks from a unitary Brownian motion in the limit for which the dimensions of these blcks tend to infinity. A natural extension of this result is investigated here; square blocks are replaced by rectangular ones. Our main results are stated in Theorem \ref{thm1rect} and Theorem \ref{thm2cluster}. Given a partition $\lmss{d}$ of the dimension, we construct for each time $t \geq 0$ two quantum processes on two amalgamated Zhang algebras (defined below in Section \ref{cluster} and \ref{rect}) that extract rectangular blocks in the matrix $\munitaryfd(t)$. For one of these processes, product of blocks, even if the dimensions match, may be equal to zero while it is never the case for the other process. More on this point is explained below.
The method we use in this section for proving the convergence of the one dimensional marginals of these processes is similar to the one used in Section \ref{squareextractions}. However, to prove the convergence of the multi-dimensional marginals, Theorem $\ref{collins_sniady}$ can not be applied.
We fix an integer $n\geq 1$. Let $N \geq 1$ an integer and let $\lmss{d}$ be a partition of $N$ into $n$ parts. The algebra $\mathcal{M}_{N}(\mathbb{R})$ can be endowed with a structure of operator valued probability space. In fact, denote by $\mathcal{D}_{\lmss{d}}$ generated by the projectors:
\begin{equation*}
	p_{i}(k,l) = \left\{
	\begin{array}{ll}
		\delta_{k,l}~& \textrm{ if } k,l \in [d_{1}+\cdots+d_{i-1},~d_{1}+\cdots+d_{i}] \\
		0 & \textrm{ otherwise }
	\end{array}
	\right.,~i\leq n.
\end{equation*}
and define the complex linear form $E_{\lmss{d}}:\mathcal{M}_{N}(\mathbb{R})\to \mathcal{D}_{\lmss{d}}$ by
\begin{equation*}
E_{\lmss{d}}\left[A\right] = \sum_{i=1}^{n}\frac{1}{d_{i}}\lmss{Tr}(p_{i}Ap_{i})p_{i}.
\end{equation*}
The algebra $\mathcal{M}_{N}(\mathbb{R})$ is a $\mathcal{D}_{\lmss{d}}$ bimodule and $E_{\lmss{d}}$ is a positive bimodule map: $(\mathcal{M}_{N}(\mathbb{R}), E_\lmss{d})$ is an operator valued probability space.

If considering matrices with entries in $\mathbb{C}$, the conditional expectation we choose on $\mathcal{M}_{N}(\mathbb{C})$ is the same as for the real case. If we consider matrices with entries in the quaternionic division algebra, we choose for the conditional expectation $E_{\lmss{d}}[A]=\sum_{i=1}^{n}\frac{1}{d_{i}}\mathcal{R}e\lmss{Tr}\left(p_{i}Ap_{i}\right)p_{i}$, $A\in \mathcal{M}_{N}(\mathbb{H})$. For $\mathbb{K}=\mathbb{R},\mathbb{C}$ or $\mathbb{H}$, we denote by $\mathcal{M}_{\lmss{d}}(\mathbb{K})$ the operator valued probability space we just defined. Considering matrices with random entries in a $L^{\infty-}(\Omega,\mathcal{F},\mathbb{P})$, a rectangular probability space amalgamated over $ \mathcal{D}_{\sf d}$ is obtained by taking the mean of the conditional expectation $E_{\lmss{d}}$. Set $\mathbb{E}_{\sf d} = \mathbb{E} \circ E^{\lmss{d}}$. The rectangular probability space $\left(\mathcal{M}_{\sf d},\mathbb{E}_{d},\frac{\sf d}{N}\right)$) is denoted $\left(\mathcal{M}_{\sf d}(L^{\infty-}(\mathbb{K})\right)$.
\par Put $n^{\prime} = \sum_{d \in \lmss{d}} d$. The two sections \ref{cluster} and \ref{rect} are devoted to the definition of the two processes we are interested in and their structure algebras. These two processes depend on a partition $\lmss{d}$. The first one takes its values in $\mathcal{M}_{n^{\prime}}(\mathbb{K})$, extracts blocks of dimensions prescribed by the partition $\lmss{d}$ and puts blocks of same size at the same place into a matrix of dimension $n^{\prime}\times n^{\prime}$. The second one extracts blocks but holds them in place in the matrix, meaning that all the other blocks are set to zero.
\subsubsection{Dimension cluster algebra}
\label{cluster}
Let $N \geq 1$ an integer and pick a partition $\lmss{d}$ of $N$ into $n$ parts. We use the short notation $\{\lmss{d}\}$ for the set $\{d,~d \in \lmss{d}\}$. If $d \in \{\lmss{d}\}$, we denote by $n_{d}$ the number of occurrences of $d$ in $\lmss{d}$ and we set $n^{\prime} = \sum_{d \in \{\lmss{d}\}} d$. We define the first process which asymptotics in high dimensions is the object of study.

Pick two integers $d,d^{\prime} \in \{\lmss{d}\}$. To that pair of dimensions $(d,d^{\prime})$ we associate formal variables $x_{d,d^{\prime}}(i,j),~x^{\star}_{d^{\prime},d}(j,~i)~ i \leq n_{d},~j\leq n_{d^{\prime}}$ and define the matrix $X$ of size $n$ with entries in the free algebra generated by the variables $x_{d,d^{\prime}}(i,j)$ by
\begin{equation}
X(k,l) = x_{d_{k},d_{l}}\left(\sharp\{1\leq i\leq k:d_{i}=d_{k}\}, \sharp\{1 \leq i \leq l:d_{i}=d_{l}\}\right).
\end{equation}
\begin{definition}[Cluster Rectangular Unitary algebra]
	Cluster Rectangular Unitary algebra is the involutive unital associative algebra, denoted $\cluster$, that is generated by all the variables $x_{dd^{\prime}}(i,j)$ and projectors $p_{d},d\in\{\lmss{d}\}$ subject to the relations:
\begin{gather*}
	\label{relationscluster}
XX^{\star} = X^{\star}X = 1,
p_{d}^{\star}=p_{d},~p_{d}p_{d^{\prime}} = \delta_{dd^{\prime}}p_{d},~ \\
p_{d_{1}}x_{d_{2}d_{3}}(i,j)p_{d_{4}} = \delta_{d_{1}d_{2}}\delta_{d_{3}d_{4}}x_{d_{2}d_{3}}(i,j),~p_{d_{1}}x^{\star}_{d_{2}d_{3}}(i,j)p_{d_{4}} = \delta_{d_{1}d_{3}}\delta_{d_{4}d_{2}}x_{d_{2}d_{3}}(i,j).
\end{gather*}
\end{definition}
As a consequence of the set of relations \ref{relationscluster}, the family of projectors $\{p_{d},d\in\{\lmss{d}\}$ for a complete set of projectors:
\begin{equation*}
	\sum_{d\in\{\lmss{d}\}}p_{d}=1 \in \cluster.
\end{equation*}
To a pair of dimensions $(d,d^{\prime})$ we associate the sub-algebra $\cluster_{d,d^{\prime}}$ generated by the set of variables $\{x_{d,d^{\prime}}(i,j), x^{\star}_{d^{\prime},d}(j,i),~ i \leq n_{d},~ j \leq n_{d^{\prime}}\}$. The unital algebra generated by the projectors $\{p_{d}, d \in \{\lmss{d}\}\}$ is denoted $\mathcal{D}_{\lmss{d}}$. Finally, we define a $\mathcal{D}_{\lmss{d}}$-amalgamated Zhang algebra structure on $\cluster$ by setting for the structural morphisms $S,\Delta,\varepsilon$:
\begin{align}
&\Delta: \cluster \to \cluster \sqcup_{\mathcal{D}_{\lmss{d}}} \cluster,~&&\Delta(X) = X_{|1}X_{|2},\nonumber \\ \nonumber
&S:\cluster \to \cluster,~&&S(X) = X^{\star}, \\
&\varepsilon: \cluster \to D_{\lmss{d}},~&&\varepsilon(X) = 1.\nonumber
\end{align}
\par We denote by $\lmss{d}^{\prime}$ the partition of $n^{\prime}$ obtained by sorting $\{\lmss{d}\} = \{d_{1},\ldots,d_{p}\}$ in ascending order. As we should see below, this choice for the partition $d^{\prime}$ is quite arbitrary and is, somehow, maximal.

\par To a matrix $A \in \mathcal{M}_{n}(\mathbb{K})$, we associate a random variable $j_{A} : \cluster\rightarrow \mathcal{M}_{\lmss{d}^{\prime}}(\mathbb{K})$ defined as follows. For $d,d^{\prime} \in \{\lmss{d}\}, i \leq n_{d},j \leq n_{d^{\prime}}$: $j_{A}(x_{d,d^{\prime}}(i,j))$ is the block of dimensions $d \times d^{\prime}$ at position $(i,j)$ in the matrix of size $n_{d}d\times n_{d^{\prime}}d^{\prime}$ obtained by extracting all blocks of size $d \times d^{\prime}$ (and keeping them at their relative position) of $A$.

We make few remarks regarding the way we chose to define the algebra cluster and the random variables $j_{A}^{\lmss{d}}$, with obvious notations. We indexed the generators of $\cluster$ by using, in particular, the set $\{\lmss{d}\}$. It is important to notice that the algebra $\cluster$ depends solely on the kernel of $\lmss{d}$, if another sequence $\lmss{d}_{2}$ of integers of length $n$ has the same kernel as $\lmss{d}$ then $\cluster = \mathcal{C}\mathcal{R}\mathcal{O}\langle \lmss{d}_{2}\rangle$. Only the random variabl $j^{\lmss{d}}_{A}$ depends on the sequence $\lmss{d}$.

If $V = \{i_{1} < \ldots < i_{p}\}$ and $V^{\prime}=\{i^{\prime}_{1} < \ldots < i^{\prime}_{p^{\prime}} \}$ are two blocks of $\lmss{Ker}(\lmss{d})$, we denote by $x_{V,V^{\prime}}(i_{k},i^{\prime}_{l})$ the element $x_{\lmss{d}(V)}(k,l)$, and accordingly $p_{d(V)} = p_{V}$. The constitutive relations of $\cluster$ can thus be written:
\begin{gather}
\label{constitutiverelation2}
XX^{\star} = X^{\star}X = 1, p_{V}p_{V^{\prime}} = \delta_{V,V^{\prime}}p_{V}, p^{\star}_{V}=p_{V} \\
p_{V_{1}}x_{V,V^{\prime}}(i,j)p_{V_{2}}=\delta_{V_{1},V}\delta_{V_{2},V^{\prime}}x_{V,V^{\prime}}(i,j), i \in V,j\in V^{\prime}
\end{gather}
If $\pi$ is a partition of $\llbracket 1, n \rrbracket$, we denote by $\mathcal{C}\mathcal{R}\mathcal{O}\langle \pi \rangle$ the algebra generated by the random variables $\{x_{V,V^{\prime}}(i,j), V,V^{\prime} \in \pi,~i \in V,~j \in V^{\prime}\}$ subject to the relations \eqref{constitutiverelation2}.
\par Let $A$ be a matrix of size $N$ and $\lmss{d}$ a partition of $N$ into $n$ parts which kernel is coarser than $\pi$. Set $p$ equal to the number of blocks of $\pi$ and pick $\sigma$ a permutation of $\mathfrak{S}_{p}$. Using the lexicographic order on the block of $\pi$, we write $\pi = \{V_{1} < \ldots < V_{p} \}$ and define $j^{\lmss{d},\pi,\sigma}_{A}$ the random variable that takes it values in the rectangular probability space $\mathcal{M}_{(d(V_{\sigma(1)}),\ldots d(V_{\sigma(p)}))})$ and defined in the same way as $j^{\lmss{d}}_{A}$. It is clear that the random variable $j^{\lmss{d}}_{A}$ we defined previously corresponds to the choice $\pi=\lmss{Ker}(\lmss{d})$ for some permutation $\sigma$. We settle quite a level of generalities, in the sequel we use only the random varirable $j_{A}^{\pi,\lmss{d},1_{p}}$ for which we use the shorter notation $j^{\pi,\lmss{d}}$. After all all these defintions, we can now define the process which asymptotic asymptotic in high dimensions is studied:
\begin{equation}
U^{\mathbb{K}}_{\{\lmss{d}\},\pi}(t)=j^{\lmss{d},\pi}_{\munitaryfd(t)},~\text{ for all times } t \geq 0.
\end{equation}
\begin{figure}[!htb]
	\includegraphics[scale=0.5]{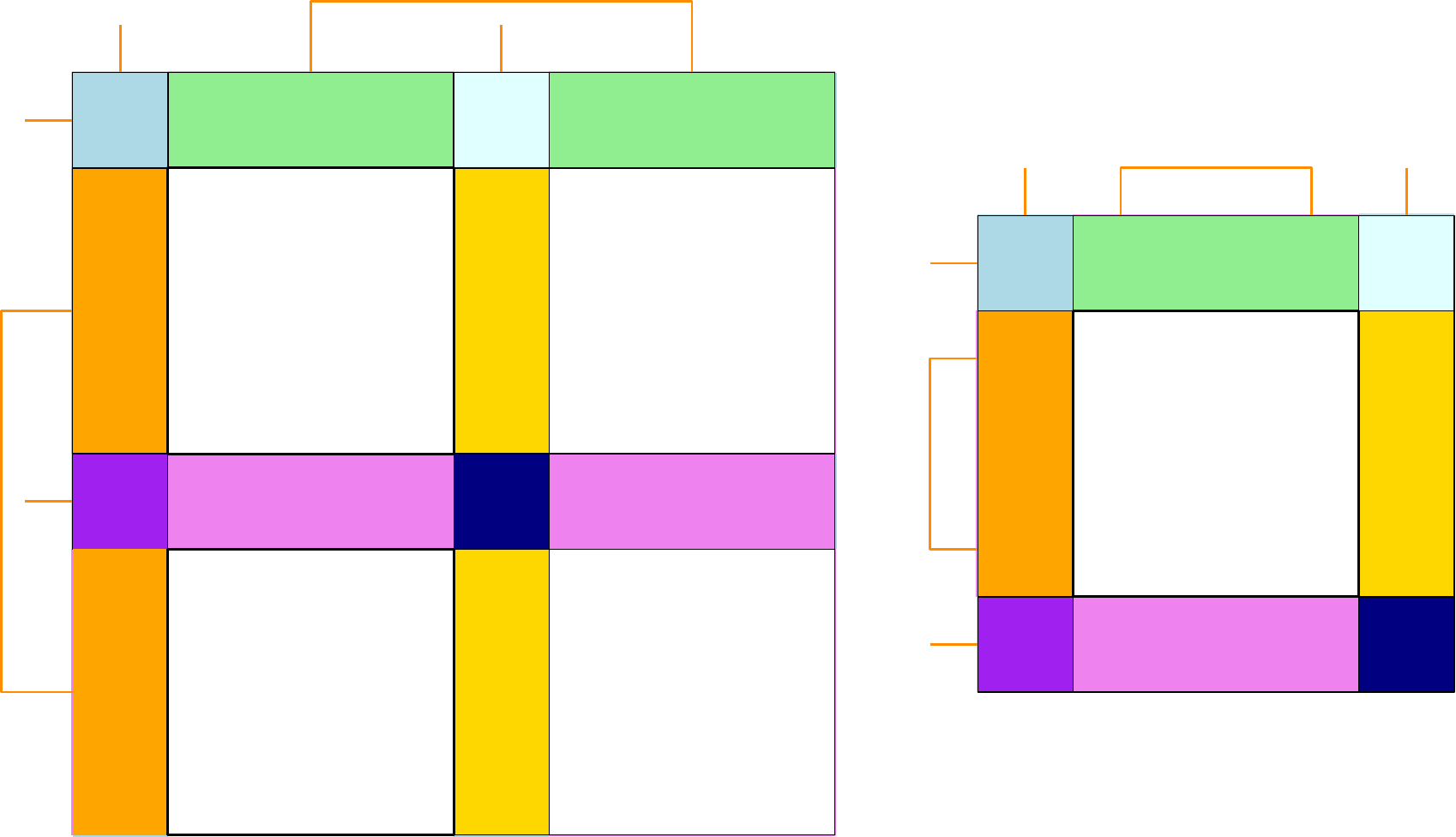}
\caption{\label{rearrangement} This figure pictures the action of the morphism $j_{A}^{\lmss{d},\pi,\id_{4}}$, with $\lmss{d}=(1,3,1,3)$ and $ \pi=\{\{1\},\{2,4\}, \{3\}\}$ on a matrix $A$, drawn on the left hand side of the figure. Blocks of $A$ coloured in the same way are sent by $j_{A}^{\lmss{d},\pi,\id_{4}}$ on the block in the matrix on the right hand side coloured with their common colour.}
\end{figure}
\subsubsection{Rectangular unitary algebra}
\label{rect}
Let $N,n\geq 1$ integers and $\lmss{d}$ a partition of $N$ into $n$ parts. In that section, we define the rectangular unitary algebra $\runitaryalg$ that is nothing more but the dual Voiculescu group augmented with auto-adjoint projectors. Hence, as an unital algebra, $\runitaryalg$ is generated by $n$ auto-adjoint mutually orthogonal idempotent elements $\left(p_{i},  i \leq n \right)$ and an unitary element $u$ subject to the relations:
\begin{equation}
\label{rrelations}
	p_{k}p_{l} = 0,~k\neq l, \quad p_{k}^{\star} = p_{k}, \quad p_{k}^{2} = p_{k}, \quad k,l \leq n, \mathrm{~and~} u^{\star} = u^{-1}.
\end{equation}
Denote by $\mathcal{D}_{n}$ the algebra generated by the projectors $p_{i},~i \leq n$. The algebra $\runitaryalg$ is an involutive $\mathcal{D}_{n}$ bimodule algebra. Let $\pi$ be a partition finer than $\lmss{Ker}(\lmss{d})$. At that point, let us draw comparisons between the algebra $\mathcal{C}\mathcal{R}\mathcal{O}(\pi)$ we introduced in the last section and the rectangular unitary algebra. We claim that there exists a surjective morphism $\phi : \mathcal{C}\mathcal{R}\mathcal{O}(\pi) \rightarrow \runitaryalg$ taking the following values on the generators:
\begin{equation*}
\phi(x_{V,V^{\prime}}(i,j)) =  p_{i}up_{j},~ d,d^{\prime} \in \{\lmss{d}\}, \phi(p_{V})=\sum_{q=1}^{\sharp V}p_{q}.
\end{equation*}
with $i \in V,~j \in V^{\prime}, V,V^{\prime} \in \pi$.
For each time $t\geq 0$, the random variable $\runitaryfd(t) : \runitaryalg \rightarrow L^{\infty-}(\Omega,\mathcal{A},\mathbb{P},\mathbb{R}) \otimes \mathcal{M}_{N}(\mathbb{K})$ that extract blocks from $\munitaryfd$ but hold them in place, is defined by:
\begin{equation}
	\label{definition_process}
		\begin{array}{llll}
			\runitaryfd(t): & \runitaryalg & \rightarrow & \mathcal{M}_{N}\left(L^{\infty}\left(\Omega,\mathcal{F},\mathbb{P},\mathbb{K} \right) \right)\\
		&u & \mapsto & \munitaryfd(t) \\
		&p_{i} &\mapsto & p_{i}
	\end{array}.
\end{equation}
\par The algebra $\runitaryalg$ is an involutive $\mathcal{D}_{n}$-bimodule and is endowed with a Zhang algebra structure. We define three bi-module morphisms $\Delta, \varepsilon$, and $S$ by specifying their values on the generators of $\runitaryalg$ and show that $(\runitaryalg, \Delta,\varepsilon,S)$ is a Zhang algebra.
We claim that there exist involutive algebra morphisms $S$,$\Delta$, $\varepsilon$, defined on the free real algebra generated by $\{u,u^{\star},p_{1},\ldots,p_{n}\}$ satisfying:
\begin{align}
\label{threemorphisms}
&\Delta: \mathbb{R}\left[u,u^{\star},p_{1},\ldots,p_{n} \right] \to \runitaryalg \sqcup_{R} \runitaryalg,~&&\Delta(u) = u_{|1}u_{|2},~\Delta(p_{i}) = p_{i},~ 1\leq i\leq n \nonumber\\
&S:\mathbb{R}\left[u,u^{\star},p_{1},\ldots,p_{n} \right] \to \runitaryalg,~&&S(u) = u^{\star},~S(p_{i}) = p_{i},~1 \leq i \leq n \\
&\varepsilon: \mathbb{R}\left[u,u^{\star},p_{1},\ldots,p_{n} \right] \to \mathcal{B},~&&\varepsilon(u) = 1,~ \varepsilon(p_{i}) = p_{i}, i \leq n \nonumber
\end{align}
The algebra $\runitaryalg$ is a quotient of $\mathbb{R}\left[u,u^{\star},p_{1},\ldots,p_{n} \right]$ by the relations \eqref{rrelations}. Hence, the three maps in $\eqref{threemorphisms}$ descend to morphisms on $\runitaryalg$ if the images of the generator $\{u,u^{\star},p_{1},\ldots,p_{n}\}$ by these maps satisfy the same relations \eqref{rrelations}. This verification shows no difficulties and we omit it for brevity.
With this definition, $\runitaryfd$ is non-commutative process on the Zhang algebra $H$ taking its values in the rectangular probability space $\mathcal{M}_{\lmss{d}}(L^{\infty-}(\Omega,\mathcal{F},\mathbb{P},\mathbb{K}))$.
The next sections are devoted to the investigation of the convergence in non commutative distribution of $U_{d_{N}}^{\mathbb{K}}$ as the dimension $N$ tends to infinity, with $(d_{N})_{N\geq 1}$ a sequence of partitions of $N$ into a fixed number $n$ of parts.
\begin{figure}[!htb]
	\includegraphics[scale=0.5]{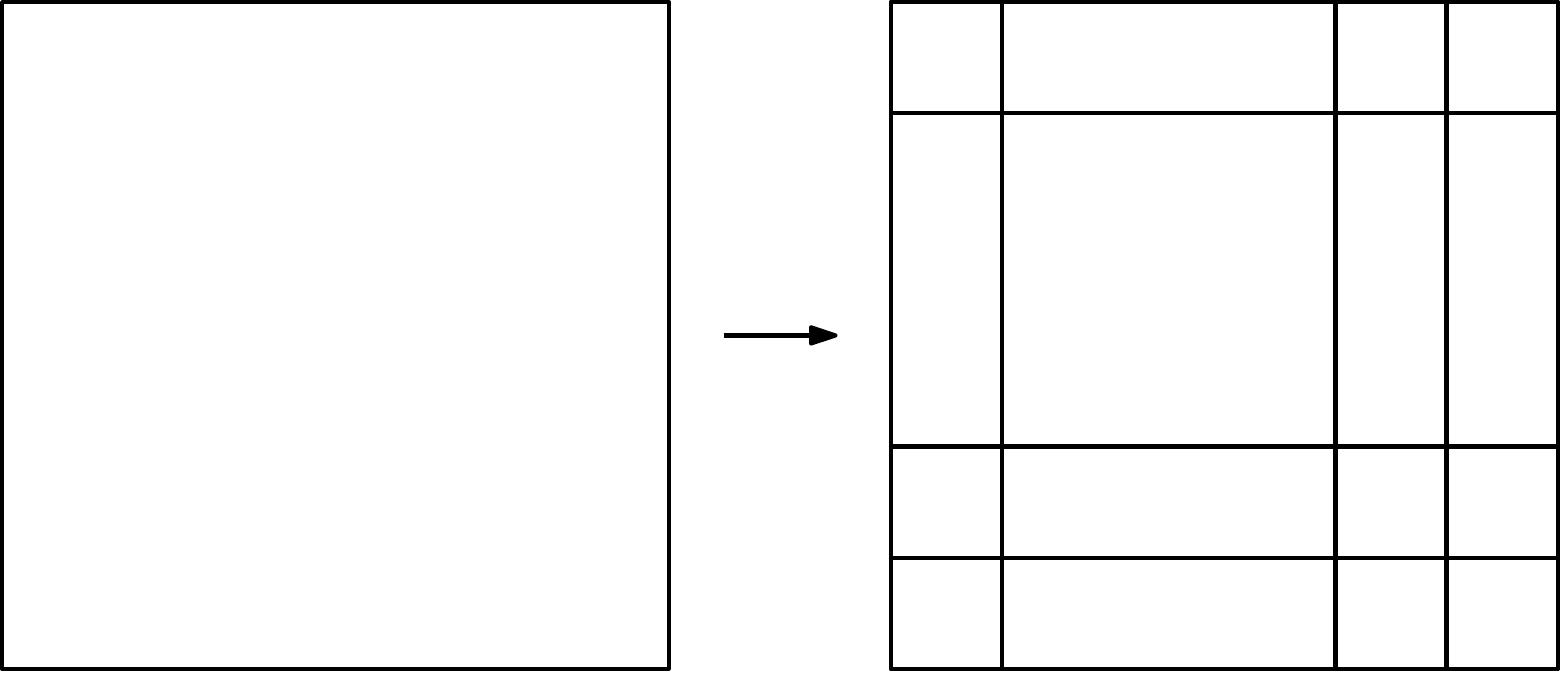}
	\caption{\label{rectpicture} A picture standing for the action of the morphism $j_{A}$ that cuts a matrix $A$(on the left) into $16$ blocks which sizes exhaust the set $\{1,3\}\times\{1,3\}$.}
\end{figure}
\subsubsection{Statistics of the extractions processes}
Let $N \geq 1$ an integer and pick a partition $d_{N}$ into $n$ parts of $N$. The purpose of the forthcoming sections are to prove the convergence in high dimensions, in non-commutative distribution, of the two processes $\rdunitaryfdN$ and $\runitaryfdN$. The method we use has been expounded in Section \ref{squareextractions} to prove the convergence in non-commutative distribution of the process square blocks extraction from an unitary Brownian motion. We recall some features of that method. If $N=nd$ and $d_{N}=(d,\ldots,d)$, we defined  for each time $t\geq 0$ a statistic $\mathbb{m}^{\mathbb{K}}_{d}(t)$ which is a function on the set of Brauer diagrams which range comprises the distribution of $\runitaryfdN$.
We proved next the convergence of this statistic by exhibiting a differential system $\mathbb{m}^{\mathbb{K}}_{d_{N}}$ is solution of. Hence, in order to apply this method, we need first to define the statistic $\lmss{m}^{\mathbb{K}}_{\lmss{d}}$ that is the rectangular counterpart of $\lmss{m}^{\mathbb{K}}_{d}$. After that, to normalize this statistic we use the functions $\lmss{fnc}_{d}$, $d \in \{d_{N}\}$ we introduced in Section \ref{schur_weyl}. To explain why the use of these functions is needed, recall that for a matrix $A \in \mathcal{M}_{N}(\mathbb{R})$ and $(b,s) \in \brauer$ an oriented irreducible Brauer diagram,
\begin{equation}
	\label{cyclic_linear}
	\tag{$\star$}
	\lmss{Tr}^{\otimes k}\left(\rho_{d_{N}}(b) \circ \left(A \otimes \cdots\otimes A\right)\right)
	= \lmss{Tr}^{\otimes k}\left(A\left(c_{b}(i^{\prime}_{1})^{s_{b}(1)}, c_{b}(i_{1})\right) \cdots A\left(c_{b}(i^{\prime}_{k}), c_{b}(i_{k})^{s_{b}(k)} \right) \right)
\end{equation}
with $\sigma_{\ncc{b}} = (i_{1}, \ldots, i_{k})$. Of course,the right hand side of $\eqref{cyclic_linear}$ does only depend on the cyclic order induced by $\sigma_{\ncc{b}}$ and truly independent of the orientation $s$. However, to normalize these quantities (and retrieve the distribution of $\runitaryfdN$) we need first to choose which block sits at front, this amongst to pick a linear order on $(1^{\prime},\ldots,k^{\prime})$. By doing this, we are able to associate a word on the blocks of $A$ to each Brauer diagram, not only a cyclic word.
A second step is to choose the dimension we use to normalize, either the number of lines, either the number of columns of the blocks that sits at front. If this block is transposed, we normalized by the number of columns, and it is not we normalize by the number of lines. By doing this we break an other symmetry of the right hand side of \eqref{cyclic_linear}, which is the invariance by transposition.
Let's draw an example with three blocks and $b = \left((1,2,3), (1,2)(1,1)(1,2)\right)$, one has
\begin{equation*}
	\lmss{Tr}\left(\rho_{\lmss{d}} \circ A \otimes A \otimes A \right) = \lmss{Tr}\left(A^{2}_{1} A^{1}_{1} A^{1}_{2}\right) = \lmss{Tr}\left(A^{1}_{1} A^{1}_{2} A^{2}_{1} \right) = \lmss{Tr}\left(A^{2}_{1}  A^{1}_{1} A^{1}_{2} \right) .
\end{equation*}
For each of three linear orders on $(1,2,3)$, we get the normalizations
\begin{equation*}
	\begin{split}
		& \frac{1}{d_{2}}\lmss{Tr}\left(A^{2}_{1} A^{1}_{1} A^{1}_{2}\right)  ~1<2<3,~ \frac{1}{d_{1}} \lmss{Tr}\left(A^{1}_{1} A^{1}_{2} A^{2}_{1} \right)~1<3<2,~\frac{1}{d_{2}} \lmss{Tr}\left(A^{1}_{1} A^{1}_{2} A^{2}_{1} \right)~3<2<1.
	\end{split}
\end{equation*}
\par The support of a cycle $c$ of a Brauer diagram $\ncc{b} \in \brauerz$ is seen as being endowed with the linear order that is left by putting the minimum, for the natural order, of the support of $c$ out of the cyclic order induced by $c$.
\par After this discussin, the definition of the rectangular extractions' statistic will seem natural for the reader. The function $\lmss{m}^{\mathbb{R}}_{d_{N}}$ on the set of oriented Brauer diagrams $\mathcal{O}\brauer$ and valued in the space of linear forms on $\mathcal{M}_{N}(\mathbb{R})^{\otimes k}$ is defined by, for matrices $A_{1},\ldots,A_{k} \in \mathcal{M}_{N}(\mathbb{K})$, and an oriented Brauer diagram $~(b,s) \in \mathcal{O}\brauer:$
\begin{equation}
\label{defstatreal}
\lmss{m}^{\mathbb{K}}_{d_{N}}((b,s))(A_{1} \otimes \cdots \otimes A_{k}) = \left(\prod_{d\in\{d_{N}\}}d^{-\sf{fnc}_{d}((b,s))} \right)\lmss{Tr}_{\mathbb{K}}(\rho^{\mathbb{R}}_{\sf{d}}(b)\circ A_{1}\otimes \cdots \otimes A_{k})),
\end{equation}
where $\mathbb{K}=\mathbb{R},\mathbb{C}$ or $\mathbb{H}$ and $\lmss{Tr}_{\mathbb{R}}=\lmss{Tr}_{\mathbb{C}}=\lmss{Tr}$, $\lmss{Tr}_{\mathbb{H}}=\mathcal{R}e\circ\lmss{Tr}$.
In the following section, we are making some hypothesis on the sequence of partitions $(d_{N})_{N \geq 1}$ and study the convergence of:
$$\lmss{m}_{d_{N}}^{\mathbb{K}}\big(\mathbb{E}\left[\lmss{m}_{d_{N}}^{\mathbb{K}}(\munitaryfd(s_{1},t_{1})\otimes \cdots \otimes \munitaryfd(s_{q},t_{q})\big)\right]$$ as the dimension $N$ tends to infinity (see \ref{prp:convstat}). To study the aforementioned convergence, we let $q \geq 1$ an integer and pick $\left[\munitaryfd\right]^{(1)},\ldots,\left[\munitaryfd\right]^{(q)}$ independent unitary Brownian motions. We denote by $\lmss{M}_{q}$, respectively $\overline{M}_{q}$ the free monoid generated by $q$ letters $\{x_{1},\ldots,x_{q}\}$, respectively $2q$ letters $\{x_{1},\ldots,x_{q},\bar{x}_{1},\ldots,\bar{x}_{q}\}$ and the identity element $\emptyset$. If $k\geq 1$ is an integer, we denote by $\lmss{M}_{q}(k)$, respectively $\overline{\sf M}_{q}(k)$ the subset of $\lmss{M}_{q}$ , respectively of $\overline{\sf M}_{q}$, comprising all words of length $k$.

If $\mathbb{K} = \mathbb{R}$ or $\mathbb{H}$, we are now defining for a tuple of times $\lmss{t}=(t_{1},\ldots,t_{q})$ , a word $w \in \lmss{M}_{q}(k)$ and an oriented Brauer diagram $(b,s)$:
\begin{equation*}
\begin{split}
\mathbb{m}^{\mathbb{K}}_{d_{N}}((b,s),w,\lmss{t})(u) = \lmss{m}^{\mathbb{K}}_{d_{N}}((b,s)\left(\mathbb{E}\left[w^{\otimes}\left(\left[\munitaryfd(ut_{1})\right]^{(1)}\otimes \cdots \otimes \left[\munitaryfd(ut_{q})\right]^{(q)}\right)\right]\right), u \in [0,1].
\end{split}
\end{equation*}
In the complex case, we need also to take component wise conjugation of the Brownian diffusion in order for its non-commutative distribution to be in the range of the statistic $\mathbb{m}^{\mathbb{C}}(\lmss{t})$, with $w$ a word
in $\overline{\lmss{M}}_{q}(k)$ and an oriented Brauer diagram $(b,s)$, we set:
\begin{equation*}
\begin{split}
\mathbb{m}^{\mathbb{C}}_{d_{N}}((b,s),w,\lmss{t})(u) = \lmss{m}^{\mathbb{C}}_{d_{N}}((b,s))\left(\mathbb{E}\left[w^{\otimes}\left(\left[\munitaryfdc(ut_{1})\right]^{(1)}\otimes \cdots \otimes \left[\munitaryfdc(ut_{q})\right]^{(q)}\right)\right]\right), u \in [0,1].
\end{split}
\end{equation*}
For each time $u \in [0,1]$, the statistic $\mathbb{m}^{\mathbb{K}}_{d_{N}}(\lmss{t})(u)$ extends linearly to the tensor product $\mathbb{R}\left[\mathcal{O}\brauer \right] \otimes \mathbb{R}\left[\lmss{M}_{q}(k)\right]$ (resp. to $\mathbb{R}\left[\mathcal{O}\brauer \right] \otimes \mathbb{R}\left[\overline{\lmss{M}}_{q}(k)\right]$), if $\mathbb{K} = \mathbb{R}$ or $\mathbb{H}$ (resp. if $\mathbb{K}=\mathbb{C}$).
\subsection{Convergence of the extraction processes' statistics}
We assume in this section that as $N$ tends to infinity, the ratio $\frac{d_{N}(i)}{N}$ converges for each integer $1 \leq i \leq n$:
\begin{equation}
\label{assumption}
\tag{$\Delta$}
\frac{d_{N}(i)}{N} \underset{N\to+\infty}{\longrightarrow} r_{i} \in ]0,1],~\text{ for all } 1 \leq i \leq n.
\end{equation}
Recall that we denote by $\lmss{Ker}(d_{N})$ the partition of $\llbracket 1,n \rrbracket$ of all level sets of the function $ \llbracket 1,n\rrbracket\ni i \mapsto \frac{d_{N}(i)}{N}$.
As noticed in Section \ref{schur_weyl}, if two dimensions functions $f$ and $f^{\prime}$ satisfy $\lmss{Ker}(f) = \lmss{Ker}(f^{\prime})$ then $\brauerdf{f} = \brauerdf{f^{\prime}}$. This section is devoted to the proof of the following proposition, which main corollary is Theorem $\ref{thm1rect}$. In the sequel we set $\lmss{r}=(r_{1},\ldots,r_{n})$. The sequence $\lmss{r}$ is a sequence of positive integers, we means that the dimensions of the blocks that are extracted grow linearly compared to the total dimension $N$.
\begin{proposition}
\label{prp:convstat}
\par  As $N \rightarrow +\infty$, for each non-mixing oriented Brauer diagram $(b,s)$ and word $w$ in $\lmss{M}_{q}(k)$, $\mathbb{m}_{d_{N}}^{\mathbb{R}}((b,s),w,\lmss{t})$ and $\mathbb{m}_{d_{N}}^{\mathbb{H}}((b,s),w,\lmss{t})$ converge to the same limit.
\par As $N \rightarrow +\infty$, for each non-mixing Brauer diagram $(b,s)$ and word $w$ in $\overline{\lmss{M}}_{q}(k)$, $\mathbb{m}^{\mathbb{C}}_{d_{N}}((b,s),w,\lmss{t})$ converges.
\par In addition, if we assume that the sequence of Kernels $\left(\lmss{ker}(d_{N})\right)_{N \geq 1}$ is bounded from below by $\lmss{Ker}(d_{1})$, $\lmss{Ker}(d_{1}) \leq \lmss{Ker}(d_{N}),~\forall N \geq 1$, the above convergence is extended to the whole set of Brauer diagrams $\brauerdf{d_{1}}$.
\end{proposition}
Let us explain with more details the hypothesis we made on the sequence of kernels, requiring for each element of this sequence to be larger than $d_{1}$ is the same as cutting a matrix into blocks, letting each block growing while maintaining the dimensions of blocks that were equal initially, equal. The kernel $\lmss{ker}(\lmss{r})$ is, in general, greater than the kernels $\lmss{ker}(d_{N})$, $N\geq 1$. Note that the sequence of algebras $\brauer(d_{N})$ is a sequence of \emph{linearly} isomorphic spaces, and each algebra $\brauer(d_{N})$ is injected canonically into the limit algebra $\brauer(\lmss{r})$.
\par We denote by $\mathbb{m}_{\lmss{r}}(\lmss{t})$ and $\overline{\mathbb{m}}_{\sf r}(\lmss{t})$ the limit of $\mathbb{m}^{\mathbb{R}}_{d_{N}}(\lmss{t})$, respectively, $\mathbb{m}^{\mathbb{C}}_{d_{N}}(\lmss{t})$. In the course of proving Proposition \ref{prp:convstat}, we find two differential systems the functions $\mathbb{m}_{\lmss{r}}(\lmss{t})$ and $\overline{\mathbb{m}}_{\lmss{r}}$ are solutions of. The generators of these systems are denoted $L_{\lmss{r}}$ and $\bar{L}_{\lmss{r}}$ and are defined below as operators acting on $\mathbb{R}\left[\mathcal{O}\brauer \right] \otimes \mathbb{R}\left[\lmss{M}_{q}(k)\right]$ for $L_{\lmss{r}}$ and $\mathbb{R}\left[\mathcal{O}\brauer \right] \otimes \mathbb{R}\left[\overline{\lmss{M}}_{q}(k)\right]$ for $\overline{L}_{\lmss{r}}$.
To provide tractable formulae for these operators, with a slight abuse of notation, we introduce for each positive dimension function $f$, the function $\lmss{f}$ on the set product $\mathcal{O}\cbrauerf{f}\times \mathcal{O}\cbrauerf{f}$ by the equation:
\begin{equation}
\label{brauerdimensionfunction}
\lmss{f}(((b,s),\lmss{o}),((b^{\prime},s^{\prime}),\lmss{o}^{\prime}))= \displaystyle\prod_{d\in \{\lmss{f}\}} d^{\lmss{fnc}_{d}((b,s),\lmss{o}) -\lmss{fnc}_{d}((b^{\prime},s^{\prime}),\lmss{o}^{\prime})},
\end{equation}
with $((b,s),\lmss{o}),((b^{\prime},s^{\prime}),\lmss{o}^{\prime})\in \mathcal{O}\cbrauerf{f}\times \mathcal{O}\cbrauerf{f}$ and remark that $\lmss{f}$ is well defined owing to the positivity of $f$.
We set $$c_{N}^{\mathbb{R}} = -\frac{1}{2}\frac{N-1}{N},~c_{N}^{\mathbb{H}} = -\frac{1}{2}\frac{N-3}{N},~c_{N}^{\mathbb{C}}=-\frac{1}{2}.$$
\par We are now ready to compute the derivatives of the statistics $\mathbb{m}_{d_{N}}^{\mathbb{K}}$. By using formulae \eqref{meancomplex}, \eqref{meanreal} and \eqref{meanquaternion} of Section \ref{squareextractions} for the mean $\mathbb{E}\left[\munitaryfd(t_{i}) \right]$, $i \leq q$ and obtain for each integer $1 \leq i\leq q$ existence of an operator $L_{i,N}^{\mathbb{K}}$ acting on the space $\mathbb{R}\left[\mathcal{O}\brauer \right] \otimes \mathbb{R}\left[\lmss{M}_{q}(k)\right]$ for $\mathbb{K} = \mathbb{R}$ or $\mathbb{H}$ and on $\mathbb{R}\left[\mathcal{O}\brauer \right] \otimes \mathbb{R}\left[\overline{\lmss{M}}_{q}(k)\right]$ if $\mathbb{K}=\mathbb{C}$ such that:
\begin{equation}
\label{formulader}
\frac{d}{du}\mathbb{m}_{d_{N}}^{\mathbb{K}}(\lmss{t})(u) = \sum_{i=1}^{q}\mathbb{m}^{\mathbb{K}}_{d_{N}}(\lmss{t})(u) \circ L^{\mathbb{K}}_{i,N} = \mathbb{m}_{d_{N}}^{\mathbb{K}}(\lmss{t})(u)\circ L^{\mathbb{K}}_{N},\text{ for all } u \in [0,1],
\end{equation}
with $L_{N}^{\mathbb{K}} = \sum_{i=1}^{q}t_{i}L_{i,N}^{\mathbb{K}}$.
\par The two sets of coloured Brauer diagrams $\brauerdf{d_{N}}$ and $\brauerdf{r_{N}}$ are equal, the operators $L_{i,N}^{\mathbb{K}}$ may thus be seen as acting on the linear span of the tensor product of $\mathbb{R}\left[\mathcal{O}\brauerdf{r_{N}}\right]$ with $\mathbb{R}\left[\overline{\lmss{M}}_{q}(k)\mathbb{R}\right]$(or $\mathbb{R}\left[\lmss{M}_{q}(k)\right]$ in the complex case).
For $1 \leq i \leq n$ and a word in $w \in \lmss{M}_{q}$ (resp. in $\overline{\lmss{M}_{q}}$), we denote by $\lmss{n}_{i}(w)$ the number of letters in the word $w$ equal to $x_{i}$ (resp. to $x_{i}$ or $\overline{x}_{i}$). We recall that the sets of non mixing coloured Brauer diagrams which underlying component is a transposition of a projection is denoted, respectively, by $\trnm$ and $\prnm$. Let $b \in \mathcal{O}\brauerdf{d_{N}}$ a Brauer diagram and $w \in \lmss{M}_{q}(k)$, for $\mathbb{K}=\mathbb{R}$ or $\mathbb{H}$,
\begin{align}
	L^{\mathbb{K}}_{i,N}(b\otimes w) &= c_{N}^{\mathbb{K}}\lmss{n}_{i}(w)(b \otimes w) + \sum_{\substack{e_{ij} \in \prnm, \\ w_{i} = w_{j} = x_{i}}} N^{\lmss{nc}(\ncc{b} \vee \ncc{e}_{ij})-\lmss{nc}(\ncc{b}\vee 1)-1}\lmss{\lmss{r}}_{N}\big(\overset{\circ}{e}_{ij}\diamond \overset{\circ}{b},\overset{\circ}{b}\big)(e_{ij}\diamond b \otimes w)\nonumber \\ &\hspace{3cm}-\sum_{\substack{\tau_{ij} \in \trnm, \\ w_{i} = w_{j} = x_{i}}} N^{\lmss{nc}(\ncc{b} \vee \ncc{\tau}_{ij})-\lmss{nc}(\ncc{b}\vee 1)-1}\lmss{\lmss{r}}_{N}\big(\overset{\circ}{\tau}_{ij}\diamond \overset{\circ}{b},\overset{\circ}{b}\big)(\tau_{ij}\diamond b \otimes w)\label{genrh}.
\end{align}
with the finite sequence $\lmss{r}_{N}$ defined by
$
\lmss{r}_{N} = \left(\frac{d_{N}(1)}{N},\ldots,\frac{d_{N}(n)}{N}\right).
$
For the the complex case, we let $w$ be a word in the monoid $\overline{\lmss{M}}_{q}(k)$,
\begin{align}
L^{\mathbb{C}}_{i,N}(b \otimes w) &= c_{N}^{\mathbb{C}}\lmss{n}_{i}(w)(b \otimes w) + \sum_{\substack{e_{ij} \in \prnm, \\ w_{i}, w_{j} \in \{x_{i},\overline{x}_{i}\}}} N^{\lmss{nc}(\ncc{b} \vee \ncc{e}_{ij})-\lmss{nc}(\ncc{b}\vee 1)-1}\lmss{\lmss{r}}_{N}\big(\overset{\circ}{e}_{ij}\diamond \overset{\circ}{b},\overset{\circ}{b}\big)(e_{ij}\diamond b \otimes w) \nonumber\\ &\hspace{3cm}-\sum_{\substack{\tau_{ij} \in \trnm, \\ w_{i} = w_{j} \in \{x_{i},\overline{x}_{i}\}}} N^{\lmss{nc}(\ncc{b} \vee \ncc{\tau}_{ij})-\lmss{nc}(\ncc{b}\vee 1)-1}\lmss{\lmss{r}}_{N}\big(\overset{\circ}{\tau}_{ij}\diamond \overset{\circ}{b},\overset{\circ}{b}\big)(\tau_{ij}\diamond b \otimes w)\label{genc}.
\end{align}
Let us detail the computations for the real case. For an oriented Brauer diagram $(b,s)$, set $D_{N}(b,s) = \prod_{d\in\{{d}_{N}\}}d^{-\sf{fnc}_{d}(b)}$. By definition, we have
\begin{equation}
\label{derstat}
\frac{d}{du}\mathbb{m}_{d_{N}}^{\mathbb{R}}(\lmss{t},w,b)(u) = D_{N}(b,s)\lmss{Tr}\left(\frac{d}{du}\mathbb{E}\left[w^{\otimes}([\munitaryfdr]^{1}(ut_{q})\otimes \cdots \otimes [\munitaryfdr(ut_{N}))]^{q}\right] \circ \rho^{\mathbb{R}}_{\lmss{d}_{N}}(b) \right).
\end{equation}
Owing to formulae in Section \ref{squareextractions} we proved for the mean of tensor monomials of the unitary Brownian diffusion and mutual independence of the family $\left\{[\munitaryfd]^{i},~ 1 \leq i \leq q \right\}$,
\begin{equation*}
\begin{split}
&\frac{d}{du}\mathbb{E}\left[w^{\otimes}([\munitaryfdr]^{1}(ut_{q})\otimes \cdots \otimes [\munitaryfdr(ut_{N}))]^{q}\right] = \mathbb{E}\left[w^{\otimes}([\munitaryfdr]^{1}(ut_{q})\otimes \cdots \otimes [\munitaryfdr(ut_{q}))]^{q}\right] \\
&\hspace{5.5cm} \times \sum_{i=1}^{q}t_{i}\bigg(c_{N}^{\mathbb{R}}\lmss{n}_{i}(w) + \sum_{\substack{e_{ij} \in \prnm, \\ w_{i} = w_{j} = x_{i}}} \frac{1}{N}e_{ij} -\sum_{\substack{\tau_{ij} \in \trnm, \\ w_{i} = w_{j} = x_{i}}}\frac{1}{N}\tau_{ij}\bigg)
\end{split}
\end{equation*}
We insert this last equation into formula \eqref{derstat} to obtain:
\begin{equation*}
\label{step2}
\begin{split}
&\frac{d}{du}\mathbb{m}_{d_{N}}^{\mathbb{R}}(\lmss{t},w,b)(u) = \sum_{i}t_{i}\Big(\lmss{n}_{i}(w)\mathbb{m}^{\mathbb{R}}_{d_{N}}((b,s),w,\lmss{t})(u) \\
&\frac{1}{N}\sum_{\substack{e_{ij} \in \prnm, \\ w_{i} = w_{j} = x_{i}}} \prod_{d}d^{\mathcal{K}_{d}(e_{ij},b)}D_{N}(b,s)\lmss{Tr}\left(\mathbb{E}\left[w^{\otimes}([\munitaryfdr]^{1}(ut_{q})\otimes \cdots \otimes [\munitaryfdr(ut_{q}))]^{q}\right] \circ \rho_{d_{N}}^{\mathbb{R}}(e_{ij}\circ b)\right)\\
&-\frac{1}{N}\sum_{\substack{\tau_{ij} \in \trnm, \\ w_{i} = w_{j} = x_{i}}}\prod_{d}d^{\mathcal{K}_{d}(\tau_{ij},b)}D_{N}((b,s))\lmss{Tr}\left(\mathbb{E}\left[w^{\otimes}([\munitaryfdr]^{1}(ut_{q})\otimes \cdots \otimes [\munitaryfdr(ut_{q}))]^{q}\right] \circ \rho_{d_{N}}^{\mathbb{R}}(\tau_{ij}\circ b)\right)\Big)
\end{split}
\end{equation*}
We orient the coloured Brauer diagrams $e_{ij}\circ b$ and $\tau_{ij} \circ b$ using the $\diamond$ operator. We multiply each terms in the first sum, respectively the second sum of the last equation by the normalization factor $D_{N}(\tau_{ij} \diamond (b,s))$, respectively $D_{N}(e_{ij} \diamond (b,s))$ to get, with $r\in \{e,\tau\}$:
\begin{equation*}
\resizebox{\hsize}{!}{$%
\frac{\prod_{d\in\{d_{N}\}}d^{\mathcal{K}_{d}(r_{ij},b)}D_{N}(b,s)}{ND_{N}(r_{ij}\diamond b)}\mathbb{m}^{\mathbb{R}}_{d_{N}}(r_{ij}\diamond b,w,\lmss{t})=N^{\lmss{nc}(\ncc{b} \vee \ncc{r}_{ij})-\lmss{nc}(\ncc{b}\vee 1)-1}\lmss{\lmss{r}}_{N}\left(\overset{\circ}{r}_{ij}\diamond \overset{\circ}{b},\overset{\circ}{b}\right)f_{d_{N}}(r_{ij}\diamond (b,s),w,\lmss{t}).$%
}
\end{equation*}
\par The functional $\delta_{\Delta_{k}}$ is the indicator function of the set of coloured oriented Brauer diagrams in $\mathcal{O}\brauerdf{d_{1}}$ that are diagonally coloured,
\begin{equation*}
\delta_{\Delta_{k}}((b,s)\otimes w) = \prod_{i=1}^{q}\delta_{c_{b}(i)=c_{b}(i^{\prime})}.
\end{equation*}
 Since the family $\{L_{i,N}^{\mathbb{K}}, i \leq N\}$ is a commuting family of operators, we get the following formula for $\mathbb{m}^{\mathbb{K}}_{d(N)}(\lmss{t})$,
\begin{equation}
\label{eq:stat}
\mathbb{m}_{d(N)}^{\mathbb{K}}(\lmss{t})(1) = \delta_{\Delta_{k}}\circ \displaystyle\prod_{i=1}^{q}e^{t_{i}L_{i,N}^{\mathbb{K}}}=\delta_{\Delta_{k}} \circ  \exp\bigg(\sum_{i=1}^{q}t_{i}L_{i,N}^{\mathbb{K}}\bigg).
\end{equation}
We draw the reader's attention on the fact that the domain of definition of the statistic $\mathbb{m}_{d(N)}^{\mathbb{K}}(\lmss{t})$ and the generators $L^{\mathbb{K}}_{i,N}$ with $i \leq q$ rests on the kernel $\lmss{Ker}(d_{N})$ of the dimension function $d_{N}$. This prevents us to simply let $N$ tends to infinity in the formulae \eqref{genrh} and \eqref{genc} without further assumption on the sequence $d_{N}$. Nevertheless, under the assumption made in Proposition \ref{prp:convstat} on the sequence $\left(d_{N}\right)_{N\geq1}$ the aforementioned issue does not show up; the generators $L_{i,N}^{\mathbb{K}}$ are defined on the real vector space $\mathbb{R}\left[\brauerdf{\sf d_{1}}\right]$ included in all the spaces $\mathbb{R}\left[\brauerdf{\lmss{ker}(d_{N})}\right]$, $N \geq 1$.

\par Let $1 \leq i,j \leq k$ two integers. The quantities $\overset{\circ}{\tau}_{ij}\diamond \overset{\circ}{b}$ and $\overset{\circ}{e}_{ij}\diamond \overset{\circ}{b}$ are computed in the algebra $\brauer(\lmss{r}_{N})$. To make explicit this dependence we write instead for the next few lines $\overset{\circ}{\tau}\diamond_{\lmss{r}_{N}}\overset{\circ}{b}$ and $\overset{\circ}{e}_{ij}\diamond_{\lmss{r}_{N}} \overset{\circ}{b}$. Since the limiting ratios $\lmss{r}_{i}$ are all positive,
\begin{equation*}
\lmss{\lmss{r}}_{N}\left(\overset{\circ}{e}_{ij}\diamond_{\lmss{r}_{N}} \overset{\circ}{b},\overset{\circ}{b}\right)\underset{N \to +\infty}{\longrightarrow} \lmss{r}\left(\overset{\circ}{e}_{ij}\diamond_{\lmss{r}} \overset{\circ}{b},\overset{\circ}{b}\right),~\lmss{\lmss{r}}_{N}\left(\overset{\circ}{\tau}_{ij}\diamond_{\lmss{r}_{N}} \overset{\circ}{b},\overset{\circ}{b}\right)\underset{N \to +\infty}{\longrightarrow} \lmss{r}\left(\overset{\circ}{\tau}_{ij}\diamond_{\lmss{r}} \overset{\circ}{b},\overset{\circ}{b}\right).
\end{equation*}
\par We repeat the discussion we made in Section \ref{squareextractions} in which the convergence of the statistics $\mathbb{m}^{\mathbb{K}}_{d}(t), t \geq 0$ was investigated as the dimension $d$ tends to infinity: the sums over the elementary non-mixing diagrams $r$ in equation \eqref{genrh} and \eqref{genc} localize over the set of diagrams that create a cycle or a loop if multiplied with $b$: $r \in \trp{b}\cup\prp{b}$.
\par We are non convinced that we can let $N$ tends to infinity in equation \eqref{genrh} and \eqref{genc}. For each Brauer diagram in $\brauerdf{d_{1}}$, and word $w$ in $\lmss{M}_{2}$,
\begin{equation*}
L_{i,N}^{\mathbb{R}}(b \otimes w) \underset{N\rightarrow +\infty}{\longrightarrow} L_{i,\lmss{r}}(b \otimes w),~L_{i,N}^{\mathbb{H}}(b \otimes w) \underset{N\rightarrow +\infty}{\longrightarrow} L_{i,\lmss{r}}(b \otimes w),
\end{equation*}
with the generator $L_{i,\lmss{r}}$ defined for $b \in \mathcal{O}\brauerdf{\lmss{r}}$:
\begin{gather}
\label{genlimit}
\raisetag{-40pt}
\resizebox{\hsize}{!}{$
	\begin{split}
		L_{i,\lmss{\lmss{r}}}(b \otimes w)= -\frac{1}{2}\lmss{n}_{i}(w)(b,w) +&\sum_{\substack{e_{ij} \in \prpnm{b} \\ w_{i} = w_{j} = x_{i}}} \lmss{r}(\overset{\circ}{e_{ij}} \diamond \overset{\circ}{b},\overset{\circ}{b})(e_{ij}\diamond b \otimes  w) + \sum_{\substack{\tau_{ij} \in \trpnm{b}, \\ w_{i} = w_{j}=x_{i}}} \lmss{r}(\overset{\circ}{\tau}_{ij}\diamond \overset{\circ}{b},\overset{\circ}{b})(\tau_{ij}\diamond b \otimes w).
	\end{split}$}
\end{gather}
For the complex case, if $w$ is a word in $\overline{\lmss{M}}_{2}$ and $b \in \mathcal{O}\brauerdf{d_{1}}$, if we let $N$ tends to infinity, we obtain the convergence for each integer $1 \leq i\leq n$ of $\overline{L}^{\mathbb{C}}_{i,N}(b,w)$ to $\overline{L}_{i,\lmss{r}}(b \otimes w)$ with $L_{i,\lmss{r}}$ defined by the equation
\begin{gather}
\raisetag{-40pt}
\label{genlimitbar}
\resizebox{\hsize}{!}{$
	\begin{split}
		\overline{L}_{i,\lmss{\lmss{r}}}(b \otimes w) = -\frac{1}{2}\lmss{n}_{i}(w)(b \otimes w) +&\sum_{\substack{e_{ij} \in \prpnm{b} \\ w_{i} \neq w_{j} \\ w_{i},w_{j} \in \{x_{i},\bar{x}_{i}\}}} \lmss{r}(\overset{\circ}{e_{ij}} \diamond \overset{\circ}{b},\overset{\circ}{b})(e_{ij}\diamond b \otimes w) + \sum_{\substack{\tau_{ij} \in \trpnm{b}, \\ w_{i} = w_{j} \\ w_{i},w_{j} \in \{x_{i},\bar{x}_{i}\}}} \lmss{r}(\overset{\circ}{\tau}_{ij}\diamond \overset{\circ}{b},\overset{\circ}{b})(\tau_{ij}\diamond b \otimes w).
	\end{split}%
$}
\end{gather}
Set $L_{\lmss{r}} = \sum_{i=1}^{q}t_{i}L_{i,\lmss{r}}$ and $\overline{L}_{\lmss{r}} = \sum_{i=1}^{q}t_{i}\overline{L}_{i,\lmss{r}}$.
Since the generators $L_{i,N}^{\mathbb{K}}, i \leq q, \mathbb{K}=\mathbb{R},\mathbb{C}$ or $\mathbb{H}$ and $L_{i,\lmss{r}}, \bar{L}_{i,\lmss{r}}$ act on finite dimensional spaces, we can let $N$ tends to infinity in equation $\eqref{eq:stat}$ and get, for $\mathbb{K} = \mathbb{R}$ or $\mathbb{H}$
\begin{equation*}
\mathbb{m}_{d_{N}}^{\mathbb{K}}(b,w,\lmss{t})(1) \underset{N \rightarrow +\infty}{\longrightarrow} \delta_{\Delta_{k}} \circ e^{\sum_{i=1}^{q}t_{i}L_{i,\lmss{r}}(b \otimes w)},~\mathbb{m}_{d_{N}}^{\mathbb{C}}(b \otimes w,\lmss{t})(1) \underset{N \rightarrow +\infty}{\longrightarrow} \delta_{\Delta_{k}} \circ e^{\sum_{i=1}^{q}t_{i}\overline{L}_{i,\lmss{r}}(b\otimes w)}.
\end{equation*}
We want now to empasize an important feature of the generators $L_{\lmss{r}}$ and $\overline{L}_{\lmss{r}}$ that is most easily discussed in the case $q=1$. If $c_{1},\ldots,c_{l}$ are irreductible Brauer diagrams of sizes $k_{1},\ldots,k_{l}$ we denote by $c_{1}\otimes \cdots \otimes c_{l}$ the coloured Brauer diagram of size $k=k_{1}+\cdots+k_{l}$ which cycles are the diagrams $c_{1},\ldots,c_{l}$ (in this order):

\begin{figure}[!htb]
\includegraphics[scale=0.75]{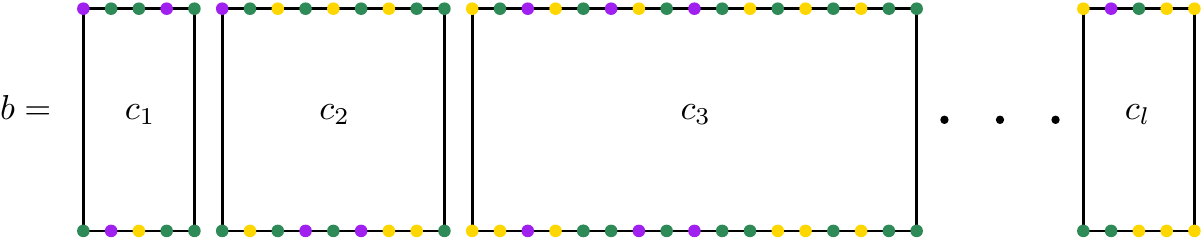}
\end{figure}

Owing to the definition of the operator $L_{\lmss{r}}$, for words $w_{i} \in \lmss{M}_{q}(k_{i})$, $1 \leq i \leq l$,
\begin{equation*}
	L_{\lmss{r}}(b\otimes w) = L_{\lmss{r}}(c_{1}\otimes w_{1}) \otimes \cdots \otimes c_{k}\otimes w_{k} + c_{1} \otimes L_{\lmss{r}}(c_{2}) \otimes \cdots \otimes c_{k} + \cdots c_{1} \otimes c_{2} \otimes \cdots \otimes L_{\lmss{r}}(c_{k}).
\end{equation*}
This last equation implies $\mathbb{m}_{\lmss{r}}(c_{1},w_{1})\cdots \mathbb{m}_{\lmss{r}}(c_{l},w_{l}) = \mathbb{m}(b \otimes w_{1}\cdots w_{l})$. This factorization property extends to general Brauer diagrams:
\begin{equation}
\label{factorization}
	\mathbb{m}_{\lmss{r}}(b,w) = \mathbb{m}_{\lmss{r}}(b_{V_{1}},w_{V_{1}}) \cdots \mathbb{m}_{\lmss{r}}(b_{V_{p}}, w_{V_{p}})
\end{equation}
where $\ncc{b}\vee 1 = \{V_{1}, \ldots, V_{p}\}$, $w_{V}$ is the word $w$ restricted to $V$, $w_{V} = w_{i_{1}}\cdots w_{i_{k}}$ if $V=\{i_{1},\ldots,i_{k}\}$ and $b_{V}$ is the part of $b$ that is contained in $V$. Of course, the same property holds for $\overline{\mathbb{m}}_{\lmss{r}}$.
\subsection{Convergence in high dimensions of the rectangular extraction processes}
\label{convergencesprocesses}
In this section, we prove that Proposition \ref{prp:convstat} implies the convergence in non-commutative distribution of the quantum processes $\runitaryfdN$ and $\runitaryfdN$ toward free L\'evy processes. Let $N,n \geq 1$ be integers and for each $N\geq 1$, $d_{N}$ a partition of $N$ into $n$ parts. In this section, we interpret the convergence proved in Section \ref{extractionprocesses} of the statistics $f^{\mathbb{K}}_{d_{N}}$ as $N$ tends to infinity in term of convergence of the rectangular extractions process $\runitaryfdN$ and $\rdunitaryfdN$ in non-commutative distribution.
\begin{theorem}
\label{thm1rect}
Let $n \geq 1$ an integer. For each integer $N \geq 1$, pick a partition $d_{N}$ of $N$ into $n$ parts. Let $\mathbb{K}$ be one of the three divisions algebras $\mathbb{R}$, $\mathbb{C}$ or $\mathbb{H}$.
Assume that as $N$ tends to infinity, there exists positive real numbers $r_{i} \in ]0,1]$, for $1\leq i\leq n$, such that
\begin{equation*}
\frac{d_{N}(i)}{N} \underset{N\rightarrow +\infty}{\longrightarrow} r_{i},~1 \leq i \leq n.
\end{equation*}
As the dimension $N$ tends to infinity, the non-commutative distribution of $\runitaryfdN$ converges to a $\mathcal{D}_{n}$ amalgamated free semi-group.
\end{theorem}
Prior to proving Theorem \ref{thm1rect}, we first settle some notations. We use the symbol $\lmss{E}_{\sf r}$ to denote the amalgamated free semi-group which existence stated in Theorem \ref{thm1rect} with $\lmss{r}=(r_{1},\ldots,r_{n})$. Recall that $\lmss{E}_{\sf r}$ is a free amalgamated semi-group means:
\begin{equation*}
\lmss{E}_{\lmss{r}}(t+s)={\sf E}_{\lmss{r}}(t)\freem_{\mathcal{D}_{n}}{ \sf E}_{\lmss{r}}(s).
\end{equation*}
Also a formula for the generator $\mathcal{L}_{\lmss{r}}$ of this semi-group can be read of the formula for the operator $\overline{L}_{1,\lmss{r}}$. Let $x = u_{\alpha_{1},\alpha_{2}}^{\varepsilon(1)}\cdots u_{\alpha_{k-1},\alpha_{k}}^{\varepsilon(k)}$, the Brauer diagram $b$ in the following equation is defined as in equation \eqref{eq:brauerdef} and \eqref{eq:brauerdef2} below,
\begin{gather}
\raisetag{-40pt}
\label{formulagenerator}
\begin{split}
\mathcal{L}_{\lmss{r}}(x)=-\frac{1}{2}k\delta_{\Delta_{k}}(b) +&\sum_{\substack{e_{ij} \in \prpnm{b}}} \lmss{r}(\overset{\circ}{e_{ij}} \diamond \overset{\circ}{b},\overset{\circ}{b})\delta_{\Delta_{k}}(e_{ij}\diamond b) + \sum_{\tau_{ij} \in \trpnm{b}} \lmss{r}(\overset{\circ}{\tau}_{ij}\diamond \overset{\circ}{b},\overset{\circ}{b})\delta_{\Delta_{k}}(\tau_{ij}\diamond b).
\end{split}
\end{gather}
Let $i,j \leq k$. The way we chose to orient $e_{ij}\circ b$ and $\tau_{ij}\circ b$, using the operator $\diamond$ was quite arbitrary. However, the generator $\mathcal{L}_{\lmss{r}}$ does no depend on such a choice, since $\delta_{\Delta_{k}}$ is the support function of diagonally coloured Brauer diagrams.
We fix, once for all, a division algebra $\mathbb{K}$ and to lighten the notation, we drop the symbol $\mathbb{K}$ in the notations introduced so far.
\par We focus on the cases $\mathbb{K}=\mathbb{R}$ or $\mathbb{H}$. In fact, the function $\overline{\mathbb{m}}_{\lmss{r}}$ is equal to the function $\mathbb{m}_{\lmss{r}}$ on the linear span of tensors $(b,s)\otimes \overline{w}$ with $((b,s),\overline{w})$ a pair of compatible word and diagrams meaning that:
\begin{equation*}
\mathbb{m}_{\lmss{r}}((b,s),w) = \lmss{m}((b,s)\otimes \overline{w})
\end{equation*}
for any word $w \in \lmss{M}_{q}(k)$ such that for all integer $1 \leq i\leq k$, $\overline{w}_{i}\in\{w_{i},\overline{w}_{i}\}$.
The joint distribution of the random variables $\left[\munitaryfdc\right]^{(1)},\ldots,\left[\munitaryfdc\right]^{(q)}$ is contained in the range of the statistic $\mathbb{m}^{\mathbb{C}}_{d_{N}}$ restricted to the linear span tensor of compatible words and diagrams, hence the limiting distribution of the process $\runitaryfdNc$ is equal to the limiting distribution of $\runitaryfdNr$ (and equal to the limiting distribution of $\runitaryfdNh$). We continue with a small reminder. In Section \ref{extractionprocesses}, we introduced $\mathcal{M}_{d_{N}}$ as the rectangular probability space $\left(\mathcal{M}_{N}(L^{\infty}(\Omega,\mathcal{F},\mathbb{P},\mathbb{K})),~\mathbb{E}_{d_{N}},\frac{d_{N}}{N} \right)$. For all couple of integers $\alpha,\beta$ in the interval $[1,\ldots,n]$ we denote by $\mathcal{M}_{d_{N}}(\alpha,\beta)$ the compressed space $p_{\alpha}\mathcal{M}_{d_{N}}p_{\beta}$.
\par The conditional expectation $\mathbb{E}_{d_{N}}$ is a $\mathcal{D}_{n}$-bimodule map and for each integer $k\geq 1$, it defines an other $\mathcal{D}_{n}$-bimodule map $\mathbb{E}^{k}_{d_{N}}: \mathcal{M}_{d_{N}}\otimes_{\mathcal{D}_{n}} \cdots \otimes_{\mathcal{D}_{n}}\mathcal{M}_{d_{N}} \to \mathcal{D}_{n}$ by the formula:
\begin{equation*}
\mathbb{E}^{k}_{d_{N}}(M_{1}\otimes \cdots \otimes M_{k}) = \mathbb{E}^{d_{N}}\left(M_{1}\cdots M_{k} \right).
\end{equation*}
To the family of maps  $\left(\mathbb{E}^{k}_{d_{N}}\right)_{k \geq 1}$ is associated a multiplicative functional, also denoted $\mathbb{E}_{d_{N}}$, on the set of non crossing partitions (of any size).
To study asymptotic amalgamated freeness, it is more convenient to work with the cumulant functional $\{c_{\pi}^{d_{N}}: \mathcal{M}_{d_{N}} \otimes_{\mathcal{D}_{n}} \cdots \otimes_{\mathcal{D}_{n}} \mathcal{M}_{d_{N}} \to \mathcal{D}_{n}$, $\pi \in \lmss{NC}_{k}\}$. These cumulant functions are obtained by mean of a M\"oebius transformation, for $r_{1},\ldots,r_{k} \in \mathcal{M}_{d_{N}}$,
\begin{equation*}
\mathbb{E}^{\pi}_{d_{N}}(r_{1}\otimes_{\mathcal{D}_{n}}\cdots \otimes_{\mathcal{D}_{n}} r_{n})= \sum_{\alpha \leq \pi} c_{d_{N}}^{\alpha}(r_{1}\otimes_{\mathcal{D}_{n}} \cdots \otimes_{\mathcal{D}_{n}} r_{k}),~\pi \in \lmss{NC}_{k},
\end{equation*}
or equivalently:
\begin{equation*}
c^{\pi}_{d_{N}}(r_{1} \otimes \cdots \otimes r_{k}) = \sum_{\alpha \leq \pi}\mu(\alpha,\pi)\mathbb{E}^{\alpha}_{d_{N}}(r_{1}\otimes \cdots\otimes r_{k}).
\end{equation*}
Asymptotic freeness of the semi-group $\lmss{E}_{d_{N}}$ is equivalent to asymptotic amalgamated freeness of the increments of the process $\runitaryfd$, which can be checked on the cumulants. Proof of Theorem \ref{thm1rect} is thus divided into two (big) steps:
\begin{enumerate}[\indent 1.]
\item For each time $t \geq 0$ and integer $k\geq 1$, we prove the convergence of $\lmss{E}_{k}^{d_{N}}(t)$, as $N$ tends to infinity.
\item For all times $s_{1} < t_{1} \leq s_{2} < t_{2} \ldots \leq s_{q} < t_{q}$, we prove that the cumulants
$$c_{d_{N}}^{k}\left(\left[\munitaryfd(s_{j_{1}},t_{j_{1}})\right]^{\varepsilon(1)}\otimes_{\mathcal{D}_{n}} \cdots \otimes_{\mathcal{D}_{n}}\left[\munitaryfd(s_{j_{k}},t_{j_{k}})\right]^{\varepsilon(k)}\right)$$
of increments of the process $\runitaryfd$ converges to $0$ for any $k$-tuples $\varepsilon(1),\ldots,\varepsilon(k)\in\{1,\star\}$ if there exists two non equal integers $j_{a} \neq j_{b}$ in the sequence $j_{1},\ldots,j_{k} \in \{1,\ldots,q\}$.
\end{enumerate}
Point $1.$ shows no difficulties, it is a simple corollary of Proposition \ref{prp:convstat} we stated and proved in Section \ref{rectangularextractions}.
\par The second point needs more precisions. Prior to expound the proof of point 2., we sketch it.
With the notations introduced in point $2.$, set $\lmss{t} = (t_{j_{1}}-s_{t_{j_{1}}},\ldots,t_{j_{q}}-s_{j_{q}})$.
\par Let $\alpha_{0},\alpha_{1},\ldots,\alpha_{k}$ a $k$-tuple of integers in $[1,\ldots,n]$ and $\varepsilon(1),\ldots,\varepsilon(k) \in \{1,\star\}$. The first step is about finding a Brauer diagram $b \in \brauer$ and a word $w \in \lmss{M}_{q}(k)$ for which the asymptotic
\begin{equation}
\label{eq:condensation}
\mathbb{m}^{\mathbb{K}}_{d_{N}}(b,w,\lmss{t})p_{\alpha_{0}} - \mathbb{E}^{\pi}_{d_{N}}\left(p_{\alpha_{0}}\left[\munitaryfd(s_{j_{1}},t_{j_{1}})\right]^{\varepsilon(1)}p_{\alpha_{1}} \otimes \cdots \otimes p_{\alpha_{k-1}}\left[\munitaryfd(s_{j_{k}},t_{j_{k}})\right]^{\varepsilon(k)}p_{\alpha_{k}}\right)\underset{N\to +\infty}{\longrightarrow} 0
\end{equation}
holds.
\par Owing to the fact that  $\munitaryfd$ is a L\'evy process, we can substitute to the set of increments of $\munitaryfd$, $\{\munitaryfd(s_{1},t_{1}),\ldots,\munitaryfd(s_{q},t_{q})\}$, a set of independent copies of the process $\runitaryfd$ evaluated at times $t_{1}-s_{1},\ldots,t_{q}-s_{q}$ to compute the cumulants.
This step leans on \textit{condensation} property of $\runitaryfdN$ we exposed in at the end of the last section, see equation \eqref{eq:condensation}. We then write $\mathbb{m}_{\lmss{r}}(b,w,\lmss{t})$ (respectively $\overline{\lmss{m}}_{\lmss{r}}(b,w,\lmss{t})$) as a sum:
\begin{equation}
\label{eq:limitcumulant}
\mathbb{m}_{\lmss{r}}(b,w,\lmss{t}) = \sum_{\gamma \leq \pi} c_{\gamma}(\alpha,w,\lmss{t}).
\end{equation}
To that end we use the differential systems that are satisfied by the limiting statistic $\mathbb{m}_{\lmss{r}}$, and formulae \eqref{genrh}, \eqref{genc} found in Section \ref{extractionprocesses} for the generators to give an explicit formula for the coefficients $c_{\gamma}(\alpha,\beta,\lmss{t})$, $\gamma \in \lmss{NC}_{k}$. From equation \eqref{eq:condensation} and \eqref{eq:limitcumulant}, we infer that:
\begin{equation}
\label{eq:premoebius}
\mathbb{E}^{\pi}_{d_{N}}\left(p_{\alpha_{0}}\left[\munitaryfd(s_{j_{1}},t_{j_{1}})\right]^{\varepsilon(1)}p_{\alpha_{1}} \otimes_{\mathcal{D}_{n}} \cdots \otimes_{\mathcal{D}_{n}}p_{\alpha_{k-1}}\left[\munitaryfd(s_{j_{k}},t_{j_{k}})\right]^{\varepsilon(k)}p_{\alpha_{k}}\right) \underset{N \to +\infty}\longrightarrow \sum_{\gamma \leq \pi} c_{\gamma}(\alpha,\lmss{t}).
\end{equation}
Since this last equality is valid for all non-crossing partitions $\pi$ in $\lmss{NC}_{k}$, we can apply M\"oebius transformation to both side of \eqref{eq:premoebius} and deduce that:
\begin{equation}
  \label{eqn:brauerdiagrampi}
c_{d_{N}}^{k}\left(\left[\munitaryfd(s_{j_{1}},t_{j_{1}})\right]^{\varepsilon(1)} \otimes_{\mathcal{D}_{n}} \cdots \otimes_{\mathcal{D}_{n}}\left[\munitaryfd(s_{j_{k}},t_{j_{k}})\right]^{\varepsilon(k)}\right) \underset{N\to +\infty}{\longrightarrow}c_{k}(\alpha,\beta,\lmss{t}).
\end{equation}
We begin, of course, with the first step. Let $\pi \in \lmss{NC}_{k}$ a non-crossing partition and write $\pi = \{c_{1},\ldots,c_{p}\}$. The linear order on $[1,\ldots,k]$ along with the partition $\pi$ define a permutation $\sigma_{\pi}$ of $[1,\ldots,k]$: the cycles of $\pi$ are the blocks $c_{i}$'s, $i\leq p$, endowed with the natural cyclic order. We define a non-coloured Brauer diagram by:
\begin{equation}
\label{eq:brauerdef}
\ncc{b} = \displaystyle\prod_{\substack{1 \leq i \leq k : \\\varepsilon(i) = \star}} \Twist(\sigma_{\pi}).
\end{equation}
The permutation $\sigma_{b}$ associated with $b$ and defined in Section \ref{schur_weyl} is,  equal to $\sigma_{\pi}$. We now define the colourization of the non-coloured Brauer diagram $\ncc{b}$.
Set $\lmss{i}^{\prime} = (\alpha_{0},\ldots,\alpha_{k-1}),~\lmss{j}^{\prime} = (\alpha_{1},\ldots, \alpha_{k})$. By using $\lmss{i}$, $\lmss{j}$ and $\varepsilon$, we define a colourization $\lmss{i}$ of the bottom line of $\ncc{b}$ and an other one, which we call $\lmss{j}$ of the bottom line of $\ncc{b}$ by setting:
\begin{equation}
\lmss{i}_{i}=\lmss{i}^{\prime}_{i}, \text{ if } \varepsilon_{i}=1,~\lmss{i}_{i}=\lmss{j}^{\prime}_{i} \text{ if } \varepsilon_{i} = \star, \text{ and }\lmss{j}_{i}=\lmss{j}^{\prime}_{i}, \text{ if } \varepsilon_{i}=1,~\lmss{j}_{i}=\lmss{i}^{\prime}_{i} \text{ if } \varepsilon_{i} = \star.
\end{equation}
Finally, define the colourization $c$ in $\{1,\ldots,k,1^{\prime},\ldots,k^{\prime}\}$ by $c(i) = \lmss{j}_{i}$ and $c(i^{\prime}) = \lmss{j}_{i}$ for $1 \leq i \leq k$. We can not affirm that for all $\pi$, the colourization $c$ is an admissible colourization of the non-coloured Brauer diagram $\ncc{b}$ defined by equation \eqref{eq:brauerdef}. Hence, define the element $b$ in $\mathbb{R}\left[\brauer \right]$ as:
\begin{equation}
\label{eq:brauerdef2}
  b = (\ncc{b}, c) \text{ if } c \in C(\ncc{b}) \text{ and } b = 0 \text{ if } c \not\in C(\ncc{b}).
\end{equation}
Let $w$ be the word $w=x_{j_{1}} \cdots x_{j_{k}}$. We now prove that asymptotic \eqref{eq:condensation} holds. We recall first basic properties of the maps $\mathbb{E}^{k}_{d_{N}},~k\geq1$. Let $k\geq 1$ an integer and two finite sequences $\gamma_{0},\ldots,\gamma_{k-1}$ and $\beta_{1},\ldots,\beta_{k}$ of integers in the interval $[1,\ldots,n]$. Since $\mathbb{E}^{k}_{d_{N}}$ is a $\mathcal{D}_{n}$-bi-module map, $\mathbb{E}^{k}_{d_{N}}(a_{1} \otimes_{\mathcal{D}_{n}} \cdots \otimes_{\mathcal{D}_{n}}a_{k})=0$ with $a_{i} \in \mathcal{M}_{d_{N}}(\gamma_{i},\beta_{i})$ if there exists at least one integer $0 \leq i \leq k-1$ such that $\gamma_{i+1} \neq \beta_{i}$ (with the convention $\gamma_{k=1} = \gamma_{0}$) and
$\mathbb{E}_{d_{n}}^{k}(a_{1}\otimes_{\mathcal{D}_{n}}\cdots\otimes_{\mathcal{D}_{n}}) =\mathbb{E}_{d_{n}}^{k}(a_{1}\otimes_{\mathcal{D}_{n}}\cdots\otimes_{\mathcal{D}_{n}})p_{\beta_{0}}$.
Let $\beta_{0},\ldots,\beta_{k} \in [1,\ldots,n]$ a finite sequence of integers. A direct induction proves that for any non-crossing partition $\gamma \in \lmss{NC}_{k}$, the map $\mathbb{E}_{\gamma}^{d_{n}}$ evaluates to zero on all $k$-tuples of elements $(a_{1},\ldots,a_{k})$ with $a_{i}$ in the compressed algebra $\mathcal{M}_{d_{N}}(\beta_{i},\beta_{i+1})$, $0 \leq i\leq k$, if there exists a block $V = \{v_{1} < \ldots < v_{s}\}$ of $\gamma$ such that $\beta_{v_{1}-1} \neq \beta_{v_{k}}$.
Focusing on the tuple $\alpha$ and non-crossing partition $\pi$ we chose, requiring that $\alpha_{v_{1}-1} = \alpha_{v_{k}}$ for all $\{v_{1} < \cdots < v_{k}\}$ of $\pi$ is the same as demanding that the colourization $c$ is in $C(\ncc{b})$, which means $(\ncc{b},c) \in \brauer$.
$\mathbb{E}^{\pi}_{d_{N}}\left(p_{\alpha_{0}}\left[\munitaryfd(s_{j_{1}},t_{j_{1}})\right]^{\varepsilon(1)}p_{\alpha_{1}} \otimes_{\mathcal{D}_{n}} \cdots \otimes_{\mathcal{D}_{n}}p_{\alpha_{k-1}}\left[\munitaryfd(s_{j_{k}},t_{j_{k}})\right]^{\varepsilon(k)}p_{\alpha_{k}}\right) = 0$ for all integers $N \geq 1$ if, with the above notation, $\alpha_{v_{1}-1} \neq \alpha_{v_{k}}$ for at least one block of $\pi$.
We write $\mathbb{m}(b,w,\lmss{t})$ as a sum over non-crossing partitions. It cumbersome to introduce some new notations to give explicit formulas for the coefficients $c_{\gamma}(\alpha,\lmss{t},w)$. First, we introduce
\begin{equation*}
\lmss{R}^{+}_{s}(b) = \{ (r_{i_{1},j_{1}},\ldots,r_{i_{s},j_{s}}) \in (\lmss{R}_{k})^{s} : r_{i_{l},j_{l}} \in \lmss{T}^{+} \cup \lmss{W}^{+}({r_{i_{l+1},j_{l+1}}\circ  \cdots \circ r_{i_{s},j_{s}}\circ b})\}.
\end{equation*}
Let us give a geometric interpretation of the set $\lmss{R}^{+}_{s}(b)$. First, define a graph $\mathcal{G}$ that have as vertices the set $\brauer$ and by considering two Brauer diagrams as adjacent in this graph if one is obtained from this other by concatenation with a transposition or a projector. This graph have loops, since $e \circ e = e$. These loops can be broken if instead of considering as vertices the set $\brauer$, we replace it with the central extension $\overline{\brauer}$ and requiring for two Brauer elements in $\overline{\brauer}$ to be neighbours if one is obtained from the other by multiplication  with a transposition / projection.  For example, we would have $e e = (e, \lmss{o})$ for some loop $\lmss{o}$ that belongs to $\overline{\mathcal{G}}$. A tuple $\lmss{r}$ in $\lmss{R}^{+}_{s}(b)$ is a path in $\overline{\mathcal{G}}$) that starts at $(b,\emptyset)$ a visit successively Brauer elements that have one more loop or cycle comparing to the last one visited. We insist on the fact that the set $\lmss{R}^{+}_{s}(b)$ is not solely determined by the partition $\pi$ but  depends also on the sequence $\varepsilon$ that was use to twist the diagram (this twist are responsible for loops that may be created along a path).
To be more precise, we are interested in a subset of paths in $\lmss{R}_{s}^{+}(b)$ that have increments constrained by the word $w$,
\begin{gather*}
\lmss{R}^{+}_{s}(b,w) = \{(r_{i_{1},j_{1}},\ldots,r_{i_{s},j_{s}}) \in \lmss{R}^{+}_{s}(b) : w_{i_{1}} = w_{j_{1}}, \cdots, w_{i_{s}} = w_{j_{s}} \}.
\end{gather*}
We are now splitting the set $\lmss{R}^{+}_{s}(b,w)$ according to the cycle partition of the end-point of a path, and define for this the function:
\begin{equation*}
\gamma_{s}\left((r_{i_{1},j_{1}},\ldots,r_{i_{s},j_{s}})\right) =\left((\ncc{r}_{i_{1},j_{1}}\circ \ldots \circ \ncc{r}_{i_{s},j_{s}})\vee \mathbf{1}_{k}\right)\cap [1,\ldots,k].
\end{equation*}
For $\beta$ a partition of $[1,\ldots,k]$, set $\lmss{R}^{+}_{s}(b,w,\beta) = \{\gamma_{s} = \beta\} \cap \lmss{R}^{+}_{s}(b,w)$.
\begin{lemma}
\label{lemmacumulant}
Let $b$ a Brauer diagram and denote by $\pi$ the trace of the partition $\ncc{b}\vee 1$ on $[1,\ldots,k]$. Assume that $\pi$ is non crossing, then, for all tuple $\lmss{r} \in \lmss{R}^{+}_{s}(b)$, the partition $\gamma_{s}(\lmss{r})$ is non-crossing and $\gamma_{s}(\lmss{r}) \leq \pi$. In addition for all words $w \in \lmss{M}_{q}$, $\lmss{R}_{s}^{+}(b,w,\alpha) = \emptyset$ if there at least one couple of integers $1 \leq i,j \leq k$ with $i\sim_{\pi}j$ and $w_{i} \neq w_{j}$.
\end{lemma}
In the following, we set $\lmss{t}^{\lmss{r}} = t_{i_{1}} \cdots t_{i_{s}}$. By using formulae proved in the previous section for the generators $L_{i,\lmss{r}}$, we write as a sum over $\lmss{R}^{+}_{s}(b,w)$ the exponential $\exp\left( \sum_{i=1}^{q}t_{i}L_{i,\lmss{r}}\right)$ to obtain the following expression:
\begin{equation*}
\begin{split}
\exp&\left(\sum_{i=1}^{q}t_{i}L_{i,\lmss{r}}\right)(b,w)=e^{\frac{-1}{2}\sum_{i=1}^{q}t_{i}\lmss{n}_{i}(w)}\\ &\times \sum_{s=0}^{\infty} \sum_{\,\lmss{r}\, \in \lmss{R}_{s}^{+}(b,w)}\frac{\lmss{t}^{\lmss{r}}}{s!}\prod_{k=1}^{s}\lmss{r}(r_{i_{k},j_{k}}\diamond \ldots \diamond r_{i_{s},j_{s}},r_{i_{k+1},j_{k+1}}\diamond\ldots\diamond r_{i_{s},j_{s}}\diamond b)\cdot(r_{i_{1},j_{1}}\diamond \ldots \diamond r_{i_{s},j_{s}}\diamond b,w)
\end{split}
\end{equation*}
Next, we split the sum over $\lmss{R}^{+}_{s}(b)$ in the last equation into sums over the level sets of $\gamma_{s}$, each of these sum defines an operator $L_{\lmss{r}}(\beta,\lmss{t})$, $\beta \in \lmss{NC}_{k}$
\begin{equation*}
L_{\lmss{r}}(\beta,\lmss{t})(b,w) = \sum_{s=0}^{\infty}\frac{\lmss{t}^{\lmss{r}}}{s!}\sum_{\,\lmss{r}\, \in \lmss{R}_{s}^{+}(b,w,\alpha)}\prod_{k=1}^{s}\lmss{r}(r_{i_{k},j_{k}}\diamond\ldots\diamond r_{i_{s},j_{s}},r_{i_{k+1},j_{k+1}}\diamond\ldots\diamond r_{i_{s},j_{s}}\diamond b)(r_{i_{1},j_{1}}\diamond \ldots \diamond r_{i_{s},j_{s}}\diamond b,w).
\end{equation*}
Finally, we obtain the following expression for $\mathbb{m}(b,w,\lmss{t})$:
\begin{equation*}
\mathbb{m}(b,w,\lmss{t}) = \delta_{\Delta_{k}}\left(\exp\left(\sum_{i=1}^{q}t_{i}L_{i,\lmss{r}}\right)(b,w)\right) = e^{-\frac{1}{2}\sum_{i=1}^{q}t_{i}\lmss{n}_{i}(w)}\sum_{\beta \leq \pi} \delta_{\Delta_{k}}(L_{\lmss{r}}(\beta,\lmss{t})(b,w)).
\end{equation*}
It remains to show that for each non-crossing partition $\beta$ less than $\pi$, $\delta_{\Delta_{k}}\left(\mathrm{L}_{\lmss{r}}(\beta,\lmss{t})(b,w)\right)$ does not depends on $\pi$. As a matter of fact, the set $\lmss{R}_{s}^{+}(b,w,\beta)$ does not depends on $\pi$ and is determined by $\varepsilon$ and $\beta$. Furthermore, owing to its definition, $\delta_{\Delta_{k}}(b)$ is a function of the colourization $c$ (which lies on $\alpha$).

Set $c_{\beta}(\alpha, w, \varepsilon) = \delta_{\Delta_{k}}\left(L_{i,\lmss{r}}(\beta,\lmss{t})((\mathbf{0}_{k},c),w) \right)$. Lemma $\ref{lemmacumulant}$ implies $c_{\mathbf{1}}(\alpha,w,\varepsilon) = 0$ if the word $w$ contains to different letters.
The method that was used to prove theorem $\ref{thm1rect}$ can be applied verbatim to prove the following theorem, thus the proof is left to the reader.
\begin{theorem}
\label{thm2cluster}
Let $n\geq 1$ an integer and for each integer $N$ greater than one, let $d_{N}$ be a partition of $N$ into $n$ parts.
\par We assume that for all integer $i \leq n$, the ratio $r_{N}(i)=\frac{d_{N}(i)}{N}$ converges as $N$ tends to infinity to a positive value $r_{i}$ less than one. We assume further that the kernel of the partition $d_{1}$ is finer than the kernels $\lmss{Ker}(d_{N})$, $N\geq1$. Let $\mathbb{K}$ be one the three divisions algebras $\mathbb{R},\mathbb{C}$ or $\mathbb{H}$.
\par As the dimension $N$ tends to infinity, the non-commutative distribution of $U^{\mathbb{K}}_{\{d_{N}\},\lmss{ker}(d_{1})}$ converges to a $\mathcal{D}_{d_{1}}$-amalgamated free semi-group.
\end{theorem}
\bibliographystyle{plain}
\bibliography{biblio}
\end{document}

%% file: environnement.tex
\newtheorem{theorem}{Theorem}
\newtheorem*{remarque}{Remark}

\theoremstyle{plain}
\newtheorem{proposition}[theorem]{Proposition}

\newtheorem*{theorem*}{Theorem}

\newtheorem{lemma}[theorem]{Lemma}
\theoremstyle{definition}
\newtheorem{definition}[theorem]{Definition}

%% file: notation.tex
\newcommand{\freeunitary}{{U^{\langle n \rangle}}}
\newcommand{\mfreeunitary}{{\lmss{U}^{\langle n \rangle}}}

\newcommand{\unitaryfd}{U^{\mathbb{K}}_{n,d}}
\newcommand{\unitaryfdr}{U^{\mathbb{R}}_{n,d}}
\newcommand{\unitaryfdc}{U^{\mathbb{C}}_{n,d}}
\newcommand{\unitaryfdh}{U^{\mathbb{H}}_{n,d}}
\newcommand{\munitaryfd}{\lmss{U}_{N}^{\mathbb{K}}}
\newcommand{\munitaryfdr}{\lmss{U}_{N}^{\mathbb{R}}}
\newcommand{\munitaryfdc}{\lmss{U}_{N}^{\mathbb{C}}}
\newcommand{\munitaryfdh}{\lmss{U}_{N}^{\mathbb{H}}}

\newcommand{\noisefree}{\lmss{W}}

\newcommand{\hermitianalg}{\mathcal{H}\langle n \rangle}

\newcommand{\udualgroup}{\mathcal{O}\langle n \rangle}

\newcommand{\tr}{\lmss{T}_{k}}
\newcommand{\pr}{\lmss{W}_{k}}
\newcommand{\trd}{\lmss{T}_{k,\lmss{d}}}
\newcommand{\prd}{\lmss{W}_{k,\lmss{d}}}

\newcommand{\trnm}{\lmss{T}_{k,n}}
\newcommand{\prnm}{\lmss{W}_{k,n}}

\newcommand{\trexc}{\lmss{T}^{\neq}_{k}}
\newcommand{\prexc}{\lmss{W}^{\neq}_{k}}
\newcommand{\trexcd}{\lmss{T}^{\neq}_{k,\lmss{d}}}
\newcommand{\prexcd}{\lmss{W}^{\neq}_{k,\lmss{d}}}

\newcommand{\trp}[1]{\lmss{T}^{+}_{k}(#1)}
\newcommand{\prp}[1]{\lmss{W}^{+}_{k}(#1)}
\newcommand{\trpnc}[1]{\lmss{T}^{+,\scalebox{0.6}{$\bullet$}}_{k}(#1)}
\newcommand{\prpnc}[1]{\lmss{W}^{+,\scalebox{0.6}{$\bullet$}}_{k}(#1)}

\newcommand{\trpnm}[1]{\lmss{T}^{+}_{k,n}(#1)}
\newcommand{\prpnm}[1]{\lmss{W}^{+}_{k,n}(#1)}

\newcommand{\brauer}{{\mathcal{B}_{k}}}
\newcommand{\brauerdf}[1]{\mathcal{B}_{k}^{#1}}

\newcommand{\Brauer}{{\mathcal{B}_{k}}}
\newcommand{\cbrauer}{\overset{\circ}{\mathcal{B}}_{k}}
\newcommand{\cbrauerd}{\overset{\circ}{\brauerd}}
\newcommand{\cbrauerf}[1]{\overset{\circ}{\mathcal{B}^{#1}_{k}}}

\newcommand{\brauerz}{\mathcal{B}^{{ \scalebox{0.6}{$\bullet$}}}_{k}}
\newcommand{\ncc}[1]{{#1}^{\scalebox{0.6}{$\bullet$}}}
\newcommand{\nccc}[1]{#1^{\scalebox{0.6}{$\bullet$}}}

\newcommand{\brauergen}{(\ncc{b},c_{b})}
\newcommand{\Twist}{\lmss{Tw}^{\scalebox{0.6}{$\bullet$}}}
\newcommand{\genbrauerz}{(\ncc{b},c_{b})}
\newcommand{\brauerd}{\mathcal{B}_{k}^{\lmss{d}}}

\newcommand{\trdiag}{{\lmss{T}_{k}^{=}}}
\newcommand{\prdiag}{{\lmss{W}_{k}^{=}}}
\newcommand{\trdiagd}{{\lmss{T}_{k,\lmss{d}}^{=}}}
\newcommand{\prdiagd}{{\lmss{W}_{k,\lmss{d}}^{=}}}

\newcommand{\trexcp}[1]{
\ensuremath{\lmss{T}_{k}^{{\neq},+}(#1)}}
\newcommand{\prexcp}[1]{
\ensuremath{\lmss{W}_{k}^{{\neq},+}(#1)}}

\newcommand{\trdiagp}[1]{
\ensuremath{\lmss{T}_{k}^{{=},+}(#1)}}

\newcommand{\cluster}{{\mathcal{C}\mathcal{R}\mathcal{O}\langle\lmss{d}\rangle}}
\newcommand{\runitaryalg}{\mathcal{R}\mathcal{O}\langle n \rangle}

\newcommand{\runitaryfd}{U^{\mathbb{K}}_{\lmss{d}}}
\newcommand{\rdunitaryfdN}{U^{\mathbb{K}}_{\{d_{N}\}}}
\newcommand{\runitaryfdN}{U^{\mathbb{K}}_{d_{N}}}
\newcommand{\runitaryfdNr}{U^{\mathbb{R}}_{d_{N}}}
\newcommand{\runitaryfdNc}{U^{\mathbb{C}}_{d_{N}}}
\newcommand{\runitaryfdNh}{U^{\mathbb{H}}_{d_{N}}}

\newcommand{\twist}{\lmss{Tw}^{\scalebox{0.6}{$\bullet$}}}

\def\Alg{{\rm Alg}}
\def\biMod{{\rm biMod}}

\def\freem{\mathrel{\dot\sqcup}}
\def\free{\mathrel{\sqcup}}

\def\id{{\rm id}}

\def\path{{\sf P}}

\def\bcyan{{\color{cyan} \scalebox{0.8}{$\bullet$}}}
\def\bmage{{\color{magenta} \scalebox{0.8}{$\bullet$}}}
\def\boran{{\color{orange} \scalebox{0.8}{$\bullet$}}}